\documentclass[11pt,a4paper]{article}
\usepackage{indentfirst,latexsym,bm}
\usepackage{amscd}
\usepackage{amsthm}
\usepackage{amsmath}
\usepackage{amsfonts}
\usepackage{amssymb}
\usepackage[all]{xy}
\usepackage[dvips,draft]{graphics}
\usepackage{hyperref}
\usepackage{xcolor}
\usepackage{url}
\allowdisplaybreaks

\begin{document}
\title{\bf A Remark on Gromov-Witten Invariants of Quintic Threefold}
\author{\sl Longting Wu\footnote{longtingw@pku.edu.cn}\\[3pt]
            Beijing International Center for Mathematical Research,\\
            Peking University}
\date{}
\maketitle
\begin{abstract}
The purpose of the article is to give a proof of a conjecture of Maulik and Pandharipande for genus 2 and 3. As a result, it gives a way to determine Gromov-Witten invariants of the quintic threefold for genus 2 and 3.
\end{abstract}

\setcounter{tocdepth}{2}
\renewcommand*{\contentsname}{\centerline{Contents}}
\tableofcontents

\numberwithin{equation}{section}

\newtheorem{theorem}{Theorem}[section]
\newtheorem{conjecture}[theorem]{Conjecture}
\newtheorem{lemma}[theorem]{Lemma}
\newtheorem{corollary}[theorem]{Corollary}
\newtheorem{example}{Example}
\newtheorem{definition}[theorem]{Definition}
\newtheorem{remark}[theorem]{Remark}

\section{Introduction}
Let $Q$ be the quintic threefold in $\mathbb{P}^4$. $\mathbb{P}(N_{Q/\mathbb{P}^4}\oplus \mathcal{O}_Q) $ is the projective bundle associated to the vector bundle $N_{Q/\mathbb{P}^4}\oplus \mathcal{O}_Q $ over $Q$. $D_0$ is a divisor of $\mathbb{P}(N_{Q/\mathbb{P}^4}\oplus \mathcal{O}_Q)$ determined by the factor $N_{Q/\mathbb{P}^4}$.

Gathmann \cite{Ga1} used relative virtual localization technique to reduce some relative Gromov-Witten invariants of the pair $(\mathbb{P}(N_{Q/\mathbb{P}^4}\oplus \mathcal{O}_Q),D_0)$ to the absolute Gromov-Witten invariants of $Q$ when genus $g\leq 1$. Combining it with degeneration formula (\ref{degfm}), which relates Gromov-Witten invariants of $\mathbb{P}^4$ to relative invariants of the pairs $(\mathbb{P}^4,Q)$ and $(\mathbb{P}(N_{Q/\mathbb{P}^4}\oplus \mathcal{O}_Q),D_0)$, he could recursively determine Gromov-Witten invariants of the quintic threefold $N_{g,d}$ (\ref{def-Q}) for genus $g\leq 1$. For a discussion of the history of computing Gromov-Witten invariants of quintic threefold, we recommend the reader to see \cite{Li3}\cite{Liu}.

Later, Maulik and Pandharipande have found an algorithm (see \cite{MP}, Theorem 1) to determine relative invariants of the pair $(\mathbb{P}(N_{Q/\mathbb{P}^4}\oplus \mathcal{O}_Q),D_0)$ from the absolute invariants of $Q$ without the constraint of genus. Inspired by Gathmann's proposal, they proposed the following conjecture:

\begin{conjecture}[\cite{MP}]\label{conj}
The system of equations obtained from the degeneration formula (\ref{degfm}) (set $(V,W)=(\mathbb{P}^4,Q)$ in the formula) and the Maulik-Pandharipande's algorithm (see Section \ref{algm} or \cite{MP}, Theorem 1) can be used to determine both the relative theory of the pair $(\mathbb{P}^4,Q)$ and the Gromov-Witten invariants $N_{g,d}$ of $Q$.
\end{conjecture}

\begin{remark}
Conjecture \ref{conj} for $g=0,1$ directly follows from the idea of Gathmann. Maulik and Pandharipande have claimed in their paper that they have proven Conjecture \ref{conj} for genus 2, but they did not give a proof.
\end{remark}

In this paper, we prove that
\begin{theorem}\label{thm-2}
The Conjecture \ref{conj} is true for $g=2,3$.
\end{theorem}

As a consequence of Theorem \ref{thm-2}, it gives an algorithm to determine $N_{g,d}$ for $g=2,3$. Here, we do not claim any priority to the proof of Conjecture \ref{conj} for genus 2. We may owe it to Maulik and Pandharipande.

\begin{remark}
In the same paper \cite{MP}, Maulik and Pandharipande also gave a calculation scheme to determine all $N_{g,d}$, which is different from the method of Conjecture \ref{conj}. But they also remarked that the method provided by Conjecture \ref{conj} appears more suitable for calculation.
\end{remark}

Since Gromov-Witten invariants of $\mathbb{P}^4$ are known by several methods \cite{Ga2}\cite{Gi}\cite{GP}, we treat Gromov-Witten invariants of $\mathbb{P}^4$ as known constants in the proof of Theorem \ref{thm-2}.

Now our strategy to prove Theorem \ref{thm-2} can be summarized as follows.

We determine $N_{g,d}$ at first. It is well known that $N_{g,0}$ can be derived from Hodge integral in the moduli space of curves \cite{FP1}
\begin{equation}\label{inv-d0}
N_{g,0}=(-1)^g\frac{\chi(Q)}{2}\frac{|B_{2g}||B_{2g-2}|}{2g(2g-2)(2g-2)!},
\end{equation}
where $\chi(Q)$ is the Euler characteristic of $Q$ and $B_{2g}$ are Bernoulli numbers.

In order to explain our approach to determining $N_{g,d>0}$, we need to introduce the following notations.

\begin{definition}\label{maindf-1}
Let $S_{g,d}$ be a set, which consists of sets of integer pairs $\{(a_i,b_i)\}_{i=1}^r$ such that $(a_i,b_i)\in\mathbb{Z}_{\geq 0}\times\{0,1,2,3\}$, $(a_i,b_i)\neq (0,0)$ and $\sum(a_i+b_i)=5d+1-g$.
\end{definition}

Here, the number of integer pairs is not fixed.

\begin{definition}\label{maindf-2}
Let $S_{g,d}'$ be the subset of $S_{g,d}$, which we further require $\{(a_i,b_i)\}_{i=1}^r$ to satisfy
\[\sum(a_i+1)\leq 5d.\]
\end{definition}

In order to compute $N_{g,d>0}$, we firstly show by degeneration formula (\ref{degfm}) that

\begin{theorem}\label{mainthm-1}
For fixed $d>0$, $g\geq 0$ and $\zeta\in S_{g,d}$, there exist uniquely determined constants $C_{\rho}$ such that
\begin{equation}\label{mainrl-1}
A_{g,d}(\zeta)-\sum_{\rho\in S_{g,d}'}A_{g,d}(\rho)C_{\rho}=B_{g,d}(\zeta)-\sum_{\rho\in S_{g,d}'}B_{g,d}(\rho)C_{\rho}+NPT.
\end{equation}
\end{theorem}
Here, those $A_{g,d}$ are Gromov-Witten invariants of $\mathbb{P}^4$, those $B_{g,d}$ are relative Gromov-Witten invariants of the pair $(Y,D_0)$, and $NPT$ stands for non-principal terms in the degeneration formula, which can be determined by those $N_{g',d'}$ such that
\[g'<g,~d'\leq d~~~\text{or}~~~g'\leq g,~d'<d.\]

We refer to Theorem \ref{keythm} and its proof for more details of the notations
in Theorem \ref{mainthm-1}.

Next, we specialize to $g=2,3$. Combining relative virtual localization formula (\ref{lc-0}) and virtual push-forward properties of Lemmas \ref{vpfp-1} and \ref{vpfp-2}, we can show that
\begin{theorem}\label{mainthm-2}
For each pair $(g,d)$ with $g=2,3$ and $d>0$, there always exists $\zeta_{g,d}\in S_{g,d}$ such that term
\begin{equation}\label{mainrl-2}
B_{g,d}(\zeta_{g,d})-\sum_{\rho\in S_{g,d}'}B_{g,d}(\rho)C_{\rho}
\end{equation}
on the R.H.S of (\ref{mainrl-1}) equals to $C_{g,d}N_{g,d}$ where $C_{g,d}\neq 0$.
\end{theorem}

For more details of Theorem \ref{keythm2},  we refer to Theorem \ref{mainthm-2} and its proof.

Since the proofs of Theorem \ref{mainthm-2} for $g=2,3$ are similar, we give a detailed proof for $g=3$ in Section \ref{partII} and a short proof for $g=2$ in Appendix A.

Combining (\ref{mainrl-1}) and (\ref{mainrl-2}), we can recursively determine $N_{g,d}$ for $g=2,3$ by
\begin{equation*}
N_{g,d}=\frac{1}{C_{g,d}}\big\{A_{g,d}(\zeta_{g,d})-\sum_{\rho\in S_{g,d}'}A_{g,d}(\rho)C_{\rho}-NPT\big\}.
\end{equation*}

Once $N_{g,d}$ are known, relative Gromov-Witten invariants of the pair $(\mathbb{P}^4, Q)$ can also be recursively determined by \cite{MP}, Theorem 2.

To show the effectiveness of our method, we give a computation of $N_{2,1}$ at Appendix D.

In order to prove Conjecture \ref{conj} for $g\geq 4$, we may try to generalize Theorem \ref{mainthm-2}. But we show in Section \ref{section-4} that this is impossible for all the pairs $(g,d)$. The key reason is that we need to choose $\zeta_{g,d}\in S_{g,d}\backslash S_{g,d}'$. Otherwise
\[B_{g,d}(\zeta_{g,d})-\sum_{\rho\in S_{g,d}'}B_{g,d}(\rho)C_{\rho}=0.\]
The constraint $\zeta_{g,d}\in S_{g,d}\backslash S_{g,d}'$ then requires that
\[d\geq \frac{2g-1}{5}.\]

So we can not generalize Theorem \ref{mainthm-2} for all the pairs $(g,d)$. But we conjecture that
\begin{conjecture}\label{maincj}
For each pair $(g,d)$ with $d\geq \frac{2g-1}{5}$, there always exists $\zeta_{g,d}\in S_{g,d}$ such that
\[B_{g,d}(\zeta_{g,d})-\sum_{\rho\in S_{g,d}'}B_{g,d}(\rho)C_{\rho}\]
equals to $C_{g,d}N_{g,d}$ with $C_{g,d}\neq 0$.
\end{conjecture}

If this conjecture is true, then we can recursively determine all $N_{g,d}$ from those $N_{g,d}$ with $0<d<\frac{2g-1}{5}$.

The paper is organized as follows. In Section \ref{section-2}, we give some necessary preliminaries for further discussion. In Section \ref{section-3}, we give a proof of Theorem \ref{thm-2}. In Section \ref{section-4}, we give a further remark on Conjecture \ref{conj} for $g>3$.

Finally, we mention that there are also two very different approaches to computing $N_{g\geq 2,d}$ by Chang-Li-Li-Liu \cite{CLLL2}\cite{CLLL1} and Guo-Janda-Ruan \cite{GJR} recently.

\textbf{Acknowledgement.} First, the author would like to thank my advisor Xiaobo Liu for suggesting to me this interesting problem and numerous valuable suggestions when writing this paper. Secondly, he thanks, Melissa Liu, Rahul Pandharipande, Davesh Maulik, Huai-Liang Chang, Jun Li and Wei-Ping Li for helpful conversations. Research of the author was partially supported by SRFDP grant 20120001110051.

\section{Preliminaries}\label{section-2}
\subsection{Absolute Gromov-Witten Invariants}
Let $V$ be a smooth projective variety.

We use $\overline{\mathcal{M}}_{g,m}(V,\beta)$ to denote the moduli space of stable maps from genus $g$, $m$-pointed curve $\Sigma$ to $V$ with curve class $\beta$.

The cotangent line of the $i$th marked point forms a line bundle $L_i$ on $\overline{\mathcal{M}}_{g,m}(V,\beta)$. We denote the first Chern class of $L_i$ as $\psi_i$. The evaluation map of the $i$th marked point on $\overline{\mathcal{M}}_{g,m}(V,\beta)$ is defined by
\begin{eqnarray*}
ev_i:  & \overline{\mathcal{M}}_{g,m}(V,\beta) &\longrightarrow  V\\
       & [(\Sigma,x,f)]& \longmapsto  f(x_i),
\end{eqnarray*}
where $[(\Sigma,x,f)]$ is an isomorphic class of the stable map.

Let $\theta_1,\theta_2,...,\theta_{n_V}$ be a basis of $H^*(V,\mathbb{Q})$. The Gromov-Witten invariants of $V$ are defined by
\begin{equation}\label{abGW}
\Big\langle\prod_{i=1}^m\tau_{k_i}(\theta_{l_i})\Big\rangle^{V}_{g,\beta}:=\int_{[\overline{\mathcal{M}}_{g,m}(V,\beta)]^{vir}}\prod_{i=1}^m\psi_i^{k_i}ev_i^*(\theta_{l_i})
\end{equation}
where $[\overline{\mathcal{M}}_{g,m}(V,\beta)]^{vir}$ is the virtual fundamental class of $\overline{\mathcal{M}}_{g,m}(V,\beta)$  with virtual dimension
\begin{equation}\label{abvirdim}
c_1(T_V)\cdot\beta+(1-g)(\text{dim}V-3)+m.
\end{equation}

To distinguish them to relative Gromov-Witten invariants, we call the above invariants absolute Gromov-Witten invariants.

\subsection{Relative Gromov-Witten Invariants}
The relative Gromov-Witten invariants were defined by Ionel-Parker \cite{IP1}, Li-Ruan \cite{LR} in symplectic geometry, and Jun Li \cite{Li1} in algebraic geometry. We will give an overview of relative Gromov-Witten invariants and fix notations throughout the paper. Our presentation is based on \cite{GrVa}\cite{Ka}\cite{MP}.

\subsubsection{Definition}
Let $W$ be a smooth connected divisor on $V$. If $E$ is a vector bundle on $V$, we use $\mathbb{P}(E)$ to denote the projectization of $E$.

Let $N_{W/V}$ denote the normal bundle of $W$ in $V$. We set $Y=\mathbb{P}(N_{W/V}\oplus \mathcal{O}_W)$. $D_0$ and $D_{\infty}$ are two divisors in $Y$ corresponding to $\mathbb{P}(N_{W/V})$ and $\mathbb{P}(\mathcal{O}_W)$ respectively.

\begin{definition}
Let $Y_s$ be the union of $s$ copies of $Y$, so that the $i$th copy of $D_{\infty}$ is glued to $(i+1)$th copy of $D_0$.
\end{definition}

We define $D^{i}_0$ (resp. $D^{i}_{\infty}$) to be the $i$th copy of $D_0$ (resp. $D_{\infty}$). The singular locus of $Y_s$ is $Sing (Y_s)=\bigcup_{i=1}^{s-1}D^{i}_{\infty}$. For simplicity, we also use $D_0$ (resp. $D_{\infty}$) to denote the first (resp. last) copy of $D_0$ (resp. $D_{\infty}$).

The projection map from $Y$ to $W$ can be naturally extended to $Y_s$. We denote the extended map as $p_s$.

Let $Z_s=D_0\bigcup Sing(Y_s)\bigcup D_{\infty}$. We define $Aut_{Z_s}(Y_s)$ to be the group of isomorphisms of $Y_s$ fixing $Z_s$. Obviously, $Aut_{Z_s}(Y_s)\backsimeq (\mathbb{C}^*)^s$.

\begin{definition}
Let $\widetilde{V}_s$ be the union of $V$ and $Y_s$, so that the divisor $W$ of $V$ is glued to $D_0$ of $Y_s$.
\end{definition}

The singular locus of $\widetilde{V}_s$ is
$Sing(\widetilde{V}_s)=W\bigcup Sing~Y_s$. The contraction map $c_s:\widetilde{V}_s\rightarrow V$ is defined by $c_s|_V=Id$ and $c_s|_{Y_s}=p_s$.

We set $T_s=V\bigcup Z_s$. Let $Aut_{T_s}(\widetilde{V}_s)$ denote the group of isomorphisms of $\widetilde{V}_s$ fixing $T_s$. Obviously, $Aut_{T_s}(\widetilde{V}_s)=Aut_{Z_s}(Y_s)$.

Let $\Gamma$ be a tuple $(g,\beta,\{\nu_i\}_{i=1}^n,m)$, where $g,m$ are nonnegative integers, $\beta$ is a curve class, and $\{\nu_i\}_{i=1}^n$ is a set of positive integers such that $\sum_i \nu_i=W\cdot \beta$.

\begin{definition}\label{relmap}
A stable relative map of the pair $(V, W)$ with data $\Gamma$ is a tuple $(\Sigma,x_1,...,x_m,y_1,...,y_n,f,\widetilde{V}_s)$, where
\begin{itemize}
\item[(a1)] $\Sigma$ is a connected nodal curve with arithmetic genus $g$. $x_1,\ldots,y_n$ are distinct smooth marked points on $\Sigma$. We set $x=\{x_1,\ldots,x_m\}$, $y=\{y_1,\ldots,y_n\}$ and call $x$ the absolute marked points and $y$ the boundary (relative) marked points.
\item[(a2)]  $f$ is a morphism from $\Sigma$ to $\widetilde{V}_s$, such that $(c_s\circ f)_*([\Sigma])=\beta$ and $f^{-1}(D_{\infty})=\Sigma\nu_j y_j$.
\item[(a3)] Predeformability condition: the preimage of $Sing(\widetilde{V}_s)$ contains only the nodes of $\Sigma$; for each node $q$ mapped to $Sing(\widetilde{V}_s)$, the two branches of $\Sigma$ at $q$ are mapped to two different irreducible components of $\widetilde{V}_s$, and the orders of contact with $Sing(\widetilde{V}_s)$ on both sides are equal.
\item[(a4)] $|Aut(f)|<\infty$, where $Aut(f):=\{(\varphi,\psi)\in Aut(\Sigma,x,y)\times Aut_{T_s}(\widetilde{V}_s)|f\circ \varphi=\psi\circ f\}$.
\end{itemize}
\end{definition}
Two stable relative maps $(\Sigma,x,y,f,\widetilde{V}_s)$ and $(\Sigma',x',y',f',\widetilde{V}_{s'})$ are said to be isomorphic if there are isomorphisms $\varphi: (\Sigma,x,y)\rightarrow (\Sigma',x',y')$ and $\psi:\widetilde{V}_s\rightarrow \widetilde{V}_{s'}$, such that $\psi\circ f=f'\circ \varphi$ and $c_{s'}\circ \psi=c_s$.

We use $\overline{\mathcal{M}}_{\Gamma}(V,W)$ to denote the moduli space of stable relative maps of the pair $(V,W)$ with data $\Gamma$.

Let $L_i$ denote the cotangent line bundle associated to the absolute marked point $x_i$ and set $\psi_i=c_1(L_i)$. The evaluation map associated to $x_i$ is defined by
\begin{eqnarray*}
ev_i:  & \overline{\mathcal{M}}_{\Gamma}(V,W) &\longrightarrow  V\\
       & [(\Sigma,x,y,f,\widetilde{V}_s)]& \longmapsto (c_s\circ f)(x_i)
\end{eqnarray*}
where $[(\Sigma,x,y,f,\widetilde{V}_s)]$ is an isomorphic class. The evaluation map associated to the boundary marked point $y_j$ is defined by:
\begin{eqnarray*}
\widetilde{ev}_j:  & \overline{\mathcal{M}}_{\Gamma}(V,W) &\longrightarrow  W\\
       & [(\Sigma,x,y,f,\widetilde{V}_s)]& \longmapsto (c_s\circ f)(y_j).
\end{eqnarray*}
We note that the boundary marked point $y_j$ is mapped to $W$.

Let $\delta_1,\delta_2,...\delta_{n_W}$ be a basis of $H^*(W,\mathbb{Q})$. A cohomology weighted partition $\nu$ is an unordered set of weighted pairs
$$\big\{(\nu_1,\delta_{s_1}),(\nu_2,\delta_{s_2})~...~(\nu_n,\delta_{s_n})\big\}.$$
The group of permutation symmetries of $\nu$ is denoted as $Aut(\nu)$. We may use $l(\nu)$ to denote the number of weighted pairs in $\nu$. Here $l(\nu)=n$.

The standard order of weighted pairs is given by
\[(\nu_i,\delta_i)>(\nu'_{i'},\delta_{i'}),~\text{if}~\nu_i>\nu'_{i'}~\text{or}~\nu_i=\nu'_{i'},~ i>i'.\]
Without further explanation, we always assume that $\nu$ is written in decreasing order by standard order.

The relative Gromov-Witten invariant of the pair $(V,W)$ is defined by
\begin{equation}\label{relGW}
\begin{aligned}
\Big\langle\prod_{i=1}^m&\tau_{k_i}(\theta_{l_i})\Big | \nu\Big\rangle_{g,\beta}^{V,W}\\
&:=\frac{1}{|Aut(\nu)|}\int_{[\overline{\mathcal{M}}_{\Gamma}(V,W)]^{vir}}\prod_{i=1}^m\psi_i^{k_i}ev_i^*(\theta_{l_i})\prod_{j=1}^{n} \widetilde{ev}_j^*(\delta_{s_j}),
\end{aligned}
\end{equation}
where $\Gamma=(g,\beta,\{\nu_i\}_{i=1}^n,m)$ and $[\overline{\mathcal{M}}_{\Gamma}(V,W)]^{vir}$ is the virtual fundamental class of $\overline{\mathcal{M}}_{\Gamma}(V,W)$ with virtual dimension
\begin{equation}\label{relvirdim}
c_1(T_V)\cdot\beta+(dim V-3)(1-g)+m+n-W\cdot\beta.
\end{equation}

We also need the moduli space of relative stable maps with possibly disconnected domains. Firstly, we need to introduce a relative graph $G^r$ to keep track of the data of different components.

 The relative graph $G^r$ consists of
\begin{itemize}
\item[(i)] a set $S$ of vertices,
\item[(ii)] an assignment of genera $g:S\rightarrow \mathbb{Z}_{\geq 0}$, and an assignment of curve classes $\beta:S\rightarrow H_2(V,\mathbb{Z})$,
\item[(iii)] an assignment of absolute marked points $a:\{1,2,\ldots,m\}\rightarrow S$,
\item[(iv)] an assignment of boundary marked points $b:\{1,\ldots,n\}\rightarrow S$,
\item[(v)]an assignment of contact orders $\tilde{\nu}:\{1,\ldots,n\}\rightarrow \mathbb{Z}_{>0}$, such that for each vertex $s_i\in S$, we have $\beta(s_i)\cdot W=\sum_{j\in b^{-1}(s_i)}\tilde{\nu}(j)$.
\end{itemize}

We set $g(G^r)=\sum_{s_i\in S} g(s_i)-|S|+1$ to be the genus of $G^r$. We set $\beta(G^r)=\sum_{s_i\in S}\beta(s_i)$ to be the curve class of $G^r$. We set
\begin{equation}\label{relgrapy-inf}
\Gamma(G^r)=\big(g(G^r),\beta(G^r),m,n,\tilde{\nu}\big)
\end{equation}
to be the total data of relative graph $G^r$.

Two relative graphs $G^r_1$ and $G^r_2$ are said to be isomorphic if they have the same numbers of absolute and relative marked points, the same assignment of contact orders and there exists a bijection of vertices $\phi:S_1\rightarrow S_2$ which commutes with the assignments (ii)-(iv).

\begin{definition}
A stable relative map of the pair $(V,W)$ with relative graph $G^r$ is a tuple $(\Sigma,x_1,...,x_m,y_1,...,y_n,f,\widetilde{V}_s)$ such that
\begin{itemize}
\item[(a1')] $\Sigma$ is a disjoint union of connected nodal curves $\bigsqcup_{s_i\in S}\Sigma_{s_i}$, such that the genus of
$\Sigma_{s_i}$ is given by $g(s_i)$, the absolute marked point $x_i$ lies on $\Sigma_{a(i)}$ and the boundary marked point $y_j$ lies on $\Sigma_{b(j)}$.
\item[(a2')] $f$ is a morphism from $\Sigma$ to $\widetilde{V}_s$. If we denote $f_i$ to be the restriction of $f$ to $\Sigma_{s_i}$, then $f_i$ is required to satisfy $(c_s\circ f_i)_*([\Sigma_{s_i}])=\beta(s_i)$, $f_i^{-1}(D_{\infty})=\sum_{\{j\in b^{-1}(s_i)\}}\tilde{\nu}(j) y_j$.
\item[(a3')] $f$ is also required to satisfy conditions (a3)-(a4) in the Definition \ref{relmap}.
\end{itemize}
\end{definition}

We use $\overline{\mathcal{M}}_{G^r}^{\bullet}(V,W)$ to denote the corresponding moduli space associated to the relative graph $G^r$. The possibly disconnected relative invariant can be defined as (\ref{relGW}), which factors as a product of the individual connected invariants. It is denoted by $\langle\cdots\rangle^{\bullet, V,W}_{G^r}$.

\subsubsection{Type I or II invariants}
Let $[D_0],[D_{\infty}]\in H^*(Y,\mathbb{Q})$ be cohomology classes
associated to divisors $D_0,D_{\infty}$. We may view $\delta_1,\delta_2,...\delta_{n_W}$ as elements of $H^*(Y,\mathbb{Q})$ via pullback by the projection map $p:Y\rightarrow W$. Define classes in $H^*(Y,\mathbb{Q})$ by
\begin{align*}
\gamma_i&=\delta_i,\\
\gamma_{n_W+i}&=\delta_i \cdot [D_0],\\
\gamma_{2{n_W}+i}&=\delta_i \cdot [D_{\infty}],
\end{align*}
where $i=1,2,...n_W$.

We denote the relative invariants of the pair $(Y,D_{\infty})$ by ignoring the superscripts:
\begin{equation*}
\Big\langle\prod_{i=1}^m\tau_{k_i}(\gamma_{l_i})\Big |\nu\Big \rangle_{g,\beta}=\Big\langle\prod_{i=1}^m\tau_{k_i}(\gamma_{l_i})\Big | \nu\Big\rangle_{g,\beta}^{Y,D_{\infty}}.
\end{equation*}

For the pair $(Y,D_0)$, we write the relative condition $\mu=\{(\mu_j,\delta_{t_j})\}_{j=1}^{l(\mu)}$ on the left side and define
\begin{equation}\label{Def-TypeI}
\begin{aligned}
&\Big\langle\mu\Big |\prod_{i=1}^m\tau_{k_i}(\gamma_{l_i}) \Big\rangle_{g,\beta}\\
&~:=\frac{1}{|Aut(\mu)|}\int_{[\overline{\mathcal{M}}_{\Gamma}(Y,D_0)]^{vir}}\prod_{j=1}^{l(\mu)}\widetilde{ev}_{l(\mu)+1-j}^*(\delta_{t_{l(\mu)+1-j}})\prod_{i=1}^m\psi_i^{k_i}ev_i^*(\gamma_{l_i}).
\end{aligned}
\end{equation}
Here, we write the product $\prod_{j=1}^{l(\mu)}\widetilde{ev}_{l(\mu)+1-j}^*(\delta_{t_{l(\mu)+1-j}})$ in reverse order so as to give the right sign in the degeneration formula.

Relative invariants of the pairs $(Y,D_0)$ and $(Y,D_{\infty})$ are termed $type\ I$.

The definition of stable relative maps of the pair $(Y,D_0\cup D_{\infty})$ is slightly different from that of \ref{relmap}. We will describe it in the below.

Let $Y_{s,l}=Y_s\cup Y\cup Y_l$, where $D_{\infty}$ of $Y_s$ is glued to $D_0$ of $Y$ and $D_{0}$ of $Y_l$ is glued to $D_{\infty}$ of $Y$. The divisors $D_0$ of $Y_s$ and $D_{\infty}$ of $Y_l$ can naturally be seen as divisors of $Y_{s,l}$. With some abuse of notation, we also denote them as $D_0$ and $D_{\infty}$ in $Y_{s,l}$. The natural contraction from $Y_{s,l}$ to the central $Y$ is denoted as $c_{s,l}$.

We use $Sing(Y_{s,l})$ to denote the singular locus of $Y_{s,l}$. Let $Z_{s,l}=D_0\cup Sing(Y_{s,l})\cup D_{\infty}$ and $T_{s,l}$ be the union of $Z_{s,l}$ and the central $Y$ in $Y_{s,l}$.

Let $Aut_{T_{s,l}}(Y_{s,l})$ denote the group of isomorphisms of $Y_{s,l}$ fixing $T_{s,l}$. Obviously, $Aut_{T_{s,l}}(Y_{s,l})\simeq(\mathbb{C}^*)^{s+l}$.

Let $\Gamma=\{g,\beta,\{\mu_j\}_{j=1}^n,\{\nu_k\}_{k=1}^{n'},m\}$, where $g,m$ are nonnegative integers, $\beta$ is a curve class, $\{\mu_j\}_{j=1}^n$ is a set of positive integers such that $\sum_j \mu_j=D_0\cdot \beta$ and $\{\nu_k\}_{k=1}^{n'}$ is a set of positive integers
satisfies $\sum_k \nu_k=D_{\infty}\cdot \beta$.
\begin{definition}\label{type II}
A stable relative map of the pair $(Y,D_0\cup D_{\infty})$ with data $\Gamma$ is a tuple $(\Sigma,x_1,\ldots,x_m,y_1,\ldots,y_n,z_1,\ldots,z_{n'},f,Y_{s,l})$  such that
\begin{itemize}
\item[(b1)]$\Sigma$ is a genus $g$ connected nodal curve. $x_1,\ldots,z_{n'}$ are distinct smooth marked points on $\Sigma$. We set $x=\{x_1,\ldots,x_m\}$, $y=\{y_1,\ldots,y_n\}$ and $z=\{z_1,\ldots,z_{n'}\}$.
\item[(b2)] $f$ is a morphism from $\Sigma$ to $Y_{s,l}$, such that $(c_{s,l}\circ f)_*([\Sigma])=\beta$, $f^{-1}(D_0)\\=\sum_j \mu_jy_j$ and  $f^{-1}(D_{\infty})=\sum_k \nu_k z_k$.
\item[(b3)] Predeformability condition as in Definition \ref{relmap}.
\item[(b4)] $|Aut(f)|<\infty$.

\end{itemize}
Here, $Aut(f):=\{(\varphi,\psi)\in Aut(\Sigma,x,y,z)\times Aut_{T_{s,l}}(Y_{s,l})|f \circ \varphi = \psi \circ f\}$.
\end{definition}
The isomorphism between two stable relative maps can be defined similarly as before. We denote the isomorphic class as $[(\Sigma,x,y,z,f,Y_{s,l})]$.

The moduli space of stable relative maps of the pair $(Y,D_0\cup D_{\infty})$ with data $\Gamma$ is denoted as $\overline{\mathcal{M}}_{\Gamma}(Y,D_0\cup D_{\infty})$.

As before, we use $\psi_i$ to denote the first Chern class of the cotangent line bundle associated to $x_i$. The evaluation map associated to $x_i$ is defined to be
\begin{eqnarray*}
ev_i:  & \overline{\mathcal{M}}_{\Gamma}(Y,D_0\cup D_{\infty}) &\longrightarrow  Y\\
       & [(\Sigma,x,y,z,f,Y_{s,l})]& \longmapsto (c_{s,l}\circ f)(x_i).
\end{eqnarray*}
The evaluation map associated to $y_j$ is defined to be
\begin{eqnarray*}
\widetilde{ev}_j:  & \overline{\mathcal{M}}_{\Gamma}(Y,D_0\cup D_{\infty}) &\longrightarrow  W\\
       & [(\Sigma,x,y,z,f,Y_{s,l})]& \longmapsto (p\circ c_{s,l}\circ f)(y_j).
\end{eqnarray*}
The evaluation map $\widehat{ev}_k$ associated to $z_k$ can be defined as $\widetilde{ev}_j$.

The relative invariants of the pair $(Y,D_0\bigcup D_{\infty})$ can be defined by
\begin{equation*}
\begin{aligned}
&\Big\langle\mu\Big |\prod_{i=1}^m\tau_{k_i}(\gamma_{l_i})\Big | \nu\Big \rangle_{g,\beta}\\
&~~:=A_{\mu,\nu}\int_{[\overline{\mathcal{M}}_{\Gamma}(Y,D_0\cup D_{\infty})]^{vir}}\prod_{j=1}^{n} \widetilde{ev}_{n+1-j}^*(\delta_{t_{n+1-j}})\prod_{i=1}^m\psi_i^{k_i}ev_i^*(\gamma_{l_i})\prod_{k=1}^{n'} \widehat{ev}_k^*(\delta_{s_k}),
\end{aligned}
\end{equation*}
where $A_{\mu,\nu}=\frac{1}{|Aut(\mu)||Aut(\nu)|}$. The relative invariants of the pair $(Y,D_0\cup D_{\infty})$ are termed $type\ II$.

In order to define disconnected type II invariants, we need to introduce the relative graph for the pair $(Y,D_0\cup D_{\infty})$. It consists of
\begin{itemize}
\item[(i)] a set $S$ of vertices,
\item[(ii)] an assignment of genera $g:S\rightarrow \mathbb{Z}_{\geq 0}$, and an assignment of curve classes $\beta:S\rightarrow H_2(Y,\mathbb{Z})$,
\item[(iii)] an assignment of absolute marked points $a:\{1,2,\ldots,m\}\rightarrow S$,
\item[(iv)] assignments of boundary marked points
         $$b_0:\{1,\ldots,n\}\rightarrow S,~~~~~b_{\infty}:\{1,\ldots,n'\}\rightarrow S,$$
\item[(v)] assignments of contact orders
$$\tilde{\mu}:\{1,\ldots,n\}\rightarrow \mathbb{Z}_{>0},~~~~~\tilde{\nu}:\{1,\ldots,n'\}\rightarrow \mathbb{Z}_{>0},$$
such that for a given vertex $s$, we have
$$\beta(s)\cdot [D_0]=\sum_{i\in b_0^{-1}(s)}\tilde{\mu}(i),~~~~~\beta(s)\cdot [D_{\infty}]=\sum_{j\in b_{\infty}^{-1}(s)}\tilde{\nu}(j).$$
\end{itemize}
For a given relative graph $G^r$ of the pair $(Y,D_0\cup D_{\infty})$, it is very natural to define the possibly disconnected moduli space $\overline{\mathcal{M}}_{G^r}^{\bullet}(Y,D_0\cup D_{\infty})$. The corresponding possibly disconnected invariants can be computed by the product of the individual connected invariants as before.

The possibly disconnected type I and type II invariants are indicated by adding a superscript $\bullet$ to the bracket.

\subsubsection{Rubber invariants}
Rubber invariants will naturally appear in the relative virtual localization formula.

Firstly, I will describe the moduli space of stable relative maps to a non-rigid target.

 A stable relative map to a non-rigid target is a tuple $(\Sigma,x,y,z,f,Y_{s,l})$ which satisfies (b1)-(b4) of Definition \ref{type II} except that the automorphism group $Aut(f)$ is changed to be
\[Aut(f)=\{(\varphi,\psi)\in Aut(\Sigma,x,y,z)\times Aut_{Z_{s,l}}(Y_{s,l})|f \circ \varphi = \psi \circ f\}.\]
Here, the automorphism group $Aut_{Z_{s,l}}(Y_{s,l})$ does not fix the central $Y$ in $Y_{s,l}$. Since the central $Y$ is not fixed, we could replace $Y_{s,l}$ by $Y_{s+l+1}$ ignoring the choices of different $s,l$ to a fixed sum $s+l$.

The moduli space of stable relative maps to a non-rigid target is denoted as $\overline{\mathcal{M}}_{\Gamma}(Y,D_0\cup D_{\infty})^{\sim}$. We also call it a rubber space.

The evaluation maps for $x_i$, $y_j$ and $z_k$ can be defined by evaluating the corresponding marked points under the map $p_{s+l+1}\circ f$. Here, all the evaluation maps are mapped into $W$.

Next, we introduce cotangent line bundles on $\overline{\mathcal{M}}_{\Gamma}(Y,D_0\cup D_{\infty})^{\sim}$ coming from the targets. Our main reference is \cite{GrVa} Section 2.5.

Let $\mathfrak{M}_{0,2}$ denote the Artin stack of genus 0, 2-pointed pre-stable curves. $\mathcal{U}$ is an open substack of $\mathfrak{M}_{0,2}$ which parametrizing curves such that the two marked points are separated by every node.

Fixing a point $w\in W$.  There is a well-defined map
\begin{eqnarray*}
T:  & \overline{\mathcal{M}}_{\Gamma}(Y,D_0\cup D_{\infty})^{\sim} &\longrightarrow  \mathcal{U}\\
       & [(\Sigma,x,y,z,f,Y_{s})]& \longmapsto \bigr[\big(R_{s},D_0\cap R_{s}, D_{\infty}\cap R_{s}\big)\bigr],
\end{eqnarray*}
where $R_{s}=(p_s)^{-1}(w)$ is a genus 0 pre-stable curve with two marked points $D_0\cap R_{s}$ and $D_{\infty}\cap R_{s}$.

The first Chern class of the cotangent line bundle associated to the marked point $D_0\cap R_{s}$ is denoted as $\psi_{0}$ and that of $D_{\infty}\cap R_{s}$ is denoted as $\psi_{\infty}$. The pullback of $\psi_0$ (resp. $\psi_{\infty}$) to $\overline{\mathcal{M}}_{\Gamma}(Y,D_0\cup D_{\infty})^{\sim}$ by $T$ will still be denoted as $\psi_0$ (resp. $\psi_{\infty}$). We note that $\psi_0$ and $\psi_{\infty}$ do not depend on the choice of $w$.

In the following, we only consider rubber invariants of the form
\begin{equation*}
\begin{aligned}
&\Big\langle\mu\Big |\prod_{i=1}^m\tau_{k_i}(\delta_{l_i})\Big |k_{\infty}, \nu\Big \rangle_{g,\beta}^{\sim}=A_{\mu,\nu}\times\\
&\int_{[\overline{\mathcal{M}}_{\Gamma}(Y,D_0\cup D_{\infty})^{\sim}]^{vir}}\prod_{j=1}^{n} \widetilde{ev}_{n+1-j}^*(\delta_{t_{n+1-j}})\prod_{i=1}^m\psi_i^{k_i}ev_i^*(\delta_{l_i})\prod_{k=1}^{n'} \widehat{ev}_k^*(\delta_{s_k})\psi_{\infty}^{k_{\infty}}.
\end{aligned}
\end{equation*}

The graph for rubber space can be defined in the same way as for the stable relative maps of the pair $(Y,D_0\cup D_{\infty})$. We may similarly define the possibly disconnected rubber space and the corresponding possibly disconnected invariant.
As before, we add a superscript $\bullet$ to indicate it.

\begin{remark}
Rubber invariants can be determined from type II invariants by rubber calculus in \cite{MP}.
\end{remark}

\subsection{Degeneration formula}\label{ep-defm}
We will introduce degeneration formula \cite{EGH}\cite{IP2}\cite{LR}\cite{Li2} in the case of deformation to the normal cone. For simplicity, we may suppose that the cohomology classes of $V$ only contain even classes. The presentation here follows from \cite{HHKQ}\cite{Li2}.

Let $\mathcal{B}$ be the blowing up of $V\times \mathbb{C}$ along $W\times 0$. We set $\pi:\mathcal{B}\rightarrow \mathbb{C}$ to be the projection to the second factor. For $t\neq 0$, the fiber $\pi^{-1}(t)$ is isomorphic to $V$. For $t=0$, the fiber is $V\cup_{W}Y$.

We have natural inclusion map
\[i_t:\pi^{-1}(t)\rightarrow \mathcal{B}\]
and one gluing map
\[j=(j_1,j_2):V\bigsqcup Y\rightarrow \pi^{-1}(0).\]

We say cohomology class $\tilde{\theta}\in H^*(\mathcal{B},\mathbb{Q})$ is a lifting of $\theta\in H^*(V,\mathbb{Q})$, if $i_t^*(\tilde{\theta})=\theta$, for all $t\neq 0$.

Let $b:\mathcal{B}\rightarrow V\times \mathbb{C}$ be the blowing down map. All the cohomology classes of $V$ can be lifted to $\mathcal{B}$ by the pullback of $b$. The lifting of a cohomology class in $V$ may not be unique. We suppose that $\{\tilde{\theta}_i\}$ is a lifting of the basis $\{\theta_i\}$.

Let $\{\delta^i\}$ be the dual basis of $\{\delta_i\}$. For a cohomology weighted partition $\eta=\{(\eta_i,\delta_{r_i})\}$, we use $\check{\eta}$ to denote the dual partition $\{(\eta_i,\delta^{r_i})\}$.

The degeneration formula expresses absolute invariants of $V$ in terms of relative invariants of the pairs $(V,W)$ and $(Y, D_0)$,
\begin{equation}\label{degfm}
\begin{aligned}
&\Big\langle\prod_{i=1}^m\tau_{k_i}(\theta_{l_i})\Big\rangle^{V}_{g,\beta}
=\sum_{\eta}C_{\eta}\sum_{[(G^r_1,G^r_2,I)]}\Big\langle\prod_{i_1\in I_1}\tau_{k_{i_1}}\big((i_0\circ j_1)^*\tilde{\theta}_{l_{i_1}}\big)\Big | \eta\Big\rangle_{G^r_1}^{\bullet,V,W}\\
&~~~~~~~~~~~~~~~~~~~~~~~~~~~~~~~~~~~~~~~~~~\times\Big\langle\check{\eta}\Big |\prod_{i_2\in I_2}\tau_{k_{i_2}}\big((i_0\circ j_2)^*\tilde{\theta}_{l_{i_2}}\big)\Big \rangle_{G^r_2}^{\bullet}.
\end{aligned}
\end{equation}
Here, $C_{\eta}=\prod_{j=1}^{l(\eta)}\eta_j\cdot |Aut(\eta)|$, $I=(I_1,I_2)$ with $I_1\bigsqcup I_2=\{1,\ldots,m\}$, $G^r_1$ is a relative graph of the pair $(V,W)$ and $G^r_2$ is that of the pair $(Y,D_0)$. $G^r_1, G^r_2$ both have $l(\eta)$ boundary marked points which are indexed by $\{1,2,\ldots,l(\eta)\}$ and the same assignment of contact orders given by $\{\eta_i\}_{i=1}^{l(\eta)}$. The triple $(G^r_1,G^r_2,I)$ satisfies
\begin{itemize}
 \item[(i)] For any two source domains $\Sigma_1$, $\Sigma_2$ of the relative maps with graphs $G^r_1$, $G^r_2$, the gluing domain, which is given by gluing $\Sigma_1$, $\Sigma_2$ along the relative marked points with the same markings, is connected.
 \item[(ii)] $g=g(G^r_1)+g(G^r_2)+l(\eta)-1$.
 \item[(iii)] $(i_t)_*(\beta)=(i_0)_*\big((j_1)_*(\beta(G^r_1))+(j_2)_*(\beta(G^r_2))\big)$.
 \item[(vi)] The absolute marked points of $G^r_i$ are indexed by the set $I_i$.
 \end{itemize}
Two triples $(G^r_1,G^r_2,I)$ and $(G^r_{1'},G^r_{2'},I')$ with the same relative data $\{\eta_i\}_{i=1}^{l(\eta)}$ are said to be isomorphic, if $I=I'$ and $G^r_i$ is isomorphic to $G^r_{i'}$ for $i=1,2$. We use $[(G^r_1,G^r_2,I)]$ to denote the equivalent class of the triples.

Recall that $p:Y\rightarrow W$ is the projection map, $\iota:W\hookrightarrow V$ is the natural inclusion. We may deduce from condition (iii) that
\begin{equation}\label{dgfm-degcons}
\beta=\beta(G^r_1)+(\iota\circ p)_*(\beta(G^r_2)).
\end{equation}

\subsection{Maulik-Pandharipande's algorithm}\label{algm}
In \cite{MP}, Maulik and Pandharipande related relative Gromov-Witten invariants to absolute Gromov-Witten invariants. We will briefly review their method under the assumption $(V,W)=(\mathbb{P}^4,Q)$. In this case $Y=\mathbb{P}(N_{Q/\mathbb{P}^4}\oplus \mathcal{O}_Q)$.

Firstly, we briefly review Maulik and Pandharipande's algorithm (see \cite{MP}, Theorem 1) which can be used to determine type I and type II invariants from the absolute Gromov-Witten invariants of $Q$.

Let $[F]\in H_2(Y)$ be the curve class of the fiber of $Y$. Type I or Type II invariants with curve class $\beta=d[F]$ will be called fiber invariants. The fiber invariants can be completely solved by the equivariant relative theory of $\mathbb{P}^1$ (see \cite{MP}, Section 1.2).

The distinguished type II invariants are defined to be type II invariants with a distinguished insertion of the form $\tau_0([D_0]\cdot \delta)$, i.e.
\begin{equation*}
\Big\langle\mu\Big |\tau_0([D_0]\cdot \delta)\prod_{i=1}^m\tau_{k_i}(\gamma_{l_i})\Big | \nu\Big \rangle_{g,\beta}
\end{equation*}
where $\delta$ is a cohomology class of $H^*(Y,\mathbb{Q})$ pulling back from $Q$ with $deg(\delta)>0$.

They gave a partial ordering $\stackrel{\circ}{<}$ (see \cite{MP}, Section 1.3) for the distinguished type II invariants which are important to the inductive algorithm.

Now the type I and type II invariants can be determined from the absolute Gromov-Witten invariants of $Q$ by the following several steps.

Step 1. Each type I or type II invariant can be expressed as twisted Gromov-Witten invariants of $Q$ (see (\ref{defTGW})) and rubber invariants via relative virtual localization formula.

Step 2. Twisted Gromov-Witten invariants of $Q$ can be converted into absolute Gromov-Witten invariants of $Q$ by quantum Riemann-Roch theorem \cite{CG}. Rubber invariants can be computed in terms of distinguished type II invariants via rubber calculus (see \cite{MP}, Section 1.5).

Step 3. Via degeneration formula, distinguished type II invariants can be computed in terms of strictly lower distinguished type II invariants with respect to the partial ordering $\stackrel{\circ}{<}$ and some type I invariants. Those type I invariants can be expressed again in terms of Gromov-Witten invariants of $Q$ and distinguished type II invariants which are lower than the original one by Step $1$ and $2$.

Step 4. After finite steps, each type I or type II invariant can be computed in terms of Gromov-Witten invariants of $Q$ and fiber invariants.

The fiber invariants are completely determined. So each type I or type II invariant can be determined from Gromov-Witten invariants of $Q$.

Combing the algorithm given above with degeneration formula, Maulik and Pandharipande came up with a way to determine relative invariants of the pair $(\mathbb{P}^4,Q)$ from the absolute invariants of $\mathbb{P}^4$ and $Q$ (see \cite{MP}, Theorem 2). We give a brief account of it.

The degeneration formula (\ref{degfm})
(set $(V,W)=(\mathbb{P}^4,Q)$) can be seen as equations relating relative invariants of the pair $(\mathbb{P}^4,Q)$ with absolute invariants of $\mathbb{P}^4$.

If we treat relative invariants of the pair $(\mathbb{P}^4,Q)$ as unknowns. Those equations form a lower triangle system with respect to the partial ordering $\stackrel{\circ}{<}$ (see \cite{MP}, Section 2.3) on the relative invariants of the pair $(\mathbb{P}^4,Q)$.

Then we can solve the equations to recover relative invariants of the pair $(\mathbb{P}^4,Q)$ from the absolute invariants of $\mathbb{P}^4$ and relative invariants of the pair $(Y,D_0)$. The latter can be determined from absolute invariants of $Q$ using the algorithm given above.

\section{Proof of main theorem}\label{section-3}
We will prove Theorem \ref{thm-2} in this section. It relies on Theorems \ref{keythm} and \ref{keythm2}. Theorem \ref{keythm} will be proven in Section \ref{partI}. Section \ref{partII} and Appendix A give a complete proof of Theorem \ref{keythm2}. By using Theorems \ref{keythm} and \ref{keythm2}, we give a short proof of Theorem \ref{thm-2} at the end of Section \ref{partI}. In this section, we always assume that the relative pair $(V,W)=(\mathbb{P}^4,Q)$. Thus $Y=\mathbb{P}(N_{Q/\mathbb{P}^4}\oplus \mathcal{O}_Q)$.

\subsection{Part I}\label{partI}
Let $\alpha$ be the generator of curve classes of $Q$. By virtual dimension formula (\ref{abvirdim}), we have
\begin{equation}\label{virdim-Q}
\text{dim}([\overline{\mathcal{M}}_{g,m}(Q,d\alpha)]^{vir})=m.
\end{equation}
Sometimes, we will ignore $\alpha$ and denote the curve class $d\alpha$ by $d$.

Now the following definition makes sense
\begin{equation}\label{def-Q}
N_{g,d}:=\langle\rangle_{g,d}^Q,
\end{equation}
where $g\geq 2$ or $d\neq 0$. Since $N_{g,0}$ can be computed by (\ref{inv-d0}), we only need to compute those $N_{g,d}$ with $d>0$. So we always assume that $d>0$ in the following.

For any absolute Gromov-Witten invariant of $Q$
\[\Big\langle\prod_{i=1}^m\tau_{k_i}(\delta_{l_i})\Big\rangle^{Q}_{g,d},\]
the dimension constraint (\ref{virdim-Q}) implies that there must be one insertion in $\prod_{i=1}^m\tau_{k_i}(\delta_{l_i})$ with the following form: \[\tau_0(1),~~ \tau_0(\delta)~ (deg(\delta)=2)~~ \text{or}~~ \tau_1(1).\]
The insertion can be eliminated by string equation, divisor equation or dilaton equation. After finite steps all the insertions can be eliminated. So $N_{g,d}$ actually include all the information of the Gromov-Witten invariants of $Q$.

We recall that $p:Y\rightarrow Q$ is the projection map. For a fixed type I or type II invariant with genus $g$ and curve class $\beta$, Maulik-Pandharipande's algorithm discussed in Section \ref{algm} actually gives a recursive way to determine it from absolute invariants of $Q$ with genus $g'\leq g$ and degree $d'\leq d$, where $p_*(\beta)=d\alpha$. The absolute invariants can further be determined by $N_{g',d'}$.

We may also deduce from the lower triangle system discussed in Section \ref{algm} (see \cite{MP}, Theorem 2 for more details) that relative invariants of the pair $(\mathbb{P}^4,Q)$ with fixed genus $g$ and degree $d$ can be determined from absolute invariants of $\mathbb{P}^4$ and $Q$ with genus $g'\leq g$ and degree $d'\leq d$. The invariants of $\mathbb{P}^4$ are known by the method of \cite{Ga2}\cite{Gi}\cite{GP}
and that of $Q$ can be determined by $N_{g',d'}$. We may conclude that
\begin{lemma}\label{rl-ab}
Type I and type II invariants with fixed genus $g$ and curve class $\beta$ $\big(p_*(\beta)=d\alpha\big)$ can be determined by $N_{g'\leq g,d'\leq d}$. If we treat absolute invariants of $\mathbb{P}^4$ as known constants, relative invariants of the pair $(\mathbb{P}^4,Q)$ with fixed genus $g$ and degree $d$ can be determined by $N_{g'\leq g,d'\leq d}$.
\end{lemma}

Before we state next lemma, we need some preparations.

We give a $\mathbb{C}^*$-action on $\mathbb{P}^1$ by
\begin{equation}\label{caction}
\begin{aligned}
\mathbb{C}^* \times \mathbb{P}^1 & \longrightarrow \mathbb{P}^1\\
(w,[z_1,z_2]) & \mapsto [z_1,wz_2]
\end{aligned}
\end{equation}
and put $\mathbf{0}=[1,0]$ and $\infty=[0,1]$. $[\mathbf{0}]$ and $[\infty]$ are two classes in $H^*_{\mathbb{C}^*}(\mathbb{P}^1,\mathbb{Q})$ corresponding to the fixed points $\mathbf{0}$ and $\infty$.

Let $t$ be the generator of $H^*_{\mathbb{C}^*}(pt,\mathbb{Q})$ corresponding to the dual of standard representation of $\mathbb{C}^*$. $H^*_{\mathbb{C}^*}(\mathbb{P}^1,\mathbb{Q})$ is a module over the ring $\mathbb{C}[t]$.

$[\mathbf{0}]$ and $[\infty]$ satisfy
\begin{equation}\label{equrls}
[\infty]-[\mathbf{0}]=t,~~~~[\infty]\cdot [\mathbf{0}]=0.
\end{equation}

Let $H$ be the hyperplane class in $\mathbb{P}^4$. With some abuse of notations, the restriction of $H$ to $Q$ will also be denoted as $H$. We may also see $H$ as a cohomology class of $Y$ via pull-back by projection map. In the following, the term
$$H^{n_i}\prod_{k_i=0}^{l_i}(k_i\tau+[D_{\infty}])$$
should be seen as one single insertion, which is equivalent to insert the term
$$ev_i^*\big(H^{n_i}\big)\prod_{k_i=0}^{l_i}\big(k_i\psi_i+ev_i^*([D_{\infty}])\big)$$
in the definition of relative invariants (\ref{relGW}).

Let $\alpha,\beta\in \mathbb{Z}$. We will use the following convention:
\begin{equation}\label{conven}
{\alpha\choose \beta}=
\begin{cases}
  0, & \mbox{if } \alpha<\beta, \\
  1, & \mbox{if } \alpha=\beta, \\
  0, & \mbox{if } \alpha>\beta,~\beta<0,\\
  \frac{(\alpha)!}{\beta!(\alpha-\beta)!}, & \mbox{if } \alpha>\beta,~\beta\geq 0.
\end{cases}
\end{equation}

The next lemma compute some genus $0$ fiber invariants of the pair $(Y,D_0)$, which will be used in the rest of the paper.
\begin{lemma}\label{typeI-g0}
Let $m,\mu\in\mathbb{Z}_{>0}$, $l_1, l_2,\ldots,l_m\in \mathbb{Z}_{\geq 0}$ and $h\in\{0,1,2,3\}$. If $1-\mu+\sum_{i=1}^m l_i\neq h$, then
\[\Big\langle (\mu,H^{3-h})\Big |\prod_{i=1}^m\Bigr\{\prod_{k_i=0}^{l_i}(k_i\tau+[D_{\infty}])\Bigr\}\Big\rangle_{0,\mu[F]}=0.\]
If $1-\mu+\sum_{i=1}^m l_i= h$ and
\[\sum_{i\in I}l_i\leq \mu,~~~\forall I\subset\{1,2,\ldots,m\}\text{~s.t.~}|I|\leq m-2,\]
then we have
\[\Big\langle (\mu,H^{3-h})\Big |\prod_{i=1}^m\Bigr\{\prod_{k_i=0}^{l_i}(k_i\tau+[D_{\infty}])\Bigr\}\Big\rangle_{0,\mu[F]}=5^{h+1}\mu^{m-2}{h+m-2\choose m-2},\]
where we use the convention (\ref{conven}) for the notation ${*\choose *}$.
\end{lemma}

\begin{proof}
The dimension constraint (\ref{relvirdim}) requires that
\begin{equation}\label{dimcons}
\sum_{i=1}^m l_i-\mu+1=h.
\end{equation}
So if $\sum_{i=1}^m l_i-\mu+1\neq h$, we have
\[\Big\langle (\mu,H^{3-h})\Big |\prod_{i=1}^m\Bigr\{\prod_{k_i=0}^{l_i}(k_i\tau+[D_{\infty}])\Bigr\}\Big\rangle_{0,\mu[F]}=0.\]
We may assume $\sum_{i=1}^m l_i-\mu+1=h$ in the following.

We set
\[\overline{\mathcal{M}}_Y=\overline{\mathcal{M}}_{\Gamma}(Y,D_0)\]
where the relative data $\Gamma=(0,\mu[F],\{\mu\},m)$. It is a fiber bundle over $Q$ with fiber
\[\overline{\mathcal{M}}_{\mathbb{P}^1}=\overline{\mathcal{M}}_{\Gamma'}(\mathbb{P}^1,\mathbf{0}).\]
Here, $\Gamma'=(0,\mu[\mathbb{P}^1],\{\mu\},m)$. Let
\[\pi:\overline{\mathcal{M}}_Y\longrightarrow Q\]
be the projection map.

The fiber invariant
\[\Big\langle (\mu,H^{3-h})\Big |\prod_{i=1}^m\Bigr\{\prod_{k_i=0}^{l_i}(k_i\tau+[D_{\infty}])\Bigr\}\Big\rangle_{0,\mu[F]}\]
can be computed by
\begin{equation}\label{FBIfm}
\int_QH^{3-h}\pi_*\Big(\Bigr\{\prod_{i=1}^m\prod_{k_i=0}^{l_i}\big(k_i\psi_i+ev_i^*([D_{\infty}])\big)\Bigr\}\cap [\overline{\mathcal{M}}_Y]^{vir_{\pi}}\Big)
\end{equation}
according to \cite{MP}, Formula (4).

The $\mathbb{C}^*$-action of (\ref{caction}) induces a $\mathbb{C}^*$-action on $\overline{\mathcal{M}}_{\mathbb{P}^1}$ by composition.

According to the analysis in \cite{MP} Section 1.2, the push-forward
\[\pi_*\Big(\Bigr\{\prod_{i=1}^m\prod_{k_i=0}^{l_i}\big(k_i\psi_i+ev_i^*([D_{\infty}])\big)\Bigr\}\cap [\overline{\mathcal{M}}_Y]^{vir_{\pi}}\Big)\]
in (\ref{FBIfm}) can be computed by replacing $t$ in the equivariant integral
\[\Big\langle \mu\Big|\prod_{i=1}^m\Bigr\{\prod_{k_i=0}^{l_i}(k_i\tau+[\infty])\Bigr\}\Big\rangle_{0,\mu}^{T}:=\int_{[\overline{\mathcal{M}}_{\mathbb{P}^1}]_T^{vir}}\prod_{i=1}^m\prod_{k_i=0}^{l_i}\big(k_i\psi_i+ev_i^*([\infty])\big)\]
by $c_1(N_{Q/\mathbb{P}^4})$. Here, $[\overline{\mathcal{M}}_{\mathbb{P}^1}]_T^{vir}$ is the equivariant virtual fundamental class and the integral should be seen as equivariant push-forward to a point.

Combining equality (\ref{dimcons}) and the virtual dimension of $\overline{\mathcal{M}}_{\mathbb{P}^1}$, we may write
\[\Big\langle \mu\Big|\prod_{i=1}^m\Bigr\{\prod_{k_i=0}^{l_i}(k_i\tau+[\infty])\Bigr\}\Big\rangle_{0,\mu}^{T}\]
as $Ct^h$, where $C$ is some constant. So
\[\pi_*\Big(\Bigr\{\prod_{i=1}^m\prod_{k_i=0}^{l_i}\big(k_i\psi_i+ev_i^*([D_{\infty}])\big)\Bigr\}\cap [\overline{\mathcal{M}}_Y]^{vir_{\pi}}\Big)=C\big(c_1(N_{Q/\mathbb{P}^4})\big)^h.\]

From (\ref{FBIfm}), we may deduce that
\[\Big\langle (\mu,H^{3-h})\Big |\prod_{i=1}^m\Bigr\{\prod_{k_i=0}^{l_i}(k_i\tau+[D_{\infty}])\Bigr\}\Big\rangle_{0,\mu[F]}=5^{h+1}C.\]
So it remains to determine $C$.

By the assumption of Lemma \ref{typeI-g0}, we know that
\[\sum_{i\in I}l_i\leq \mu,~~~\forall I\subset\{1,2,\ldots,m\}\text{~s.t.~}|I|\leq m-2.\]
So by Theorem 1.6 in \cite{WL}, it is easy to see that
\[C={h+m-2\choose m-2}.\]
\end{proof}

We may also need to compute fiber invariant in the following form:
\[\Big\langle (\mu,H^{3-h})\Big |\prod_{i=1}^m\Bigr\{H^{n_i}\prod_{k_i=0}^{l_i}(k_i\tau+[D_{\infty}])\Bigr\}\Big\rangle_{0,\mu[F]}.\]
Since $\mu[F]$ is a fiber class, those evaluation maps satisfy
\[p\circ ev_i=\widetilde{ev}_1, ~~~\forall 1\leq i\leq m.\]
Now it is easy to deduce that
\[
\begin{aligned}
&\Big\langle (\mu,H^{3-h})\Big |\prod_{i=1}^m\Bigr\{H^{n_i}\prod_{k_i=0}^{l_i}(k_i\tau+[D_{\infty}])\Bigr\}\Big\rangle_{0,\mu[F]}=\\
&~~~~~~~~~~~~~~~~~~~~\Big\langle (\mu,H^{3-h+\sum_{i} n_i})\Big |\prod_{i=1}^m\Bigr\{\prod_{k_i=0}^{l_i}(k_i\tau+[D_{\infty}])\Bigr\}\Big\rangle_{0,\mu[F]}.
\end{aligned}
\]
\begin{corollary}\label{calfib-1}
\begin{equation*}
\Big\langle (s,H^m)\Big |\Big\{ H^n\prod_{k=0}^{r}(k\tau+[D_{\infty}])\Big\}\Big \rangle_{0,s[F]}=
\begin{cases}
\frac{5}{s}, & m+n=3, r=s-1,\\
0, & otherwise.
\end{cases}
\end{equation*}
\end{corollary}

\begin{corollary}\label{calfib-2}
If $l_1+l_2=\mu-1+h$, then we have
\begin{equation*}
  \Big\langle (\mu,H^{3-h})\Big |\prod_{i=1}^2\Bigr\{\prod_{k_i=0}^{l_i}(k_i\tau+[D_{\infty}])\Bigr\}\Big\rangle_{0,\mu[F]}=5^{h+1}.
\end{equation*}

\end{corollary}

\begin{proof}
Corollary \ref{calfib-1}, \ref{calfib-2} directly follow from Lemma \ref{typeI-g0}.
\end{proof}

\begin{remark}
Corollary \ref{calfib-1} is a special case of \cite{Ga1}, Corollary 5.3.4.
\end{remark}

Let $\rho=\{(r_1,n_1),\ldots,(r_p,n_p)\}$ and $\zeta=\{(l_1,m_1),\ldots,(l_q,m_q)\}$. We set
\[
\begin{aligned}
&\mathcal{F}\big(\rho\big |\zeta\big)=\prod_{i=1}^p(r_i+1)\times\\
&\sum_{I_1,\ldots,I_p}\prod_{i=1}^p\Big\langle\Big(r_i+1,\frac{H^{3-n_i}}{5}\Big)\Big|\prod_{t_i\in I_i} \Big\{H^{m_{t_i}}\prod_{k_{t_i}=0}^{l_{t_i}}(k_{t_i}\tau+[D_{\infty}])\Big\}\Big\rangle_{0,(r_i+1)[F]}
\end{aligned}
\]
where $I_1\sqcup \ldots I_p=\{1,2,\ldots,q\}$. The above invariants will naturally appear in the degeneration formula.

\begin{lemma}\label{discon-fbv}
If all the pairs $(r_i,n_i)\neq (0,0)$, then
\begin{itemize}
\item[(1)] $\mathcal{F}\big(\rho\big |\zeta\big)=0$, if $p>q$;
\item[(2)] $\mathcal{F}\big(\rho\big |\zeta\big)=0$, if $p=q$, but $\rho\neq\zeta$ as sets;
\item[(3)] $\mathcal{F}\big(\rho\big |\zeta\big)=|Aut(\rho)|=|Aut(\zeta)|$, if $\rho=\zeta$,
where $Aut(\rho)$ (resp. $Aut(\zeta)$) is the group of permutation symmetries of $\rho$ (resp. $\zeta$).
\end{itemize}
\end{lemma}
\begin{proof}

(1) If $p>q$, then some $I_j$ must be empty. The corresponding connected invariant is
\[\Big\langle\Big(r_j+1,\frac{H^{3-n_j}}{5}\Big)\Big|\Big\rangle_{0,(r_j+1)[F]}.\]
The dimension constraint (\ref{relvirdim}) requires $r_j+n_j=0$. It forces $(r_j,n_j)=(0,0)$. Since $(r_i,n_i)\neq (0,0)$ for all $i$, the above invariant is zero. So $\mathcal{F}=0$.

(2),(3) If $p=q$, according to the discussion in (1), we only need to consider those partitions $I_1\sqcup \ldots I_p=\{1,2,\ldots,q\}$ with $|I_i|=1$ for all $i$. Suppose that $I_i=\{t_i\}$, the corresponding contribution to $\mathcal{F}$ is
\[\prod_{i=1}^p(r_i+1)\Big\langle\Big(r_i+1,\frac{H^{3-n_i}}{5}\Big)\Big| \Big\{H^{m_{t_i}}\prod_{k_{t_i}=0}^{l_{t_i}}(k_{t_i}\tau+[D_{\infty}])\Big\}\Big\rangle_{0,(r_i+1)[F]}.\]
By Corollary \ref{calfib-1}, it equals to $1$ if $(r_i,n_i)=(l_{t_i},m_{t_i})$ for all $i$, and $0$ otherwise.

So if
\[\{(r_1,n_1),\ldots,(r_p,n_p)\}\neq \{(l_1,m_1),\ldots,(l_q,m_q)\},\]
then we have $\mathcal{F}=0$. If
\[\{(r_1,n_1),\ldots,(r_p,n_p)\}= \{(l_1,m_1),\ldots,(l_q,m_q)\},\]
then the total contribution is $|Aut(\rho)|$ (or $|Aut(\zeta)|$).
\end{proof}

Let $Id$ be the identity element of $H^*(Q,\mathbb{Q})$. $\iota:Q\hookrightarrow \mathbb{P}^4$ is the natural inclusion. Similarly, we treat
\[H^{m}\prod_{k=0}^{l}(k\tau+\iota_*(Id))\]
as one single insertion.

Let $\iota_{\infty}:Q\rightarrow Y$ be the section which is determined by $D_{\infty}$. $\alpha_{\infty}$ is the push-forward of $\alpha$ under  $\iota_{\infty}$.

In order to apply degeneration formula (\ref{degfm}) to the pair $(V,W)=(\mathbb{P}^4,Q)$, we need to lift cohomology classes of $\mathbb{P}^4$. We use notations from Section \ref{ep-defm}. Let $\hat{\iota}:Q\times \mathbb{C}\hookrightarrow \mathcal{B}$ be the inclusion through strict transform.

We lift $\iota_*(H^s)$ to $\hat{\iota}_*(H^s\otimes \mathbf{1}_{\mathbb{C}})$ where $\mathbf{1}_{\mathbb{C}}$ is the identity element of $H^*(\mathbb{C})$. The central fiber of $\mathcal{B}$ is $\mathbb{P}^4\cup Y$. It is easy to see that $\hat{\iota}_*(H^s\otimes \mathbf{1}_{\mathbb{C}})$ becomes $0$ when restricted to $\mathbb{P}^4$, $H^s\cdot[D_{\infty}]$ when restricted to $Y$.

So we have
\begin{equation}\label{appdegfm}
\begin{aligned}
&\Big\langle\prod_{i=1}^s\Big\{H^{m_i}\prod_{k_i=0}^{l_i}(k_i\tau+\iota_*(Id))\Big\}\Big\rangle_{g,d}^{\mathbb{P}^4}=\\
&~~~~~\sum_{\eta}C_{\eta}\sum_{[(G^r_1,G^r_2,I)]}\Big\langle\Big|\eta\Big\rangle_{G^r_1}^{\bullet,\mathbb{P}^4,Q}\Big\langle\check{\eta}\Big|\prod_{i=1}^s \Big\{H^{m_i}\prod_{k_i=0}^{l_i}(k_i\tau+[D_{\infty}])\Big\}\Big\rangle_{G_2^r}^{\bullet}.
\end{aligned}
\end{equation}
Here, we identify classes of the form $H^m\big(\iota_*(Id)\big)^k$ on the L.H.S of (\ref{appdegfm}) with $5^{k-1}\iota_*(H^{m+k-1})$. We also identify $5^{k-1}H^{m+k-1}\cdot [D_{\infty}]$ with $H^m[D_{\infty}]^k$ on the R.H.S of (\ref{appdegfm}) using the fact that $[D_{\infty}]^2=5H\cdot[D_{\infty}]$.

Since those $(l_i,m_i)=(0,0)$ can be removed by applying divisor equation on both sides of (\ref{appdegfm}), we may assume that $(l_j,m_j)\neq (0,0)$ for all $j$. We need $m_i\leq 3$. Otherwise, both sides of (\ref{appdegfm}) vanish.

(A) Suppose that there is a vertex $v^{(1)}\in G_1^r$ satisfying
\[\big(g(v^{(1)}),\beta(v^{(1)})\big)=(g,d[l])\]
where $[l]$ is the curve class of a line in $\mathbb{P}^4$. Firstly, we claim that there will be only one vertex in $G_1^r$. We will show it by contradiction.

If there is another vertex $\tilde{v}^{(1)}\in G_1^r$, then we claim that $\beta(\tilde{v}^{(1)})\neq 0$. The reason is as follows.

Let $\Sigma_1$ (resp. $\Sigma_2$) be a disjoint union of connected curves corresponding to vertices in $G_1^r$ (resp. $G_2^r$). Then if we glue $\Sigma_1$ and $\Sigma_2$ along relative marked points, it yields a connected curve. Let $\Sigma_{\tilde{v}^{(1)}}$ be the connected component corresponding to vertex $\tilde{v}^{(1)}$. Then $\Sigma_{\tilde{v}^{(1)}}$ must contain some relative marked points which mapped to $Q$. So $\beta(\tilde{v}^{(1)})\cdot [Q]\neq 0$. It further implies that $\beta(\tilde{v}^{(1)})\neq 0$.

By the constraint (\ref{dgfm-degcons}) of degeneration formula, we need
\begin{equation}\label{curcls1}
d[l]=\sum_{v\in G_1^r}\beta(v)+\sum_{v\in G_2^r}(\iota\circ p)_*\big(\beta(v)\big)
\end{equation}
where $p:Y\rightarrow Q$ is the natural projection map and $\iota:Q\hookrightarrow \mathbb{P}^4$ is the natural inclusion.

So we have
\[d[l]<\beta(v^{(1)})+\beta(\tilde{v}^{(1)})\leq \sum_{v\in G_1^r}\beta(v)\leq d[l].\]

The contradiction implies that $G_1^r$ only contains one vertex $v^{(1)}$. So the part
\[\Big\langle\Big|\eta\Big\rangle_{G^r_1}^{\bullet,\mathbb{P}^4,Q}\]
on the R.H.S of (\ref{appdegfm}) can be written as
\[\Big\langle\Big|\eta\Big\rangle_{g,d}^{\mathbb{P}^4,Q}.\]

As for those vertices in $G_2^r$, we have
\[\sum_{v\in G_2^r}(\iota\circ p)_*\big(\beta(v)\big)=0\]
by (\ref{curcls1}). Since $\iota_*$ is injective, we have $p_*(\beta(v))=0$ for each vertex $v\in G_2^r$. So each $\beta(v)$ is a fiber class.

Since the gluing domain, which is given by gluing $\Sigma_1$ and $\Sigma_2$ along relative marked points, is a connected genus $g$ curve. The condition $g(v^{(1)})=g$ implies that each vertex $v\in G_2^r$ satisfies $g(v)=0$, and only one relative marked point is assigned to $v$. So there are exactly $l(\eta)$ vertices in $G_2^r$. We set $\{v_1^{(2)},\ldots,v_{l(\eta)}^{(2)}\}$ to be those vertices. Without losing of generality, we may assume that relative marked point $y_i$ is assigned to $v_i^{(2)}$.

Recall that $\eta=\{(\eta_i,\delta_{r_i})\}_{i=1}^{l(\eta)}$ and $\check{\eta}=\{(\eta_i,\delta^{r_i})\}_{i=1}^{l(\eta)}$. The contact order assigned to relative marked point $y_i$ is $\eta_i$.

The contribution of such triple $(G_1^r,G_2^r,\eta)$ to the R.H.S of (\ref{appdegfm}) becomes
\begin{equation}\label{contr-A}
C_{\eta}\Big\langle\Big|\eta\Big\rangle_{g,d}^{\mathbb{P}^4,Q}\Big\langle\check{\eta}\Big|\prod_{i=1}^s \Big\{H^{m_i}\prod_{k_i=0}^{l_i}(k_i\tau+[D_{\infty}])\Big\}\Big\rangle_{G_2^r}^{\bullet},
\end{equation}
where $C_{\eta}=|Aut(\eta)|\prod_{i=1}^{l(\eta)}\eta_i$ and
\[
\begin{aligned}
\Big\langle\check{\eta}\Big|&\prod_{i=1}^s \Big\{H^{m_i}\prod_{k_i=0}^{l_i}(k_i\tau+[D_{\infty}])\Big\}\Big\rangle_{G_2^r}^{\bullet}=\\
&\frac{1}{|Aut(\eta)|}\prod_{i=1}^{l(\eta)}\Big\langle\Big(\eta_i,\delta^{r_i}\Big)\Big|\prod_{t_i\in I_i} \Big\{H^{m_{t_i}}\prod_{k_{t_i}=0}^{l_{t_i}}(k_{t_i}\tau+[D_{\infty}])\Big\}\Big\rangle_{0,\eta_i[F]}.
\end{aligned}
\]
Here, $I_1,\ldots, I_{l(\eta)}$ form a partition of $\{1,2\ldots,s\}$ which are determined by the assignment of absolute marked points in $G_2^r$. So
\[
\begin{aligned}
C_{\eta}\Big\langle\check{\eta}\Big|&\prod_{i=1}^s \Big\{H^{m_i}\prod_{k_i=0}^{l_i}(k_i\tau+[D_{\infty}])\Big\}\Big\rangle_{G_2^r}^{\bullet}=\\
&\prod_{i=1}^{l(\eta)}\eta_i\Big\langle\Big(\eta_i,\delta^{r_i}\Big)\Big|\prod_{t_i\in I_i} \Big\{H^{m_{t_i}}\prod_{k_{t_i}=0}^{l_{t_i}}(k_{t_i}\tau+[D_{\infty}])\Big\}\Big\rangle_{0,\eta_i[F]}.
\end{aligned}
\]

If $(\eta_i,\delta_{r_i})=(1,Id)$, then $(\eta_i,\delta^{r_i})=(1,\frac{H^3}{5})$. The corresponding connected fiber invariant becomes
\[\Big\langle\Big(1,\frac{H^3}{5}\Big)\Big|\prod_{t_i\in I_i} \Big\{H^{m_{t_i}}\prod_{k_{t_i}=0}^{l_{t_i}}(k_{t_i}\tau+[D_{\infty}])\Big\}\Big\rangle_{0,[F]}.\]
The dimension constraint requires
\[\sum_{t_i\in I_i}(l_{t_i}+m_{t_i})=0.\]
Since $(l_{t_i},m_{t_i})\neq (0,0)$, it implies that $I_i$ is empty. Now fiber invariant can be computed by divisor equation and Corollary \ref{calfib-1}, i.e.
\[\Big\langle\Big(1,\frac{H^3}{5}\Big)\Big|\Big\rangle_{0,[F]}=\Big\langle\Big(1,\frac{H^3}{5}\Big)\Big|[D_{\infty}]\Big\rangle_{0,[F]}=1.\]

The dimension constraint requires $deg(\delta^{r_i})$ to be even. So all $deg(\delta_{r_i})$ are also even. We may rewrite $\eta$ as
\begin{equation*}
\big\{(l_1'+1,H^{m_1'}),\ldots,(l_{s'}'+1,H^{m_{s'}'}),(1,Id)^{Id(\eta)}\big\},
\end{equation*}
where $Id(\eta)$ is uniquely determined by the constraint $\sum_{i}(l_i'+1)+Id(\eta)=5d$.

We assume that the standard order becomes
\[(l+1,H^m)>(l'+1,H^{m'}),~\text{if}~l>l'~\text{or}~l=l',~m>m'\]
when restricted to weighted pairs of the form $(l+1,H^m)$.

If we vary $G_2^r$ but fix $\eta$ in (\ref{contr-A}), then the total contribution of those graphs to the degeneration formula (\ref{appdegfm}) becomes
\[R_{g,d}(\varrho)\mathcal{F}(\varrho|\zeta),\]
where we set $\varrho=\{(l_1',m_1'),\ldots,(l_{s'}',m_{s'}')\}$, $\zeta=\{(l_1,m_1)\ldots,(l_s,m_s)\}$ and
\[
\begin{aligned}
R_{g,d}(\varrho)=\Big\langle\Big|\eta\Big\rangle_{g,d}^{\mathbb{P}^4,Q}.
\end{aligned}
\]

Let $\{(a_i,b_i)\}_{i=1}^r$ be set of integer pairs. We give a standard order for integer pairs by
\[(a,b)>(a',b'),~\text{if}~a>a'~~~\text{or}~~~a=a',~b>b'.\]
Without further explanation, we always write integer pairs in decreasing order by standard order.

The dimension constraint for relative invariant $R_{g,d}$ requires
\[\sum (l_i'+m_i')=5d+1-g.\]
We note that $\sum(l_i'+1)+Id(\eta)=5d$. So $\sum(l_i'+1)\leq 5d$. Now we may deduce that $\varrho\in S_{g,d}'$, where $S_{g,d}'$ is given by definition \ref{maindf-2}.

The total contribution of case (A) can be written as
\[\sum_{\varrho\in S_{g,d}'}R_{g,d}(\varrho)\mathcal{F}(\varrho|\zeta).\]

(B) Suppose that there is a vertex $v^{(2)}\in G_2^r$ satisfying
\[\Big(g(v^{(2)}),(\iota\circ p)_*\big(\beta(v^{(2)})\big)\Big)=(g,d[l]).\]
Then by (\ref{curcls1}), we know that $\beta(v)=0$ for each vertex $v\in G_1^r$. But using the same argument as in the proof of case (A), we can also show that $\beta(v)\neq 0$ for each vertex $v\in G_1^r$. The contradiction implies that there will be no vertices in $G_1^r$. So $\eta$ is empty. The connectivity of gluing curve further implies that there are no other vertices in $G_2^r$.

It is easy to see that
\[\beta(v^{(2)})=(\iota_{\infty}\circ p)_*(\beta(v^{(2)}))+(\beta(v^{(2)})\cdot [D_0])[F].\]
Since $\eta$ is empty, we know that $\beta(v^{(2)})\cdot [D_0]=0$. Since
\[(\iota\circ p)_*(\beta(v^{(2)}))=d[l],\]
we have
\[p_*(\beta(v^{(2)})=d\alpha.\]
So
\[(\iota_{\infty}\circ p)_*(\beta(v^{(2)}))=d\alpha_{\infty}.\]
Now we may summarize that $\beta(v^{(2)})=d\alpha_{\infty}$.

The contribution of such triple $(G_1^r,G_2^r,\eta)$ to the R.H.S of (\ref{appdegfm}) becomes
\[
\Big\langle\Big|\prod_{i=1}^s\Big\{H^{m_i}\prod_{k_i=0}^{l_i}(k_i\tau+[D_{\infty}])\Big\}\Big\rangle_{g,d\alpha_{\infty}},
\]
where we omit the weighted cohomology partition since it is empty. We denote the above type I invariant as $B_{g,d}\big(\zeta\big)$, where $\zeta=\{(l_1,m_1),\ldots,(l_s,m_s)\}$ is the same as in case (A).

(C) For the rest of terms on the R.H.S of (\ref{appdegfm}), each vertex $v^{(1)}$ of $G_1^r$ (we suppose that $g(v^{(1)})=g_1$ and $\beta(v^{(1)})=d_1[l]$)
satisfies the condition that $(g_1,d_1)<(g,d)$. Here, the partial ordering for the pairs $(g,d)$ is given by
\[(g',d')<(g,d)~ \text{if either}~ g'<g, ~d'\leq d~\text{or}~ g'\leq g,~ d'<d. \]
Each vertex $v^{(2)}$ of $G_2^r$ (we suppose that $g(v^{(2)})=g_1$ and $\beta(v^{(2)})=\beta_2$) satisfies the condition that  $(g_2,d_2)<(g,d)$ (we set $(\iota\circ p)_*(\beta_2)=d_2[l]$). So by Lemma \ref{rl-ab}, they all can be determined by $N_{g',d'}$ with $(g',d')<(g,d)$.

We set
\[
A_{g,d}\big(\zeta \big)=\Big\langle\prod_{i=1}^s\Big\{H_i^{m_i}\prod_{k_1=0}^{l_i}(k_i\tau+\iota_*(Id))\Big\}\Big\rangle_{g,d}^{\mathbb{P}^4}.
\]
The dimension constraint for $A_{g,d}$ plus the assumption $(l_i,m_i)\neq (0,0)$ imply
\[\zeta=\{(l_1,m_1),\ldots,(l_s,m_s)\}\in S_{g,d},\]
where $S_{g,d}$ is given by definition \ref{maindf-1}.

From the above discussion, we may conclude that
\begin{equation}\label{cor-degfm}
A_{g,d}(\zeta)=\sum_{\varrho\in S_{g,d}'}R_{g,d}(\varrho)\mathcal{F}(\varrho|\zeta)+B_{g,d}(\zeta)+NPT.
\end{equation}
Here and in the below, we use "$NPT$" to denote those terms which can be determined by $N_{g',d'}$ with $(g',d')<(g,d)$.

For each $\rho\in S_{g,d}'$, we may similarly deduce from the degeneration formula that
\begin{equation}\label{substerm}
A_{g,d}(\rho)=\sum_{\varrho\in S_{g,d}'}R_{g,d}(\varrho)\mathcal{F}(\varrho|\rho)+B_{g,d}(\rho)+NPT.
\end{equation}

If we give a partial ordering on $S_{g,d}'$ according to the number of integer pairs, then the $|S_{g,d}'|\times |S_{g,d}'|$ matrix formed by $\mathcal{F}(\varrho|\rho)$ becomes upper triangle with nonzero elements on the main diagonal by Lemma \ref{discon-fbv}. So for fixed $\zeta\in S_{g,d}$, there exist constants $C_{\rho}$ such that
\begin{equation}\label{sol-rel}
\mathcal{F}(\varrho|\zeta)=\sum_{\rho\in S_{g,d}'}\mathcal{F}(\varrho|\rho)C_{\rho},~~~\forall \varrho\in S_{g,d}'.
\end{equation}

Combing (\ref{cor-degfm}), (\ref{substerm}) and (\ref{sol-rel}), we have
\begin{equation}\label{key-rel}
A_{g,d}(\zeta)-\sum_{\rho\in S_{g,d}'}A_{g,d}(\rho)C_{\rho}=B_{g,d}(\zeta)-\sum_{\rho\in S_{g,d}'}B_{g,d}(\rho)C_{\rho}+NPT.
\end{equation}

We may conclude that

\begin{theorem}\label{keythm}
For fixed $d>0$, $g\geq 0$ and $\zeta\in S_{g,d}$, there exist constants $C_{\rho}$ which satisfy (\ref{key-rel}). Those constants $C_{\rho}$ can be uniquely determined by relations (\ref{sol-rel}).
\end{theorem}

In the case of $g=0,1$, Gathmann \cite{Ga1} tries to evaluate $A_{g,d}\big(\zeta_{g,d}\big)$ by degeneration formula with $\zeta_{g,d}=\{(5d,1-g)\}$. His method actually shows that
\[A_{g,d}\big(\zeta_{g,d}\big)=B_{g,d}\big(\zeta_{g,d}\big)+NPT\]
and
\[
B_{g,d}\big(\zeta_{g,d}\big)=
\begin{cases}
dN_{0,d},~~ \text{if}~ g=0,\\
5d(5d+1)N_{1,d},~~ \text{if}~ g=1.
\end{cases}
\]
Then he recursively solve the equations
\[B_{g,d}\big(\zeta_{g,d}\big)=A_{g,d}\big(\zeta_{g,d}\big)-NPT\]
to get $N_{0,d}$ and $N_{1,d}$.

Here the method of Gathmann to derive type I invariant $B_{g,d}\big(\zeta_{g,d}\big)$ is different from that of Maulik-Pandharipande. We may directly compute $B_{g,d}\big(\zeta_{g,d}\big)$ by using Maulik-Pandharipande's algorithm. The results should be the same.

Once we know $N_{0,d}$ and $N_{1,d}$ for all $d$, the relative invariants of $(\mathbb{P}^4,Q)$ for $g=0,1$ can be determined by Lemma \ref{rl-ab}. The Conjecture \ref{conj} for $g=0,1$ follows.

In the case of $g=2,3$, we try to generalize the above method and prove that
\begin{theorem}\label{keythm2}
For each pair $(g,d)$ with $g=2,3$ and $d>0$, there always exists $\zeta_{g,d}\in S_{g,d}$ such that term
\[B_{g,d}(\zeta_{g,d})-\sum_{\rho\in S_{g,d}'}B_{g,d}(\rho)C_{\rho}\]
on the R.H.S of (\ref{key-rel}) equals to $C_{g,d}N_{g,d}$ with $C_{g,d}\neq 0$, where those $C_{\rho}$ are determined by (\ref{sol-rel}).
\end{theorem}

Since the proofs for genus 2 and 3 are similar, we will give a detailed proof for genus 3 in the subsection below and give a short proof for genus 2 in Appendix A.

\begin{proof}[Proof of Theorem \ref{thm-2}]We need to show that Conjecture \ref{conj} is true for $g=2,\,3$. Since $N_{g,0}$ are known by (\ref{inv-d0}), we can assume that $d>0$.

By Theorem \ref{keythm} \ref{keythm2}, we know that there exists $\zeta_{g,d}\in S_{g,d}$ such that
\begin{equation}\label{Final-g3}
N_{g,d}=\frac{1}{C_{g,d}}\big\{A_{g,d}(\zeta_{g,d})-\sum_{\rho\in S_{g,d}'}A_{g,d}(\rho)C_{\rho}-NPT\big\}.
\end{equation}
Here, we always assume that $g=2,\,3$.

We recall that $NPT$ can be determined by $N_{g',d'}$ with $(g',d')<(g,d)$ and $A_{g,d}$ are treated as known constants. So we can recursively solve equations (\ref{Final-g3}) to get $N_{g,d}$. Once $N_{g,d}$ are known, relative invariants of $(\mathbb{P}^4,Q)$ can be determined by Lemma \ref{rl-ab}.
\end{proof}

\subsection{Part II}\label{partII}
In order to proof Theorem \ref{keythm2} for $g=3$, we choose
\begin{equation}\label{chose-g3}
\zeta_{3,d}=\{(5d-4,0),(1,0),(1,0)\},
\end{equation}
for each $d>0$. Obviously, $\zeta_{3,d}\in S_{3,d}$ but $\zeta_{3,d}\notin S_{3,d}'$.

We need to determine those $C_{\rho}$ in Theorem \ref{keythm2}. Firstly, we assume that $d\geq 2$.

By the analysis in the above subsection, we know that
\begin{equation}\label{deg-g3}
A_{3,d}(\zeta_{3,d})=\sum_{\varrho\in S_{3,d}'}R_{3,d}(\varrho)\mathcal{F}(\varrho|\zeta_{3,d})+B_{3,d}(\zeta_{3,d})+NPT.
\end{equation}
We need to figure out those $\varrho$ such that $\mathcal{F}(\varrho|\zeta_{3,d})$ does not vanish.

We set $\varrho=\{(a_i,b_i)\}_{i=1}^r$. Since $\sum(a_i+b_i)=5d-2>0$, $\varrho$ can not be empty. We also use $|\varrho|$ to denote the number of pairs in $\varrho$.

(I) If $|\varrho|\geq|\zeta_{3,d}|=3$, then we have $\mathcal{F}(\varrho|\zeta_{3,d})=0$ by the fact that $\varrho\in S_{3,d}'$, $\zeta_{3,d}\notin S_{3,d}'$ and Lemma \ref{discon-fbv}.

(II) If $|\varrho|=1$, then since $\varrho\in S_{3,d}'$ we have
\[a_1+b_1=5d-2,~~0\leq b_1\leq 3.\]
We may rewrite it as $\{(5d-2-s_1,s_1)\}$ with $s_1\in\{0,1,2,3\}$.

Let $\rho_{1,s_1}=\{(5d-2-s_1,s_1)\}$. By Lemma \ref{typeI-g0}, we can compute that
\[\mathcal{F}\big(\rho_{1,s_1}|\zeta_{3,d}\big)=5^{s_1}(s_1+1)(5d-1-s_1)^2.\]
The contribution of those $\varrho$ to the R.H.S of (\ref{deg-g3}) is
\[\sum_{s_1=0}^3 R_{3,d}\big(\rho_{1,s_1}\big) 5^{s_1}(s_1+1)(5d-1-s_1)^2.\]

(III) If $|\varrho|=2$, we set
\[(l_1,m_1)=(5d-4,0),~~(l_2,m_2)=(1,0),~~(l_3,m_3)=(1,0).\]
By definition, we have
\begin{equation*}
\begin{aligned}
&\mathcal{F}(\varrho|\zeta_{3,d})=\\
&~~~\sum_{I_1,I_2}\prod_{i=1}^2\Big\langle\Big(a_i+1,\frac{H^{3-b_i}}{5}\Big)\Big|\prod_{t_i\in I_i} \Big\{H^{m_{t_i}}\prod_{k_{t_i}=0}^{l_{t_i}}(k_{t_i}\tau+[D_{\infty}])\Big\}\Big\rangle_{0,(a_i+1)[F]}
\end{aligned}
\end{equation*}
where $I_1\sqcup I_2=\{1,2,3\}$.

(i) If $|I_i|=0$ for some $i$, then by Lemma \ref{discon-fbv} we have
\[\Big\langle\Big(a_i+1,\frac{H^{3-b_i}}{5}\Big)\Big|\Big\rangle_{0,(a_i+1)[F]}=0.\]
So the corresponding contribution to $\mathcal{F}(\varrho|\zeta_{3,d})$ vanishes.

(ii) Now we assume that $|I_1|=1$. There are two cases.

(a) If $I_1=\{1\}$, then the corresponding connected fiber invariant in $\mathcal{F}(\varrho|\zeta_{3,d})$ becomes
\[\Big\langle\Big(a_1+1,\frac{H^{3-b_1}}{5}\Big)\Big|\prod_{k_1=0}^{5d-4}(k_1\tau+[D_{\infty}])\Big\rangle_{0,(a_1+1)[F]}.\]
By Corollary \ref{calfib-1}, it vanishes unless $(a_1,b_1)=(5d-4,0)$. Since $\varrho\in S_{3,d}'$, we may deduce that
\[a_2+b_2=2.\]
So we may write $\varrho$ as $\{(5d-4,0),(2-s_2,s_2)\}$ with $s_2\in\{0,1,2\}$.

(b) If $I_1=\{2\}~\text{or}~\{3\}$, then we can similarly deduce that the corresponding term in $\mathcal{F}(\varrho|\zeta_{3,d})$ vanishes unless
\[(a_1,b_1)=(1,0), ~(a_2,b_2)=(5d-3-s_3,s_3)\]
where $s_3\in\{0,1,2,3\}$. But since we assume that $d\geq 2$, such $\varrho$ does not satisfy $(a_1,b_1)\geq(a_2,b_2)$. So we will not choose it.

(iii) If $|I_2|=1$, we may deduce that only those $\varrho$ such that
\[(a_1,b_1)=(5d-3-s_3,s_3),~ (a_2,b_2)=(1,0)\]
give non-vanishing contributions to $\mathcal{F}(\varrho|\zeta_{3,d})$ and satisfy our requirement
\[(a_1,b_1)\geq (a_2,b_2).\]

We can summarize from the discussion of (i), (ii), (iii) that if $|\varrho|=2$, then $\mathcal{F}(\varrho|\zeta_{3,d})$ vanishes unless
\begin{equation*}
\varrho=\{(5d-4,0),(2-s_2,s_2)\}~~\text{or}~~\{(5d-3-s_3,s_3),(1,0)\}
\end{equation*}
with $s_2\in\{0,1,2\}$ and $s_3\in\{0,1,2,3\}$.

Now we begin to compute $\mathcal{F}(\varrho|\zeta_{3,d})$ for those exceptional $\varrho$.

Let
\[\rho_{2,s_2}=\{(5d-4,0),(2-s_2,s_2)\},~~\rho_{3,s_3}=\{(5d-3-s_3,s_3),(1,0)\}.\]

By Corollary \ref{calfib-1}, \ref{calfib-2}, we have
\[\mathcal{F}(\rho_{i,s_i}|\zeta_{3,d})=
\begin{cases}
  (3-s_2)5^{s_2}, & \mbox{if } i=2, \\
  2(5d-2-s_3)5^{s_3}, & \mbox{if } i=3.
\end{cases}
\]

The total contribution of those $\varrho$ with $|\varrho|=2$ to the R.H.S of (\ref{deg-g3}) is
\[
\sum_{s_2=0}^2 R_{3,d}\big(\rho_{2,s_2}\big)(3-s_2)5^{s_2}+\sum_{s_3=0}^3 R_{3,d}\big(\rho_{3,s_3}\big)2(5d-2-s_3)5^{s_3}.
\]

We may conclude from the discussion of (I), (II) and (III) that
\[
\begin{aligned}
A_{3,d}(\zeta_{3,d})=&\sum_{s_1=0}^3 R_{3,d}\big(\rho_{1,s_1}\big) 5^{s_1}(s_1+1)(5d-1-s_1)^2+\\
&\sum_{s_2=0}^2 R_{3,d}\big(\rho_{2,s_2}\big)(3-s_2)5^{s_2}+\sum_{s_3=0}^3 R_{3,d}\big(\rho_{3,s_3}\big)2(5d-2-s_3)5^{s_3}.
\end{aligned}
\]

By (\ref{substerm}), we know that
\[
A_{3,d}(\rho_{i,s_i})=\sum_{\varrho\in S_{3,d}'}R_{3,d}(\varrho)\mathcal{F}(\varrho|\rho_{i,s_i})+B_{3,d}(\rho_{i,s_i})+NPT
\]
for $i=1,2,3$.

We may write the term
\begin{equation}\label{remov-part}
\sum_{\varrho\in S_{3,d}'}R_{3,d}(\varrho)\mathcal{F}(\varrho|\rho_{i,s_i})
\end{equation}
more explicitly by using the similar analysis as above.

For fixed $\rho_{1,s_1}$, by Lemma \ref{discon-fbv} we know that $\mathcal{F}(\varrho|\rho_{1,s_1})$ vanishes unless $\varrho=\rho_{1,s_1}$. In the latter case, we have
\[\mathcal{F}(\rho_{1,s_1}|\rho_{1,s_1})=|Aut(\rho_{1,s_1})|=1.\]
So we conclude that
\[\sum_{\varrho\in S_{3,d}'}R_{3,d}(\varrho)\mathcal{F}(\varrho|\rho_{1,s_1})=R_{3,d}(\rho_{1,s_1}).\]

For fixed $\rho_{2,s_2}$, we may discuss according to $|\varrho|$.

If $|\varrho|=2$, then we may deduce that $\mathcal{F}(\varrho|\rho_{2,s_2})$ vanishes unless $\varrho=\rho_{2,s_2}$. In the latter case,
\[\mathcal{F}(\rho_{2,s_2}|\rho_{2,s_2})=|Aut(\rho_{2,s_2})|=1.\]

If $|\varrho|=1$, then similar to the discussion in (II), we know that $\varrho=\rho_{1,s_1}$ for some $s_1$. By Corollary \ref{calfib-2}, we have
\[
\mathcal{F}(\rho_{1,s_1}|\rho_{2,s_2})=
\begin{cases}
(5d-1-s_1)5^{s_1-s_2},~ \text{if}~ s_1\geq s_2,\\
0,~ \text{if}~ s_1< s_2.
\end{cases}
\]
So we have
\[\sum_{\varrho\in S_{3,d}'}R_{3,d}(\varrho)\mathcal{F}(\varrho|\rho_{2,s_2})=R_{3,d}(\rho_{2,s_2})+\sum_{s_1= s_2}^3 R_{3,d}(\rho_{1,s_1})(5d-1-s_1)5^{s_1-s_2}.\]

Similarly we can deduce that
\[\sum_{\varrho\in S_{3,d}'}R_{3,d}(\varrho)\mathcal{F}(\varrho|\rho_{3,s_3})=R_{3,d}(\rho_{3,s_3})+\sum_{s_1= s_3}^3 R_{3,d}(\rho_{1,s_1})(5d-1-s_1)5^{s_1-s_3}.\]

It happens that we may write (\ref{remov-part}) in the following way
\[\sum_{\varrho\in S_{3,d}'}R_{3,d}(\varrho)\mathcal{F}(\varrho|\rho_{i,s_i})=\sum_{j=1}^3\sum_{s_j}R_{3,d}(\rho_{j,s_j})\mathcal{F}(\rho_{j,s_j}|\rho_{i,s_i}).\]

We know that those $C_{\rho}$ in Theorem \ref{keythm2} satisfy a system of equations (\ref{sol-rel}), which are reduced to the following equations in our case.

For each $\rho_{j,s_j}$, we have
\[\mathcal{F}(\rho_{j,s_j}|\zeta_{3,d})=\sum_{i=1}^3\sum_{s_i}\mathcal{F}(\rho_{j,s_j}|\rho_{i,s_i})C_{\rho_{i,s_i}}.\]
By solving the system of equations, we get
\[\left\{
\begin{array}{rcl}
C_{\rho_{1,s_1}} & = & (5d-1-s_1)(s_1+1)(\frac{s_1}{2}-5d)5^{s_1},\\
C_{\rho_{2,s_2}} & = & (3-s_2)5^{s_2},\\
C_{\rho_{3,s_3}} & = & 2(5d-2-s_3)5^{s_3}.\\
\end{array}\right.
\]
So we have
\[
\begin{aligned}
&A_{3,d}(\zeta_{3,d})-\sum_{i=1}^3\sum_{s_i}A_{3,d}(\rho_{i,s_i})C_{\rho_{i,s_i}}=\\
&~~~~~~~~~~~~~~~~B_{3,d}(\zeta_{3,d})-\sum_{i=1}^3\sum_{s_i}B_{3,d}(\rho_{i,s_i})C_{\rho_{i,s_i}}+NPT.
\end{aligned}
\]

\begin{remark}
In the case of $d=1$, we can still get the above equality with the exception that $s_3\in\{0,1,2\}$.
\end{remark}

\begin{proof}[Proof of Theorem \ref{keythm2} for $g=3$] We need to show that
\begin{equation}\label{coef-g3}
B_{3,d}(\zeta_{3,d})-\sum_{i=1}^3\sum_{s_i}B_{3,d}(\rho_{i,s_i})C_{\rho_{i,s_i}}
\end{equation}
can be written as $C_{3,d}N_{3,d}$ with $C_{3,d}\neq 0$.

Applying Lemma \ref{typeI-nf1}, \ref{typeI-nf2} and \ref{typeI-nf3} in the below, we can compute that
\[B_{3,d}(\zeta_{3,d})-\sum_{i=1}^3\sum_{s_i}B_{3,d}(\rho_{i,s_i})C_{\rho_{i,s_i}}=24(5d-3)(5d-4)N_{3,d}.\]
Obviously, $24(5d-3)(5d-4)\neq 0$. The proof is complete.
\end{proof}

We are left to compute those $B_{3,d}$ appeared in (\ref{coef-g3}). For further applications, we will compute $B_{g,d}(\zeta)$ instead ($g$ is not fixed and $|\zeta|\leq 3$). We mainly use relative virtual localization formula and the virtual pushforward property defined by Gathmann \cite{Ga1}. Let us recall the definition of virtual pushforward property at first.

\begin{definition}[\cite{Ga1}, Definition 5.2.1]\label{vpfp}
Let $f:M\rightarrow M'$ be a morphism of moduli spaces of stable (absolute, relative or rubber) maps. Let $\gamma\in A^*(M)$ which is made up of evaluation classes and cotangent line classes.
$f$ is said to satisfy the virtual pushforward property if the following two conditions hold:
\begin{itemize}
  \item[(1)] If the dimension $\gamma\cap [M]^{vir}$ is bigger than the virtual dimension of $[M']^{vir}$, then $f_*(\gamma\cap [M]^{vir})=0$
  \item[(2)] If the dimension $\gamma\cap [M]^{vir}$ is equal to the virtual dimension of $[M']^{vir}$, then $f_*(\gamma\cap [M]^{vir})$ is a scalar multiple of $[M']^{vir}$.
\end{itemize}
\end{definition}

\begin{example}[\cite{Ga1}, Lemma 5.2.4]\label{vpfp-ex}
Let $\pi:\overline{\mathcal{M}}_{g,m}(Q,d)\rightarrow \overline{\mathcal{M}}_{g,m-n}(Q,d)$ be the forgetful morphism which forgets the last $n$ marked points. Then $\pi$ satisfies the virtual pushforward property.
\end{example}

By using Example \ref{vpfp-ex}, we can prove a vanishing result for the twisted Gromov-Witten invariants of $Q$.

Let
$$N_{g,m,d}=H^0(\Sigma,f^*N_{Q/\mathbb{P}^4})\ominus H^1(\Sigma,f^*N_{Q/\mathbb{P}^4})$$
be the virtual vector bundle (see \cite{CG} for more details) on $\overline{\mathcal{M}}_{g,m}(Q,d)$. The twisted Gromov-Witten invariants of $Q$ are defined by
\begin{equation}\label{defTGW}
\begin{aligned}
\Big\langle\prod_{i=1}^m\tau_{k_i}(\delta_{l_i})&\prod_{j=1}^s\text{ch}_{d_j}(N_{g,m,d})\Big\rangle^{Q}_{g,d}\\
&:=\int_{[\overline{\mathcal{M}}_{g,m}(Q,d)]^{vir}}\prod_{i=1}^m\psi_i^{k_i}ev_i^*(\delta_{l_i})\prod_{j=1}^s\text{ch}_{d_j}(N_{g,m,d}),
\end{aligned}
\end{equation}
where $\text{ch}_{d_j}(N_{g,m,d})$ is the $d_j$th Chern character of $N_{g,m,d}$.

The twisted Gromov-Witten invariants will naturally appear when we apply relative virtual localization formula. We always assume that $d>0$ in the following several lemmas.
\begin{lemma}\label{TGW}
If $\sum_j d_j>0$, we have
\[\Big\langle\prod_{i=1}^m\tau_{k_i}(\delta_{l_i})\prod_{j=1}^s\text{ch}_{d_j}(N_{g,m,d})\Big\rangle^{Q}_{g,d}=0.\]
\end{lemma}
\begin{proof}
Let $\pi:\overline{\mathcal{M}}_{g,m}(Q,d)\rightarrow \overline{\mathcal{M}}_{g,0}(Q,d)$ be the forgetful morphism, we have
\[\pi^*(N_{g,0,d})=N_{g,m,d}\]
by \cite{CG}, Formula iv. So $\pi^*\big(\text{ch}_{d_j}(N_{g,0,d})\big)=\text{ch}_{d_j}(N_{g,m,d})$. The pushforward
\[\pi_*\Big(\Big\{\prod_{i=1}^m\psi_i^{k_i}ev_i^*(\delta_{l_i})\prod_{j=1}^s\text{ch}_{d_j}(N_{g,m,d})\Big\}\cap[\overline{\mathcal{M}}_{g,m}(Q,d)]^{vir}\Big)\]
equals to
\[\prod_{j=1}^s\text{ch}_{d_j}(N_{g,0,d})\cap\pi_*\Big(\prod_{i=1}^m\psi_i^{k_i}ev_i^*(\delta_{l_i})\cap[\overline{\mathcal{M}}_{g,m}(Q,d)]^{vir}\Big)\]
by projection formula. The dimension constraint requires that
\[deg\Big(\prod_{j=1}^s\text{ch}_{d_j}(N_{g,m,d})\Big)+deg\Big(\prod_{i=1}^m\psi_i^{k_i}ev_i^*(\delta_{l_i})\Big)=m.\]
Since
\[deg\Big(\prod_{j=1}^s\text{ch}_{d_j}(N_{g,m,d})\Big)=\sum_jd_j>0,\]
we have
\[deg\Big(\prod_{i=1}^m\psi_i^{k_i}ev_i^*(\delta_{l_i})\Big)<m=dim [\overline{\mathcal{M}}_{g,m}(Q,d)]^{vir}.\]
So
\[dim\Big(\prod_{i=1}^m\psi_i^{k_i}ev_i^*(\delta_{l_i})\cap[\overline{\mathcal{M}}_{g,m}(Q,d)]^{vir}\Big)> dim [\overline{\mathcal{M}}_{g,0}(Q,d)]^{vir}=0.\]
Now the virtual pushforward property of Example \ref{vpfp-ex} implies that
\[\pi_*\Big(\prod_{i=1}^m\psi_i^{k_i}ev_i^*(\delta_{l_i})\cap[\overline{\mathcal{M}}_{g,m}(Q,d)]^{vir}\Big)=0.\]
The lemma follows.
\end{proof}

The next two lemmas are firstly proven by Gathmann in \cite{Ga1}, which also play an important role in the computation of $B_{g,d}$ in the below.

\begin{lemma}[\cite{Ga1}, Corollary 5.2.5]\label{vpfp-1}
Let $M$ be a moduli space of stable maps to $Q$, possibly with disconnected domains. Let $ft:M\rightarrow M'$ be a forgetful morphism which forgets a given subset of the marked points and/or connected components. ($M'$ is also a moduli space of stable maps to $Q$, with in general fewer marked points and connected components.) Then $ft$ satisfies the virtual pushforward property.
\end{lemma}

\begin{lemma}[\cite{Ga1}, Theorem 5.2.7 or \cite{QF}, Theorem 5.1]\label{vpfp-2}
Let $\mathcal{M}$ be a moduli space of stable relative maps $\overline{\mathcal{M}}_{G^r}^{\bullet,\sim}(Y,D_0)$, or stable relative maps to a non-rigid target $\overline{\mathcal{M}}_{G^r}^{\bullet,\sim}(Y,D_0\cup D_{\infty})$. Let $p:\mathcal{M}\rightarrow M$ be the morphism that projects the curves in $Y$ down to the base $Q$, forget a given subset of (absolute and/or relative) marked points and/or connected components, stabilizes the result. (Thus $M$ is a moduli space of stable maps to $Y$, possibly with disconnected domains, whose combinatorial data is determined by $p$ and $\mathcal{M}$.) We assume that $p$ is well defined, i.e. that every rational (resp. elliptic) connected component that is not forgotten by $p$ and whose curve class is a fiber class of $Y$ has at least 3 (resp. 1) marked points that are not forgotten by $p$. Then $p$ satisfies the virtual pushforward property.
\end{lemma}

\begin{remark}
Lemma \ref{vpfp-2} is reproved by F.Qu in \cite{QF} using a different method.
\end{remark}

Now let us begin to compute $B_{g,d}$ by applying relative virtual localization formula given in Appendix B and Lemmas \ref{vpfp-1} and \ref{vpfp-2}.

Firstly, we compute some $B_{g,d}(\zeta)$ with $|\zeta|=1$.
\begin{lemma}[\cite{Ga1}]\label{typeI-nf1}
Let $\zeta_{g,d,m}^{1}=\{(5d+1-g-m,m)\}\in S_{g,d}$. Then we have
\[B_{g,d}(\zeta_{g,d,m}^{1})=
\begin{cases}
X_{g,d}N_{g,d},& \text{if}~m=0,\\
dN_{g,d},& \text{if}~m=1,\\
0, & \text{if}~m=2,3\\
\end{cases}
\]
where
\begin{equation}\label{X-gd}
X_{g,d}=(5d+2-g)\big((5d+1-g)(g-1)+5d\big).
\end{equation}
\end{lemma}

\begin{remark}
If $m=2,3$, then $B_{g,d}(\zeta_{g,d,m}^{1})$ can be computed by applying the virtual pushforward property of Lemma \ref{vpfp-2}. If $m=1$, then $B_{g,d}(\zeta_{g,d,m}^{1})$ can be computed by using Proposition 5.3.2 in \cite{Ga1}. If $m=0$, then $B_{g,d}(\zeta_{g,d,m}^{1})$ can be derived from Proposition 5.4.1 in \cite{Ga1} which need some technical assumption (\cite{Ga1}, Conjecture 5.2.9). We will compute $B_{g,d}(\zeta_{g,d,m}^{1})$ in a new way when $m=0$, which does not need further assumption. For completeness, we will also give a computation of $B_{g,d}(\zeta_{g,d,m}^{1})$ when $m=1,2,3$, by using Lemma \ref{vpfp-2} and Proposition 5.3.2 in \cite{Ga1}.
\end{remark}

\begin{proof}
We recall that
\[B_{g,d}(\zeta_{g,d,m}^{1})=\int_{[\overline{\mathcal{M}}_{\Gamma}(Y,D_0)]^{vir}}ev_1^*(H^m)\prod_{k=0}^{5d+1-g-m}\Big(k\psi_1+ev_1^*([D_{\infty}])\Big),\]
where the data of stable relative maps
\[\Gamma=(g,\beta,\{\nu_i\}_{i=1}^n,m)=(g,d\alpha_{\infty},\emptyset,1).\]

We will divide it into the following three cases.

(A) Case $m=2,3$.
Let
\[f:\overline{\mathcal{M}}_{\Gamma}(Y,D_0)\rightarrow \overline{\mathcal{M}}_{g,1}(Q,d)\]
be the morphism that projects curves in $Y$ down to the base which does not forget the only marked point. Since the class $H$ of $Y$ is pulled back from the base $Q$, we have
\[ev_1^*(H)=(\underline{ev}_1\circ f)^*(H)\]
where we add an underline to the evaluation map of $\overline{\mathcal{M}}_{g,1}(Q,d)$ so as to distinguish it from that of $\overline{\mathcal{M}}_{\Gamma}(Y,D_0)$. Now by projection formula, we have
\begin{eqnarray*}
&  & f_*\Big(ev_1^*(H^m)\prod_{k=0}^{5d+1-g-m}\big(k\psi_1+ev_1^*([D_{\infty}])\big)\cap [\overline{\mathcal{M}}_{\Gamma}(Y,D_0)]^{vir}\Big)\\
&= & \underline{ev}_1^*(H^m)f_*\Big(\prod_{k=0}^{5d+1-g-m}\big(k\psi_1+ev_1^*([D_{\infty}])\big)\cap [\overline{\mathcal{M}}_{\Gamma}(Y,D_0)]^{vir}\Big).
\end{eqnarray*}
The fact that $m=2,3$ implies that the dimension of
\[\prod_{k=0}^{5d+1-g-m}\big(k\psi_1+ev_1^*([D_{\infty}])\big)\cap [\overline{\mathcal{M}}_{\Gamma}(Y,D_0)]^{vir}\]
is bigger than the virtual dimension of $\overline{\mathcal{M}}_{g,1}(Q,d)$ which equals to $1$.

So by the virtual pushforward property of Lemma \ref{vpfp-2}, we have
\[f_*\Big(ev_1^*(H^m)\prod_{k=0}^{5d+1-g-m}\big(k\psi_1+ev_1^*([D_{\infty}])\big)\cap [\overline{\mathcal{M}}_{\Gamma}(Y,D_0)]^{vir}\Big)=0\]
which further implies that $B_{g,d}(\zeta_{g,d,m}^{1})=0$.

(B) Case $m=1$. Still we have
\begin{eqnarray*}
&  & f_*\Big(ev_1^*(H)\prod_{k=0}^{5d-g}\big(k\psi_1+ev_1^*([D_{\infty}])\big)\cap [\overline{\mathcal{M}}_{\Gamma}(Y,D_0)]^{vir}\Big)\\
&= & \underline{ev}_1^*(H)f_*\Big(\prod_{k=0}^{5d-g}\big(k\psi_1+ev_1^*([D_{\infty}])\big)\cap [\overline{\mathcal{M}}_{\Gamma}(Y,D_0)]^{vir}\Big).
\end{eqnarray*}
Now by Proposition 5.3.2 in \cite{Ga1}, we have
\[f_*\Big(\prod_{k=0}^{5d-g}\big(k\psi_1+ev_1^*([D_{\infty}])\big)\cap [\overline{\mathcal{M}}_{\Gamma}(Y,D_0)]^{vir}\Big)=[\overline{\mathcal{M}}_{g,1}(Q,d)]^{vir}.\]
Then by divisor equation, we have
\[B_{g,d}(\zeta_{g,d,1}^{1})=dN_{g,d}.\]

(C) Case $m=0$. The computation, in this case, is different from the above two cases.

Instead of using the morphism $f$, we define a new morphism
\[\pi:\overline{\mathcal{M}}_{\Gamma}(Y,D_0)\rightarrow \overline{\mathcal{M}}_{g,0}(Q,d)\]
which projects the curves in $Y$ down to the base and forget the only marked point. By the virtual property of Lemma \ref{vpfp-2}, we have
\[\pi_*\Big(\omega\cap [\overline{\mathcal{M}}_{\Gamma}(Y,D_0)]^{vir}\Big)=c[\overline{\mathcal{M}}_{g,0}(Q,d)]^{vir}\]
where $c$ is the scalar and
\[\omega=\prod_{k=0}^{5d+1-g}\big(k\psi_1+ev_1^*([D_{\infty}])\big).\]
Now it is easy to see that
\[B_{g,d}(\zeta_{g,d,0}^{1})=cN_{g,d}.\]

We then use the relative virtual localization formula (\ref{lc-0}) (set $(V,W)=(\mathbb{P}^4,Q)$ in the formula) to compute $c$. Firstly, we need to give an equivariant lift of $\omega$.

The projection map $p:Y\rightarrow Q$ is $\mathbb{C}^*$-invariant. So cohomology classes pulled back from $Q$ have natural equivariant lifts. Since $D_{\infty}$ is fixed under the $\mathbb{C}^*$-action, $[D_{\infty}]$ can also be seen as an equivariant class.

The natural equivariant lift of cotangent line bundle $L_1$ will still be denoted as $L_1$. We also use $\psi_1$ to denote the equivariant first Chern class of $L_1$.

Now it is easy to see that $\omega$ has a natural equivariant lift $\omega_T$.

By the relative virtual localization formula (\ref{lc-0}), we know that
\[\pi_*\Big(\omega\cap[\overline{\mathcal{M}}_{\Gamma}(Y,D_0)]^{vir}\Big)\]
equals to the non-equivariant limit of
\begin{equation}\label{lc-1pt}
\sum_{G^l_0}\frac{1}{|Aut(G^l_0)|}(\pi\circ\tau_{G^l_0})_*\Biggr(\frac{\tau_{G^l_0}^*(\omega_T)\cap[\overline{\mathcal{M}}_{G^l_0}]^{vir}}{e_T(N_{G^l_0}^{vir})}\Biggr).
\end{equation}

Firstly, we suppose that $\pi\circ\tau_{G^l_0}$ maps $\overline{\mathcal{M}}_{G^l_0}$ to class of $\overline{\mathcal{M}}_{g,0}(Q,d)$ whose virtual codimension is bigger than $0$.

Let
\[\frac{\tau_{G^l_0}^*(\omega_T)}{e_T(N_{G^l_0}^{vir})}=\sum_{i\in \mathbb{Z}}\omega_it^i\]
be the expansion according to the equivariant parameter $t$. Then the contribution of such localization graph $G^l_0$ is
\[\frac{1}{|Aut(G^l_0)|}(\pi\circ\tau_{G^l_0})_*\Big(\omega_0\cap[\overline{\mathcal{M}}_{G^l_0}]^{vir}\Big).\]
Since
\begin{eqnarray*}
  \text{dim}\bigr(\omega_0\cap[\overline{\mathcal{M}}_{G^l_0}]^{vir}\bigr) &=& \text{dim} \bigr(\prod_{k=0}^{5d+1-g}\big(k\psi_1+ev_1^*([D_{\infty}])\big)\cap[\overline{\mathcal{M}}_{\Gamma}(Y,D_0)]^{vir}\bigr)\\
   &=& 0\\
   &=& \text{dim}\big([\overline{\mathcal{M}}_{g,0}(Q,d)]^{vir}\big)>\text{dim}[(\pi\circ\tau_{G^l_0})(\overline{\mathcal{M}}_{G^l_0})]^{vir},
\end{eqnarray*}
those virtual pushforward properties of Lemmas \ref{vpfp-1} and \ref{vpfp-2} can be used to show that
\[(\pi\circ\tau_{G^l_0})_*\Big(\omega_0\cap[\overline{\mathcal{M}}_{G^l_0}]^{vir}\Big)=0.\]

So we only need to consider those localization graphs $G^l_0$ which satisfy the condition that: $\pi\circ\tau_{G^l_0}$ maps $\overline{\mathcal{M}}_{G^l_0}$ to class of $\overline{\mathcal{M}}_{g,0}(Q,d)$ whose virtual codimension is $0$. Each such graph $G^l_0$ gives rise to a contribution to $c$. Then we add up these contributions and show that
\[c=(5d+2-g)\big((5d+1-g)(g-1)+5d\big)=X_{g,d}.\]

Let $G^l_0$ be such a graph. There are two cases.

(I) Firstly, we assume that there exists one vertex $v_{\infty}\in V_{\infty}$ of $G^l_0$ such that $g(v_{\infty})=g$ and $\beta(v_{\infty})=d\alpha_{\infty}$. There must be no edges in $G^l_0$. Otherwise,
\[d\alpha_{\infty}=\sum_{v\in V_{\infty}}\beta(v)+\sum_{v\in V_0}\beta(v)+\sum_{e\in E}d(e)[F]> \beta(v_{\infty})=d\alpha_{\infty}.\]
The connectedness of $G_0^l$ further implies that $G_0^l$ contains only one vertex $v_{\infty}$. So we have
\[\overline{\mathcal{M}}_{G^l_0}\simeq \overline{\mathcal{M}}_{g,1}(Q,d)\]
and $|Aut(G^l_0)|=1$.

In this case, it is easy to see that $ev_1\circ \tau_{G_0^l}$ factors through $D_{\infty}$, i.e.
\[ev_1\circ \tau_{G_0^l}:\overline{\mathcal{M}}_{G^l_0}\longrightarrow D_{\infty}\simeq Q
\stackrel{\iota_{\infty}}{\longrightarrow} Y.\]
The restriction of equivariant class $[D_{\infty}]$ to $D_{\infty}$ becomes
\[t+c_1(N_{Q/\mathbb{P}^4}).\]
So
\[\tau_{G_0^l}^*\big(ev_1^*([D_{\infty}])\big)=t+\tau_{G_0^l}^*\big(ev_1^*(c_1(N_{Q/\mathbb{P}^4}))\big).\]

The restriction of $\psi_1$ to $\overline{\mathcal{M}}_{G^l_0}$ becomes a $\psi$-class of $\overline{\mathcal{M}}_{g,1}(Q,d)$.

So we may conclude that
\[\tau_{G^l_0}^*(\omega_T)=\prod_{k=0}^{5d+1-g}\big(k\psi_1+ev_1^*(c_1(N_{Q/\mathbb{P}^4}))+t\big).\]
Here, since $ev_1\circ \tau_{G_0^l}$ can also be seen as an evaluation map of $\overline{\mathcal{M}}_{g,1}(Q,d)$, for simplicity we omit $\tau_{G_0^l}$ and use $ev_1$ to denote it.

It is easy to see that $\overline{F}_{G^l_0}$ is the simple fixed locus. So by (\ref{el-simpls}), we have
\[\frac{1}{e_T(N_{G^l_0}^{vir})}=\frac{1}{e_{T}(H^0/H^1(f_{v_{\infty}}^*N_{Q/\mathbb{P}^4}))}=\frac{1}{e_T(N_{g,1,d})}.\]
We may express $e_{T}(N_{g,1,d})$ as
\[t^{\text{ch}_0(N_{g,1,d})}\text{exp}\Big(\sum_{i=1}^{\infty}\frac{(-1)^{i-1}(i-1)!}{t^i}\text{ch}_i(N_{g,1,d})\Big).\]
By Riemann-Roch theorem
\[\text{ch}_0(N_{g,1,d})=dim\,H^0(\Sigma,f_{v_{\infty}}^*N_{Q/\mathbb{P}^4})-dim\,H^1(\Sigma,f_{v_{\infty}}^*N_{Q/\mathbb{P}^4})=5d+1-g.\]
So we may further write it as
\[t^{5d+1-g}\text{exp}\Big(\sum_{i=1}^{\infty}\frac{(-1)^{i-1}(i-1)!}{t^i}\text{ch}_i(N_{g,1,d})\Big).\]

Now the term
\[\frac{\tau_{G^l_0}^*(\omega_T)\cap[\overline{\mathcal{M}}_{G^l_0}]^{vir}}{e_T(N_{G^l_0}^{vir})}\]
in (\ref{lc-1pt}) becomes
\[\frac{\prod_{k=0}^{5d+1-g}\bigr(k\psi_1+ev_1^*\big(c_1(N_{Q/\mathbb{P}^4})\big)+t\bigr)}{t^{5d+1-g}\text{exp}\Big(\sum_{i=1}^{\infty}\frac{(-1)^{i-1}(i-1)!}{t^i}\text{ch}_i(N_{g,1,d})\Big)}\cap [\overline{\mathcal{M}}_{g,1}(Q,d)]^{vir}.\]

It is easy to compute that the $t^0$-part $\omega_0$ becomes
\[\frac{(5d+2-g)(5d+1-g)}{2}\psi_1+(5d+2-g) ev_1^*\big(c_1(N_{Q/\mathbb{P}^4})\big)-\text{ch}_1(N_{g,1,d}).\]

In this case, the morphism
\[\pi\circ\tau_{G^l_0}:\overline{\mathcal{M}}_{g,1}(Q,d)\longrightarrow \overline{\mathcal{M}}_{g,0}(Q,d)\]
just becomes the forgetful morphism which forget the only marked point.

So it is easy to deduce from dilaton equation, divisor equation and Lemma \ref{TGW} that
\[\frac{1}{|Aut(G^l_0)|}(\pi\circ\tau_{G^l_0})_*\Big(\omega_0\cap[\overline{\mathcal{M}}_{g,1}(Q,d)]^{vir}\Big)\]
equals to $X_{g,d}[\overline{\mathcal{M}}_{g,0}(Q,d)]^{vir}$ where $X_{g,d}$ is given by (\ref{X-gd}).

We may summarize that the contribution of the localization graph in case (I) to the scalar $c$ is $X_{g,d}$.

(II) Next, we suppose that there is a vertex $v_0\in V_0$ of $G^l_0$ such that $g(v_0)=g$ and $p_*\big(\beta(v_0)\big)=d\alpha$. Here, we recall that $p:Y\rightarrow Q$ is the natural projection map.

Firstly, we show that $v_0$ is the only vertex of $V_0$.

Suppose that the target for a stable relative map $f$ in the fixed locus $\overline{F}_{G^l_0}$ is $Y_s\cup Y$. Let $\Sigma_{v_0}$ be the connected component corresponding to $v_0$. We use $f_{v_0}$ to denote the restriction of $f$ to $\Sigma_{v_0}$. $f_{v_0}$ maps $\Sigma_{v_0}$ to $Y_s$.

Recall that
\[c_{s_1,l_1}: Y_s=Y_{s_1}\cup Y\cup Y_{l_1}\longrightarrow Y\]
is the natural contraction to the central $Y$. The curve class
\[\beta_{v_0}=(c_{s_1,l_1}\circ f_{v_0})_*([\Sigma_{v_0}]).\]
is independent of the choosing of $s_1$, $l_1$ by the predeformability condition.

We may write $\beta_{v_0}\in H_2(Y)$ as
\[d_{v_0}\alpha_{\infty}+m_{v_0}[F]\]
where $d_{v_0}$ and $m_{v_0}$ are some integers.

Now by definition,
\[\beta(v_0)=(\iota_0\circ p)_*(\beta_{v_0})=d_{v_0}\alpha_0\]
where $\alpha_0$ is the push-forward of $\alpha$ under the inclusion map $\iota_0:Q\simeq D_0\hookrightarrow Y$. The assumption $p_*\big(\beta(v_0)\big)=d\alpha$ then implies that $d_{v_0}=d$.

The condition $f^{-1}(D_0)=0$ implies that $f_{v_0}^{-1}(D_0)=0$. So
\[m_{v_0}=D_0\cdot\beta_{v_0}=0.\]

Now if there exists another vertex $v'_0$ of $V_0$. The curve class $\beta_{v'_0}$ can not be zero. Otherwise, the predeformability condition implies that no edges are incident to $v'_0$, which contradicts to the condition that $G^l_0$ is connected. We may write $\beta_{v'_0}$ as
\[d_{v'_0}\alpha_{\infty}+m_{v'_0}[F].\]
Similarly, we can show that $m_{v'_0}=0$. Since $\beta_{v'_0}\neq 0$ and $\beta_{v'_0}$ is an effective curve class, we know that $d_{v'_0}>0$.

Now since
\[d\alpha_{\infty}=\sum_{v\in V_{\infty}}\beta(v)+\sum_{v\in V_0}\beta(v)+\sum_{e_r}d(e_r)[F],\]
we have
\[d\alpha=p_*(d\alpha_{\infty})= p_*\Big(\sum_{v\in V_{\infty}}\beta(v)+\sum_{v\in V_0}\beta(v)\Big)\geq p_*\big(\beta(v_0)+\beta(v'_0)\big)>d\alpha.\]

The contradiction implies that there is only one vertex $v_0$ in $V_0$. The inequality
\[d\alpha=p_*(d\alpha_{\infty})\geq p_*\Big(\sum_{v\in V_{\infty}}\beta(v)+\sum_{v\in V_0}\beta(v)\Big)\geq p_*(\beta(v_0))=d\alpha\]
further implies that for each $v\in V_{\infty}$, $p_*(\beta(v))=0$. Since those $\beta(v)$ can be written as $d_v\alpha_{\infty}$, we may deduce that $\beta(v)=0$ for all $v\in V_{\infty}$. So all the components corresponding to vertices in $V_{\infty}$ are contractible.

Since $g(v_0)=g$, we may deduce that each vertex $v_i\in V_{\infty}$ satisfies $g(v_i)=0$ and there will be no loops in the graph $G_0^l$. So there is only one edge $e_i$ connecting $v_i$ to $v_0$.

Since there is only one marked point, the stable condition requires that components corresponding to vertices in $V_{\infty}$ must all degenerate into points.

Let
\[f_{v_0}^{-1}(D_{\infty})=\sum_{i=1}^s d_iz_i.\]
Since
\[\beta_{v_0}\cdot D_{\infty}=d\alpha_{\infty}\cdot D_{\infty}=5d,\]
we have $\sum_i d_i=5d$. By the predeformability condition, we know that each point $z_i$ corresponds to one edge (we suppose that this edge is $e_i$ and $d(e_i)=d_i$). It also implies that the number of edges $|E|=s$.

If the only marked point $x_1$ is distributed to the vertex $v_0$, then the evaluation map $ev_1$ factors through $D_0$, i.e.
\[ev_1:\overline{\mathcal{M}}_{G_0^l}\rightarrow D_0\hookrightarrow Y.\]
So $ev_1^*([D_{\infty}])=0$. Then
\[\tau_{G^l_0}^*(\omega_T)=ev_1^*([D_{\infty}])\prod_{k=1}^{5d+1-g}\big(k\psi_1+ev_1^*([D_{\infty}])\big)=0.\]

So we may assume that $x_1$ is distributed to some vertex in $V_{\infty}$. Suppose that this vertex is $v_1$. In this case, $ev_1$ factors through $D_{\infty}$, i.e.
\[ev_1:\overline{\mathcal{M}}_{G_0^l}\stackrel{ev_1^{E}}{\longrightarrow}Q\simeq D_{\infty}\hookrightarrow Y,\]
where $ev_1^{E}$ is defined as (\ref{rel-eva}).

Similar to the discussion in case (I), we have
\[ev_1^*([D_{\infty}])=t+(ev_1^{E})^*\big(c_1(N_{Q/\mathbb{P}^4})\big).\]

The $d_1$th tensor power of the cotangent line bundle $L_1$ can be identified as $(ev_1^{E})^*(N_{D_{\infty}/Y}^*)\simeq (ev_1^{E})^*(N_{Q/\mathbb{P}^4}^*)$. So
\[\psi_1=-\frac{t+(ev_1^{E})^*\big(c_1(N_{Q/\mathbb{P}^4})\big)}{d_1}.\]
Now it is easy to see that
\[
\tau_{G^l_0}^*(\omega_T)=\Big(t+(ev_1^{E})^*\big(c_1(N_{Q/\mathbb{P}^4})\big)\Big)^{5d+2-g} \prod_{k=0}^{5d+1-g}\Big(1-\frac{k}{d_1}\Big).
\]

Obviously, if $d_1\leq 5d+1-g$, then $\tau_{G^l_0}^*(\omega_T)=0$. So we may assume that $d_1\geq 5d+2-g$.

In this case, by (\ref{rvlf}) we have
\[\frac{1}{e_T(N_{G^l_0}^{vir})}=\frac{\prod_{r=1}^s d_r }{-t-{\psi}_{\infty}}\prod_{r=1}^s\frac{1}{N(v_r)}\prod_{r=1}^s\frac{t+(ev_r^{E})^*\big(c_1(N_{Q/\mathbb{P}^4})\big)}{\bigr(t+(ev_r^{E})^*\big(c_1(N_{Q/\mathbb{P}^4})\big)\bigr)^{d_r}\frac{d_r!}{d_r^{d_r}}}.\]

Since $val(v_1)=2$ and the component corresponding to $v_1$ degenerates into the marked point $x_1$, we have
\[\frac{1}{N(v_1)}=\frac{1}{t+(ev_1^{E})^*\big(c_1(N_{Q/\mathbb{P}^4})\big)}\]
by (\ref{rvlf-ii}). Since $val(v_r)=1$ for $r\geq 2$, by (\ref{rvlf-i}) we have
\[\frac{1}{N(v_r)}=\frac{1}{d_r},~r\neq 1.\]

So $\tau_{G^l_0}^*(\omega_T)/e_T(N_{G^l_0}^{vir})$ equals to
\[\frac{\Big(t+(ev_1^{E})^*\big(c_1(N_{Q/\mathbb{P}^4})\big)\Big)^{5d+2-g-d_1} \prod_{k=0}^{5d+1-g}\Big(1-\frac{k}{d_1}\Big)}{(-t-\psi_{\infty})\frac{d_1!}{d_1^{d_1+1}}\prod_{r=2}^s \bigr(t+(ev_r^{E})^*\big(c_1(N_{Q/\mathbb{P}^4})\big)\bigr)^{d_r}\frac{d_r!}{d_r^{d_r}}}.
\]
It is easy to see that when $d_1\geq 5d+2-g$, the expansion of $\tau_{G^l_0}^*(\omega_T)/e_T(N_{G^l_0}^{vir})$ according to $t$ contains only negative powers of $t$. So the $t^0$-part of
\[\frac{1}{|Aut(G^l_0)|}(\pi\circ\tau_{G^l_0})_*\Bigg(\frac{\tau_{G^l_0}^*(\omega_T)\cap[\overline{\mathcal{M}}_{G^l_0}]^{vir}}{e_T(N_{G^l_0}^{vir})}\Bigg)\]
equals to zero.

We may summarize that the contribution of each $G^l_0$ in case (II) to the scalar $c$ is zero.

In conclusion, if we add up contributions from both case (I) and case (II), then we have $c=X_{g,d}$.
\end{proof}

Next, let us evaluate some $B_{g,d}(\zeta)$ with $|\zeta|=2$.

\begin{lemma}\label{typeI-nf2}
Let $\zeta_{g,d,m,l_1,l_2}^2=\{(l_1,m),(l_2,0)\}\in S_{g,d}$. We have
\[
B_{g,d}(\zeta_{g,d,m,l_1,l_2}^2)=
\begin{cases}
Y_{g,d}(l_1,l_2)N_{g,d},&\text{if}~m=0,\\
Z_{g,d}(l_1,l_2)N_{g,d},&\text{if}~m=1,\\
0,&\text{if}~m=2,3,\\
\end{cases}
\]
where
\[Y_{g,d}(l_1,l_2)=f_{g,d}(l_1,l_2)-f_{g,d}(5d+1-g,0)+5dX_{g,d}+5d(5d+1-g)\]
with $f_{g,d}(l_1,l_2)$ equals to
\begin{equation}\label{m0}
\begin{aligned}
&\frac{5d}{2}\Big(l_1^2(l_1+1)+l_2^2(l_2+1)\Big)+\\
&\frac{g-1}{12}
\Big(l_1(l_1^2-1)(3l_1+2)+l_2(l_2^2-1)(3l_2+2)\Big)+\\
&\frac{(l_1+1)(l_2+1)}{2}\Big(25d^2(2g+1)+(5d-l_1l_2)(2g-1)(1-g)\Big).\\
\end{aligned}
\end{equation}
And $X_{g,d}$ is defined by (\ref{X-gd}),
\[\begin{aligned}
&Z_{g,d}(l_1,l_2)=\\
&~~~5d^2(l_2+1)+\frac{l_1(l_1+1)+(2g-1)l_2(l_2+1)}{2}d-\frac{(5d+1-g)(5d-g)d}{2}.
\end{aligned}\]
\end{lemma}

\begin{remark}
If $m=3$, then $B_{g,d}(\zeta_{g,d,m,l_1,l_2}^2)$ can be computed by applying the virtual pushforward property of Lemma \ref{vpfp-2}. If $m=2$, then $B_{g,d}(\zeta_{g,d,m,l_1,l_2}^2)$ can be computed by using Proposition 5.3.2 in \cite{Ga1}. We will compute $B_{g,d}(\zeta_{g,d,m,l_1,l_2}^2)$ in a new and uniform way .
\end{remark}

\begin{proof}
Let
\[\pi:\overline{\mathcal{M}}_{\Gamma}(Y,D_0)\rightarrow \overline{\mathcal{M}}_{g,0}(Q,d)\]
be the morphism which projects the curves in $Y$ down to the base and forget all the marked points. Here $\Gamma=(g,d\alpha_{\infty},\emptyset,2)$.

Let
\[\omega^{1}= ev_1^*(H^m)\prod_{k_1=0}^{l_1}\Big(k_1\psi_1+ev_1^*\big([D_{\infty}]\big)\Big)\prod_{k_2=0}^{l_2}\Big(k_2\psi_2+ev_2^*\big([D_{\infty}]\big)\Big)\]
be a cohomology class of $Y$. The condition $\{(l_1,m),(l_2,0)\}\in S_{g,d}$ implies that the dimension of
\[\text{dim}\Big(\omega^{1}\cap [\overline{\mathcal{M}}_{\Gamma}(Y,D_0)]^{vir}\Big)=0=\text{dim}\Big([\overline{\mathcal{M}}_{g,0}(Q,d)]^{vir}\Big).\]
So by Lemma \ref{vpfp-2}, we have
\[\pi_*\Big(\omega^{1}\cap [\overline{\mathcal{M}}_{\Gamma}(Y,D_0)]^{vir}\Big)=c[\overline{\mathcal{M}}_{g,0}(Q,d)]^{vir}.\]
It further implies that $B_{g,d}(\zeta_{g,d,m,l_1,l_2}^2)=cN_{g,d}$.

It is not easy to compute $c$ by directly applying the relative virtual localization formula (\ref{lc-0}). To get over it, we need to introduce another cohomology class
\[\omega^{2}=ev_1^*(H^m) \prod_{k_1=0}^{5d+1-g-m}\Big(k_1\psi_1+ev_1^*\big([D_{\infty}]\big)\Big)ev_2^*([D_{\infty}])\]
and use the following equality
\[\pi_*\Big(\omega^{1}\cap [\overline{\mathcal{M}}_{\Gamma}(Y,D_0)]^{vir}\Big)=\pi_*\Big(\bigr((\omega^{1}-\omega^{2})+\omega^{2}\bigr)\cap [\overline{\mathcal{M}}_{\Gamma}(Y,D_0)]^{vir}\Big).\]
It turns out that the part
\begin{equation}\label{part-1}
\pi_*\Big((\omega^{1}-\omega^{2})\cap [\overline{\mathcal{M}}_{\Gamma}(Y,D_0)]^{vir}\Big)
\end{equation}
can be computed by using the relative virtual localization formula (\ref{lc-0}) and
\begin{equation}\label{part-2}
\pi_*\Big(\omega^{2}\cap [\overline{\mathcal{M}}_{\Gamma}(Y,D_0)]^{vir}\Big)
\end{equation}
can be computed by using the divisor equation and Lemma \ref{typeI-nf1}.

Let us compute (\ref{part-1}) at first.

By the virtual property of Lemma \ref{vpfp-2}, we know that
\[\pi_*\Big((\omega^{1}-\omega^{2})\cap [\overline{\mathcal{M}}_{\Gamma}(Y,D_0)]^{vir}\Big)=C^{1,2}[\overline{\mathcal{M}}_{g,0}(Q,d)]^{vir},\]
where $C^{1,2}$ is some scalar to be determined.

We will determine $C^{1,2}$ by using the relative virtual localization formula.

Similar to the proof of Lemma \ref{typeI-nf1} case (C), $\omega^{1}$ and $\omega^{2}$ have natural equivariant lifts $\omega_T^{1}$ and $\omega_T^{2}$ respectively.

Now by the relative virtual localization formula (\ref{lc-0}), (\ref{part-1}) equals to the non-equivariant limit of
\begin{equation*}
\sum_{G^l_0}\frac{1}{|Aut(G^l_0)|}(\pi\circ\tau_{G^l_0})_*\Biggr(\frac{\tau_{G^l_0}^*(\omega_T)\cap[\overline{\mathcal{M}}_{G^l_0}]^{vir}}{e_T(N_{G^l_0}^{vir})}\Biggr)
\end{equation*}
where $\omega_T=\omega^{1}_T-\omega^{2}_T$.

Just as in the proof of Lemma \ref{typeI-nf1} case (C), we only need to consider the following two types of graphs $G^l_0$.

(I) If there exists one vertex $v_{\infty}\in V_{\infty}$ of $G^l_0$ such that $g(v_{\infty})=g$ and $\beta(v_{\infty})=d\alpha_{\infty}$, then we have
\[\overline{\mathcal{M}}_{G^l_0}\simeq \overline{\mathcal{M}}_{g,2}(Q,d)\]
and the automorphism group $Aut(G^l_0)$ satisfies $|Aut(G^l_0)|=1$. The morphism
\[\pi\circ\tau_{G^l_0}:\overline{\mathcal{M}}_{g,2}(Q,d)\longrightarrow \overline{\mathcal{M}}_{g,0}(Q,d)\]
just becomes the forgetful morphism which forget all the marked points.

Let us consider the case $m=0$ at first.

In this case, $\tau_{G^l_0}^*(\omega_T^1)$ becomes
\[\prod_{k_1=0}^{l_1}\Big(k_1\psi_1+ev_1^*(c_1(N_{Q/\mathbb{P}^4}))+t\Big)\prod_{k_2=0}^{l_2}\Big(k_2\psi_2+ev_2^*(c_1(N_{Q/\mathbb{P}^4}))+t\Big),\]
and
\[\frac{1}{e_{T}(N_{g,2,d})}=\frac{1}{t^{5d+1-g}\text{exp}\Big(\sum_{i=1}^{\infty}\frac{(-1)^{i-1}(i-1)!}{t^i}\text{ch}_i(N_{g,2,d})\Big)}.\]
After expansion according to $t$, the $t^0$-part of
\[\frac{\prod_{k_1=0}^{l_1}\Big(k_1\psi_1+ev_1^*(c_1(N_{Q/\mathbb{P}^4}))+t\Big)\prod_{k_2=0}^{l_2}\Big(k_2\psi_2+ev_2^*(c_1(N_{Q/\mathbb{P}^4}))+t\Big)}{t^{5d+1-g}\text{exp}\Big(\sum_{i=1}^{\infty}\frac{(-1)^{i-1}(i-1)!}{t^i}\text{ch}_i(N_{g,2,d})\Big)}\]
denoted by $\omega^{1}_0$ equals to
\begin{equation}\label{degree0-2pt}
\begin{aligned}
&\sum_{k_1=0}^{l_1}\Big(k_1\psi_1+ev_1^*\bigr(c_1(N_{Q/\mathbb{P}^4})\bigr)\Big)\sum_{k_2=0}^{l_2}\Big(k_2\psi_2+ev_2^*\bigr(c_1(N_{Q/\mathbb{P}^4})\bigr)\Big)\\
+&\sum_{0\leq k_1'<k_1\leq l_1}\Big(k_1'\psi_1+ev_1^*\bigr(c_1(N_{Q/\mathbb{P}^4})\bigr)\Big)\Big(k_1\psi_1+ev_1^*\bigr(c_1(N_{Q/\mathbb{P}^4})\bigr)\Big)\\
+&\sum_{0\leq k_2'<k_2\leq l_2}\Big(k_2'\psi_2+ev_2^*\bigr(c_1(N_{Q/\mathbb{P}^4})\bigr)\Big)\Big(k_2\psi_2+ev_2^*\bigr(c_1(N_{Q/\mathbb{P}^4})\bigr)\Big)
\end{aligned}
\end{equation}
plus terms which contain at least one $\text{ch}_i(N_{g,2,d})$ (for some $i>0$).

We need to compute
\[(\pi\circ\tau_{G^l_0})_*\Big(\omega^{1}_0\cap[\overline{\mathcal{M}}_{g,1}(Q,d)]^{vir}\Big)\]
in the next.

By Lemma \ref{TGW}, we only need to consider those terms of $\omega^{1}_0$ in (\ref{degree0-2pt}). After expanding (\ref{degree0-2pt}) and collecting the same terms, it becomes
\[
\begin{aligned}
&\frac{l_1(l_1+1)(l_2+1)}{2}\psi_1ev_2^*\bigr(c_1(N_{Q/\mathbb{P}^4})\bigr)+\frac{l_2(l_2+1)(l_1+1)}{2}\psi_2ev_1^*\bigr(c_1(N_{Q/\mathbb{P}^4})\bigr)+\\
&\frac{l_1l_2(l_1+1)(l_2+1)}{4}\psi_1\psi_2+(l_1+1)(l_2+1)ev_1^*\big(c_1(N_{Q/\mathbb{P}^4})\bigr)ev_2^*\bigr(c_1(N_{Q/\mathbb{P}^4})\bigr)+\\
&\sum_{i=1}^2\Bigg(\frac{l_i(l_i^2-1)(3l_i+2)}{24}\psi_i^2+\\
&~~~~~~~~\frac{l_i^2(l_i+1)}{2}\psi_iev_i^*\bigr(c_1(N_{Q/\mathbb{P}^4})\bigr)+\frac{l_i(l_i+1)}{2}ev_i^*\bigr(c_1(N_{Q/\mathbb{P}^4})\bigr)^2\Bigg).
\end{aligned}
\]

By string, dilaton and divisor equations, it is easy to compute that
\[
\begin{aligned}
(\pi\circ\tau_{G^l_0})_*\Big(\psi_1ev_2^*\bigr(c_1(N_{Q/\mathbb{P}^4})\bigr)\cap [\overline{\mathcal{M}}_{g,2}(Q,d)]^{vir}\Big)&=&5d(2g-1)[\overline{\mathcal{M}}_{g,0}(Q,d)]^{vir},\\
(\pi\circ\tau_{G^l_0})_*\Big(\psi_2ev_1^*\bigr(c_1(N_{Q/\mathbb{P}^4})\bigr)\cap [\overline{\mathcal{M}}_{g,2}(Q,d)]^{vir}\Big)&=&5d(2g-1)[\overline{\mathcal{M}}_{g,0}(Q,d)]^{vir},\\
(\pi\circ\tau_{G^l_0})_*\Big(\psi_1\psi_2 \cap [\overline{\mathcal{M}}_{g,2}(Q,d)]^{vir}\Big)&=& (2g-1)(2g-2)[\overline{\mathcal{M}}_{g,0}(Q,d)]^{vir},\\
(\pi\circ\tau_{G^l_0})_*\Big(\prod_{i=1}^2 ev_i^*\big(c_1(N_{Q/\mathbb{P}^4})\bigr)\cap [\overline{\mathcal{M}}_{g,2}(Q,d)]^{vir}\Big)&=& 25d^2[\overline{\mathcal{M}}_{g,0}(Q,d)]^{vir},\\
(\pi\circ\tau_{G^l_0})_*\Big(\psi_i^2\cap [\overline{\mathcal{M}}_{g,2}(Q,d)]^{vir}\Big)&=& (2g-2)[\overline{\mathcal{M}}_{g,0}(Q,d)]^{vir},\\
(\pi\circ\tau_{G^l_0})_*\Big(\psi_iev_i^*\bigr(c_1(N_{Q/\mathbb{P}^4})\bigr) \cap [\overline{\mathcal{M}}_{g,2}(Q,d)]^{vir}\Big)&=& 5d[\overline{\mathcal{M}}_{g,0}(Q,d)]^{vir},\\
(\pi\circ\tau_{G^l_0})_*\Big(ev_i^*\bigr(c_1(N_{Q/\mathbb{P}^4})\bigr)^2 \cap [\overline{\mathcal{M}}_{g,2}(Q,d)]^{vir}\Big)&=& 0.
\end{aligned}
\]
After simplification, we have
\[(\pi\circ\tau_{G^l_0})_*\Big(\omega^{1}_0\cap[\overline{\mathcal{M}}_{g,1}(Q,d)]^{vir}\Big)=f_{g,d}(l_1,l_2)[\overline{\mathcal{M}}_{g,0}(Q,d)]^{vir}\]
where $f_{g,d}(l_1,l_2)$ is given by (\ref{m0}). We remark that we have used the fact that
\[l_1+l_2=5d+1-g\]
in the simplification operation.

It is clear that the non-equivariant limit of
\[\frac{1}{|Aut(G^l_0)|}(\pi\circ\tau_{G^l_0})_*\Biggr(\frac{\tau_{G^l_0}^*(\omega^{2}_T)\cap[\overline{\mathcal{M}}_{G^l_0}]^{vir}}{e_T(N_{G^l_0}^{vir})}\Biggr)\]
can be computed by replacing $l_1$ with $5d+1-g$ and $l_2$ with $0$.

So the contribution of graph $G_0^l$ in case (I) when $m=0$ to the scalar $C^{1,2}$ is
\[f_{g,d}(l_1,l_2)-f_{g,d}(5d+1-g,0).\]

The contribution of graph $G_0^l$ in case (I) when $m=1,2,3$ can be similarly derived. If we denote the contribution of graph $G_0^l$ by $C_{I}$, then we have
\begin{equation}\label{new-2}
C_{I}=
\begin{cases}
f_{g,d}(l_1,l_2)-f_{g,d}(5d+1-g,0),&\text{if}~m=0,\\
Z_{g,d}(l_1,l_2)-Z_{g,d}(5d-g,0),&\text{if}~m=1,\\
0,&\text{if}~m=2,3,\\
\end{cases}
\end{equation}
where
\[\begin{aligned}
&Z_{g,d}(l_1,l_2)=\\
&~~~5d^2(l_2+1)+\frac{l_1(l_1+1)+(2g-1)l_2(l_2+1)}{2}d-\frac{(5d+1-g)(5d-g)d}{2}.
\end{aligned}\]

(II) Next, we suppose that there exists one vertex $v_0\in V_0$ satisfying $g(v_0)=g$ and $p_*\big(\beta(v_0)\big)=d\alpha$.

Similar to the proof of Lemma \ref{typeI-nf1} case (C), we may deduce that $v_0$ is the only vertex in $V_0$, and each component mapped to $D_{\infty}$ is contractible with genus $0$, whose corresponding vertex $v_i$ connects to $v_0$ via only one edge $e_i$. The two absolute marked points must be assigned to vertices in $V_{\infty}$. We may assume that there are $s$ edges $e_1,\ldots,e_s$ with degrees $d_1,\ldots,d_s$ respectively.

Now we need to discuss according to the assignment of the two absolute marked points.

(a) Firstly, we suppose that the two absolute marked points $x_1$, $x_2$ are assigned to two different vertices in $V_{\infty}$. Without loss of generality, we assume that $x_1$ and $x_2$ are assigned to $v_1$ and $v_2$ respectively.

The evaluation maps $ev_i$ in this case factor through $D_{\infty}$, i.e.
\[ev_i:\overline{\mathcal{M}}_{G_0^l}\stackrel{ev_i^{E}}{\longrightarrow}Q\simeq D_{\infty}\hookrightarrow Y.\]
Similar to the proof of Lemma \ref{typeI-nf1} case (C), we can compute that
\begin{eqnarray*}
\tau_{G^l_0}^*(\omega_T^1)&=&(ev_1^{E})^*(H^m)\Big(t+(ev_1^{E})^*\big(c_1(N_{Q/\mathbb{P}^4})\big)\Big)^{l_1+1} \prod_{k_1=0}^{l_1}\Big(1-\frac{k_1}{d_1}\Big)\times\\
&& \Big(t+(ev_2^{E})^*\big(c_1(N_{Q/\mathbb{P}^4})\big)\Big)^{l_2+1} \prod_{k_2=0}^{l_2}\Big(1-\frac{k_2}{d_2}\Big).
\end{eqnarray*}
Obviously, it equals to $0$ unless $d_1\geq l_1+1$ and $d_2\geq l_2+1$.

By (\ref{rvlf}), the inverse of $e_T(N_{G^l_0}^{vir})$ is given by
\[\frac{\prod_{r=1}^s d_r }{-t-{\psi}_{\infty}}\prod_{r=1}^s\frac{1}{N(v_r)}\prod_{r=1}^s\frac{t+(ev_r^{E})^*\big(c_1(N_{Q/\mathbb{P}^4})\big)}{\bigr(t+(ev_r^{E})^*\big(c_1(N_{Q/\mathbb{P}^4})\big)\bigr)^{d_r}\frac{d_r!}{d_r^{d_r}}}\]
where
\[\frac{1}{N(v_r)}=\frac{1}{t+(ev_r^{E})^*(c_1(N_{Q/\mathbb{P}^4}))},~r=1,2;~~~~~\frac{1}{N(v_r)}=\frac{1}{d_r},~r\neq 1,2.\]
So $\tau_{G^l_0}^*(\omega_T^1)/e_T(N_{G^l_0}^{vir})$ equals to
\begin{equation}\label{new-1}
\frac{(ev_1^{E})^*(H^m)\omega_{1,2}\prod_{k_1=0}^{l_1}\Big(1-\frac{k_1}{d_1}\Big)\prod_{k_2=0}^{l_2}\Big(1-\frac{k_2}{d_2}\Big)}{(-t-\psi_{\infty})\frac{d_1!}{d_1^{d_1+1}}\frac{d_2!}{d_2^{d_2+1}}\prod_{r=3}^s \bigr(t+(ev_r^{E})^*\big(c_1(N_{Q/\mathbb{P}^4})\big)\bigr)^{d_r}\frac{d_r!}{d_r^{d_r}}}
\end{equation}
where
\[\omega_{1,2}=\Big(t+(ev_1^{E})^*\big(c_1(N_{Q/\mathbb{P}^4})\big)\Big)^{l_1+1-d_1}\Big(t+(ev_2^{E})^*\big(c_1(N_{Q/\mathbb{P}^4})\big)\Big)^{l_2+1-d_2}.  \]
Obviously, if $d_1\geq l_1+1$ and $d_2\geq l_2+1$, the expansion of (\ref{new-1}) according to $t$ contains only negative powers of $t$. The same is true after we replace $\omega_T^1$ by $\omega_T^2$.

So we may conclude that the contributions of those graphs $G_0^l$ in case (a) to the scalar $C^{1,2}$ are all zero.

(b) Next, we assume that the two absolute marked points are assigned to the same vertex. Without loss of generality, we assume that this vertex is $v_1$.

The component corresponding to $v_1$ is a contractible genus $0$ curve with three special points on it. One comes from the edge $e_1$ and the other two are $x_1$ and $x_2$. So $\tau_{G_0^l}^*(\psi_1)$, $\tau_{G_0^l}^*(\psi_2)$ all become zero. In this case, we also have
\[ev_1=ev_2=ev_1^{E}.\]

Now it is easy to compute that
\[\tau_{G^l_0}^*(\omega_T^1)=(ev_1^{E})^*(H^m)\Big(t+(ev_1^{E})^*\big(c_1(N_{Q/\mathbb{P}^4})\big)\Big)^{5d+3-g-m} =\tau_{G^l_0}^*(\omega_T^2).\]
So we have
\[\tau_{G^l_0}^*(\omega_T^1-\omega_T^2)=0.\]

From the discussion in both cases (a) and (b), we may conclude that the contributions of those graphs in case (II) to the scalar $C^{1,2}$ are all zero.

Combining the discussion in both cases (I) and (II), we may conclude that the scalar
\[C^{1,2}=C_{I}\]
where $C_I$ is given by (\ref{new-2}).

We are left to compute (\ref{part-2}), which is
\[\pi_*\Big(ev_1^*(H^m) \prod_{k_1=0}^{5d+1-g-m}\Big(k_1\psi_1+ev_1^*\big([D_{\infty}]\big)\Big)ev_2^*([D_{\infty}])\cap [\overline{\mathcal{M}}_{\Gamma}(Y,D_0)]^{vir}\Big).\]
By the virtual property of Lemma \ref{vpfp-2}, we know that it equals to
\[C^2[\overline{\mathcal{M}}_{g,0}(Q,d)]^{vir}\]
where $C^2$ is some scalar to be determined. So the corresponding relative Gromov-Witten invariant
\[\Big\langle\Big|\Big\{H^m\prod_{k_1=0}^{5d+1-g-m}\big(k_1\tau+[D_{\infty}]\big)\Big\}[D_{\infty}]\Big\rangle_{g,d\alpha_{\infty}}=C^2N_{g,d}.\]
Now the L.H.S can be computed by using Lemma \ref{typeI-nf1} and divisor equation. The latter can be proven by the standard cotangent line
comparison method (see \cite{MP}, Section 1.5.4 for the rubber case). For simplicity, we omit the details of computation and just list the results in the below.
\begin{equation*}
C^2=
\begin{cases}
5dX_{g,d}+5d(5d+1-g),&\text{if}~m=0,\\
5d^2,&\text{if}~m=1,\\
0, &\text{if}~m=2,3,
\end{cases}
\end{equation*}
where $X_{g,d}$ is given by (\ref{X-gd}).

Finally, Lemma \ref{typeI-nf2} can be deduced from the following equality
\[B_{g,d}(\zeta_{g,d,m,l_1,l_2}^2)=cN_{g,d}=(C^{1,2}+C^2)N_{g,d}.\]
\end{proof}

The next lemma computes one particular invariant $B_{3,d}(\zeta_{3,d})$ appeared in (\ref{coef-g3}).
\begin{lemma}\label{typeI-nf3}
\[
B_{3,d}(\zeta_{3,d})=(1875d^4+3875d^3-1950d^2-2760d+1344)N_{3,d}.
\]
\end{lemma}

\begin{proof}
We recall that
\[B_{3,d}(\zeta_{3,d})=\int_{[\overline{\mathcal{M}}_{\Gamma}(Y,D_0)]^{vir}}\omega\]
where $\Gamma=(g,d\alpha_{\infty},\emptyset,3)$ and $\omega$ is the cohomology class
\[\prod_{k_1=0}^{5d-4}\big(k_1\psi_1+ev_1^*([D_{\infty}])\big)\prod_{k_2=0}^1\big(k_2\psi_2+ev_2^*([D_{\infty}])\big)\prod_{k_3=0}^1\big(k_3\psi_3+ev_3^*([D_{\infty}])\big).\]

As before, we set
\[\pi:\overline{\mathcal{M}}_{\Gamma}(Y,D_0)\rightarrow \overline{\mathcal{M}}_{3,0}(Q,d)\]
to be the morphism which projects the curves in $Y$ down to the base and forget all the marked points. By the virtual pushforward property of Lemma \ref{vpfp-2}, we know that
\begin{equation}\label{new-3}
\pi_*\Big(\omega\cap [\overline{\mathcal{M}}_{\Gamma}(Y,D_0)]^{vir}\Big)=c[\overline{\mathcal{M}}_{3,0}(Q,d)]^{vir}
\end{equation}
where $c$ is some scalar to be determined. So
\[B_{3,d}(\zeta_{3,d})=cN_{3,d}.\]

We will use the relative virtual localization formula to compute $c$.

Similar to the proof of Lemma \ref{typeI-nf1} case (C), $\omega$ has a natural equivariant lift $\omega_T$. Now by the relative virtual localization formula (\ref{lc-0}), the L.H.S of (\ref{new-3}) equals to the non-equivariant limit of
\begin{equation}\label{lc-3pt}
\sum_{G^l_0}\frac{1}{|Aut(G^l_0)|}(\pi\circ\tau_{G^l_0})_*\Biggr(\frac{\tau_{G^l_0}^*(\omega_T)\cap[\overline{\mathcal{M}}_{G^l_0}]^{vir}}{e_T(N_{G^l_0}^{vir})}\Biggr).
\end{equation}

Similar to the proof of Lemma \ref{typeI-nf1} case (C), we only need to consider the following two types of graphs $G^l_0$.

(I) Firstly, we suppose that there exists one vertex $v_{\infty}\in V_{\infty}$ of $G^l_0$ such that $g(v_{\infty})=3$ and $\beta(v_{\infty})=d\alpha_{\infty}$. The graph $G^l_0$ is then uniquely determined. The contribution of such a graph to the scalar $c$ can be determined by string, dilaton, divisor equations and Lemma \ref{TGW}. The computation is similar to the proof of Lemma \ref{typeI-nf2} case (I). So we omit it here. The result is that the contribution of the graph in case (I) to $c$ is
\[\frac{78125d^6}{24}-\frac{3125d^5}{8}+\frac{143125d^4}{24}-\frac{61875d^3}{8}+\frac{33125d^2}{4}-7355d+2280.\]

(II) Next, we suppose that there exists one vertex $v_0\in V_0$ satisfying $g(v_0)=3$ and $p_*\big(\beta(v_0)\big)=d\alpha$.

As before, we can deduce that $v_0$ is the only vertex in $V_0$, and each component mapped to $D_{\infty}$ is contractible with genus $0$ whose corresponding vertex connects to $v_0$ via only one edge. The absolute marked points must be assigned to vertices in $V_{\infty}$.

Now we need to discuss according to the assignment of the three absolute marked points.

(a) We suppose that the three absolute marked points are assigned to three different vertices. Then similar to the discussion in Lemma \ref{typeI-nf2} case (a), we can show that the expansion of the corresponding term in (\ref{lc-3pt}) only contains the negative powers of $t$.

So the contributions of such graphs to the scalar $c$ are all $0$.

(b) We suppose that the first two absolute marked points are assigned to the same vertex while the third marked point is assigned to a different vertex. Then the total contribution of those graphs will be computed by Lemma \ref{keytrick} in the below. The result is
\[\frac{109375d^6}{24}-\frac{315625d^5}{24}+\frac{246875d^4}{24}+\frac{125d^3}{24}-\frac{12225d^2}{4}+\frac{3820d}{3}-188.\]

(c) We suppose that the first and third marked points are assigned to the same vertex while the second marked point is assigned to a different vertex. Then the total contribution of those graphs will be the same as (b) since the insertions for the second and the third marked points are the same.

(d) We suppose that the last two marked points are assigned to the same vertex while the first marked point is assigned to a different vertex.
Then the total contribution can be computed by using the same method as in the proof of Lemma \ref{keytrick}. The result is
\[-375d^3-975d^2+1605d-528.\]

(e) We suppose that all the three marked points are assigned to the same vertex such that only one edge $e$ is incident to it. Then the total contribution of such graphs can be computed as follows.

Firstly, we fix degree $d(e)=D$. The total contribution of those graphs with fixed degree $d(e)=D$ can be determined by computing
\begin{equation}\label{test}
\Big\langle\Big |\Big\{D\tau\prod_{k=0\atop k\neq D}^{5d-2}\big(k\tau+[D_{\infty}]\big)\Big\}\Big\rangle_{3,d\alpha_{\infty}}
\end{equation}
in the following two different ways.

One by Lemma \ref{typeI-nf1}, the other by using the relative virtual localization formula (\ref{lc-0}). The two ways to compute (\ref{test}) give an identity which can be used to determine the total contribution of those graphs with fixed degree $d(e)=D$. The procedure is almost the same as in the proof of Lemma \ref{keytrick}. So we omit it here.

The result is that the total contribution of those graphs with fixed degree $d(e)=D$ equals to
\[
\begin{aligned}
\frac{(-1)^{5d-1-D}D^{5d-1}}{(5d-2-D)!D!}\Big\{\frac{-625d^4+1125d^3-750d^2+160d}{2D}+~~~~~~~~\\
\frac{1875d^4-3500d^3+2475d^2-680d+64}{2}+\\
(25d^2+35d-28)D-4D^2\Big\}.
\end{aligned}
\]
Summing over $D$, we have
\[
-\frac{296875d^6}{24}+\frac{640625d^5}{24}-\frac{591875d^4}{24}+\frac{287375d^3}{24}-\frac{12575d^2}{4}+\frac{1330d}{3}-32.
\]
The final step has used the combinatorial identities for
\[\sum_{D=1}^{m}\frac{(-1)^{m-D}D^{m+k-1}}{(m-D)!D!}\]
which are given in the Appendix C, Lemma \ref{comb-idt}.

Summing over all of the contributions, we have
\[c=1875d^4+3875d^3-1950d^2-2760d+1344.\]
The lemma directly follows.
\end{proof}

\begin{lemma}\label{keytrick}
The total contribution of those graphs discussed in Lemma \ref{typeI-nf3} case (b) to the scalar $c$ is
\[\frac{109375d^6}{24}-\frac{315625d^5}{24}+\frac{246875d^4}{24}+\frac{125d^3}{24}-\frac{12225d^2}{4}+\frac{3820d}{3}-188.\]
\end{lemma}

\begin{proof}
We recall that those graphs $G_0^l$ discussed in Lemma \ref{typeI-nf3} case (b) can be described as follows. $v_0$ is the only vertex in $V_0$ satisfying $g(v_0)=3$ and $p_*(\beta(v_0))=d\alpha$. Each component mapped to $D_{\infty}$ is contractible with genus $0$ whose corresponding vertex $v_i$ connects to $v_0$ via only one edge $e_i$. The first two absolute marked points are assigned to the same vertex $v_1$ in $V_{\infty}$, and the third point is assigned to a different vertex $v_2$ in $V_{\infty}$.

We assume that there are $s$ edges $e_1,\ldots,e_s$ with degrees $d_1,\ldots,d_s$ respectively. So there are also $s$ vertices in $V_{\infty}$. The graph $G_0^l$ has been fixed now. Similar to the proof of Lemma \ref{typeI-nf1} case (C) subcase (II), we may also deduce that
\[\sum_{i=1}^sd_i=5d.\]

The corresponding contribution to the scalar $c$ can be derived from the $t^0$-part of
\begin{equation}\label{compa-1}
\frac{1}{|Aut(G^l_0)|}(\pi\circ\tau_{G^l_0})_*\Biggr(\frac{\tau_{G^l_0}^*(\omega_T)\cap[\overline{\mathcal{M}}_{G^l_0}]^{vir}}{e_T(N_{G^l_0}^{vir})}\Biggr).
\end{equation}
Here, we recall that
\[\overline{\mathcal{M}}_{G^l_0}=\overline{\mathcal{M}}_{\infty}\times_{Q^s}\overline{\mathcal{M}}_{E}\times_{Q^s} \overline{\mathcal{M}}_{G^r_0}^{\bullet,\sim}.\]
The rubber space in this case can be identified by
\[\overline{\mathcal{M}}_{G^r_0}^{\bullet,\sim}=\overline{\mathcal{M}}_{\Gamma_b}(Y,D_0\cup D_{\infty})^{\sim}\]
with data $\Gamma_b=(3,d\alpha_{\infty},\emptyset,\{d_1,d_2,\ldots,d_s\},0)$. We may abbreviate it as $\overline{\mathcal{M}}_{\Gamma_b}^{\sim}$.

Recall that
\[\overline{\mathcal{M}}_{\infty}=\prod_{v\in V_{\infty}}\overline{\mathcal{M}}_{0,val(v)}(Q,0)=\prod_{i=1}^s\overline{\mathcal{M}}_{0,val(v_i)}(Q,0).\]
Here if $val(v)\leq 2$, then $\overline{\mathcal{M}}_{0,val(v)}(Q,0)$ should be treated as $Q$.

Since $\overline{M}_{0,3}(Q,0)=Q$ and $\forall r$, $val(v_r)\leq 3$, we may deduce that all the $\overline{M}_{0,val(v_i)}(Q,0)$ can be identified as $Q$. So
\[\overline{\mathcal{M}}_{G^l_0}=\overline{\mathcal{M}}_{E}\times_{Q^s}\overline{\mathcal{M}}_{\Gamma_b}^{\sim}.\]
By Lemma \ref{rstack}, we know that
\[\overline{\mathcal{M}}_{E}=\prod_{i=1}^s\sqrt[d_i]{N_{Q/\mathbb{P}^4}/Q}\]
which is a gerbe over $Q^s$ banded by the group $\prod_{i=1}^s\boldsymbol{\mu}_{d(e_i)}$. The coarse moduli space of $\overline{\mathcal{M}}_{E}$ is $Q^s$.

So it is easy to see that $\overline{\mathcal{M}}_{G^l_0}$ and $\overline{\mathcal{M}}_{\Gamma_b}^{\sim}$ share the same coarse moduli space. As cycles in the same coarse moduli space, we have
\[[\overline{\mathcal{M}}_{G^l_0}]^{vir}=\frac{1}{\prod_{i=1}^sd_i}
[\overline{\mathcal{M}}_{\Gamma_b}^{\sim}]^{vir}.\]
Recall that
\begin{equation*}
ev^{E}=(ev_1^{E},ev_2^{E},\ldots,ev_{s}^{E}):\overline{\mathcal{M}}_{G^l_0}\longrightarrow Q^{s}
\end{equation*}
is given by (\ref{rel-eva}) and
\[\widehat{ev}=(\widehat{ev}_1,\widehat{ev}_2,\ldots,\widehat{ev}_s):\overline{\mathcal{M}}_{\Gamma_b}^{\sim}\longrightarrow Q^{s}\]
is the evaluation map given by those relative marked points which mapped to $D_{\infty}$. As maps from the same coarse moduli space, they can be naturally identified. We may also treat $\pi\circ\tau_{G^l_0}$ as a map from the coarse moduli space of $\overline{\mathcal{M}}_{\Gamma_b}^{\sim}$ to $\overline{\mathcal{M}}_{3,0}(Q,d)$.

Recall that $\omega_T$ is the equivariant lift of
\[\prod_{k_1=0}^{5d-4}\big(k_1\psi_1+ev_1^*([D_{\infty}])\big)\prod_{k_2=0}^1\big(k_2\psi_2+ev_2^*([D_{\infty}])\big)\prod_{k_3=0}^1\big(k_3\psi_3+ev_3^*([D_{\infty}])\big).\]
The restriction of $\omega_T$ i.e. $\tau_{G^l_0}^*(\omega_T)$ can be computed as follows.

Since the first two marked points are assigned to the vertex $v_1$, whose corresponding component is a contractible genus $0$ curve with three special points on it, the restrictions of $\psi_1$, $\psi_2$ all become zero. The third marked point is assigned to vertex $v_2$ which connects to $v_0$ via $e_2$.

By stability condition, the corresponding component of $v_2$ degenerates into a point. Now the $d_2$th tensor power of the cotangent line bundle $L_3$ can be identified as
\[(ev_2^{E})^*(N_{D_{\infty}/Y}^*)\simeq (ev_2^{E})^*(N_{Q/\mathbb{P}^4}^*).\]
So the restriction of $\psi_3$ becomes
\[-\frac{t+(ev_2^{E})^*\big(c_1(N_{Q/\mathbb{P}^4})\big)}{d_2}=-\frac{t+\widehat{ev}_2^*\big(c_1(N_{Q/\mathbb{P}^4})\big)}{d_2}.\]

It is also easy to see that
\[\tau_{G^l_0}^*\big(ev_1^*([D_{\infty}])\big)=\tau_{G^l_0}^*\big(ev_2^*([D_{\infty}])\big)=t+\widehat{ev}_1^*\big(c_1(N_{Q/\mathbb{P}^4})\big)\]
and
\[\tau_{G^l_0}^*\big(ev_3^*([D_{\infty}])\big)=t+\widehat{ev}_2^*\big(c_1(N_{Q/\mathbb{P}^4})\big).\]
So we have
\[\tau_{G^l_0}^*(\omega_T)=\bigr(\widehat{ev}_1^*\big(c_1(N_{Q/\mathbb{P}^4})\big)+t\bigr)^{5d-1}\prod_{k=0}^1\Big(1-\frac{k}{d_2}\Big)\bigr(\widehat{ev}_2^*\big(c_1(N_{Q/\mathbb{P}^4})\big)+t\bigr).\]

The inverse of $e_T(N_{G^l_0}^{vir})$ can be written as
\[\frac{\prod_{r=1}^s d_r }{-t-{\psi}_{\infty}}\prod_{r=1}^s\frac{1}{N(v_r)}\prod_{r=1}^s\frac{t+\widehat{ev}_r^*\big(c_1(N_{Q/\mathbb{P}^4})\big)}{\bigr(t+\widehat{ev}_r^*\big(c_1(N_{Q/\mathbb{P}^4})\big)\bigr)^{d_r}\frac{d_r!}{d_r^{d_r}}}\]
by (\ref{rvlf}).

Since $2g(v_1)-2+val(v_1)>0$, we have
\[\frac{1}{N(v_1)}=\frac{1}{e_{T}(H^0/H^1(f_{v_1}^*N_{Q/\mathbb{P}^4}))}\frac{1}{\frac{t+\widehat{ev}_1^*\big(c_1(N_{Q/\mathbb{P}^4})\big)}{d_1}-\psi_{e_1}}\]
by (\ref{rvlf-main}). In this case, $\psi_{e_1}$ is the $\psi$-class of $\overline{\mathcal{M}}_{0,3}(Q,0)$ associated to the marked point coming from edge $e_1$. So $\psi_{e_1}=0$.

As for $e_{T}(H^0/H^1(f_{v_1}^*N_{Q/\mathbb{P}^4}))$, it is easy to compute that
\[e_{T}(H^0/H^1(f_{v_1}^*N_{Q/\mathbb{P}^4}))=t+\widehat{ev}_1^*\big(c_1(N_{Q/\mathbb{P}^4})\big).\]
So we may conclude that
\[\frac{1}{N(v_1)}=\frac{d_1}{\bigr(\widehat{ev}_1^*\big(c_1(N_{Q/\mathbb{P}^4})\big)+t\bigr)^2}.\]

By (\ref{rvlf-i}) and (\ref{rvlf-ii}), we have
\[\frac{1}{N(v_2)}=\frac{1}{\widehat{ev}_2^*\big(c_1(N_{Q/\mathbb{P}^4})\big)+t},~~~\frac{1}{N(v_r)}=\frac{1}{d_r},~r\neq 1,2.\]

As for the automorphism group, we have $|Aut(G_0^l)|=|G|$, where $G$ is a group of permutation symmetries of the set $\{d_3,d_4,\ldots,d_s\}$.

From the discussion above, we may conclude that (\ref{compa-1}) equals to
\begin{equation}\label{compa-equclass1}
(\pi\circ\tau_{G^l_0})_*\Biggr(\frac{\bigr(\widehat{ev}_1^*\big(c_1(N_{Q/\mathbb{P}^4})\big)+t\bigr)^{5d-d_1-2}\cap [\overline{\mathcal{M}}_{\Gamma_b}^{\sim}]^{vir}}{(-t-\psi_{\infty})\frac{d_1!}{d_1^{d_1+1}}N_{d_2,d_3,\ldots,d_s}}\Biggr)
\end{equation}
where
\begin{equation}\label{intem}
\frac{1}{N_{d_2,d_3,\ldots,d_s}}=\frac{(d_2-1)\bigr(t+\widehat{ev}_2^*\big(c_1(N_{Q/\mathbb{P}^4})\big)\bigr)}{|G|\prod_{r=2}^s\bigr(t+\widehat{ev}_r^*\big(c_1(N_{Q/\mathbb{P}^4})\big)\bigr)^{d_r-1}\frac{d_r!}{d_r^{d_r-1}}}.
\end{equation}
Since the dimension of the $t^0$-part of
\[\frac{\bigr(\widehat{ev}_1^*\big(c_1(N_{Q/\mathbb{P}^4})\big)+t\bigr)^{5d-d_1-2}\cap [\overline{\mathcal{M}}_{\Gamma_b}^{\sim}]^{vir}}{(-t-\psi_{\infty})\frac{d_1!}{d_1^{d_1+1}}N_{d_2,d_3,\ldots,d_s}}\]
is zero, by the virtual pushforward of Lemma \ref{vpfp-2}, we may deduce that the $t^0$-part of (\ref{compa-equclass1}) can be written as
\[f(d_1,\ldots,d_s)[\overline{\mathcal{M}}_{3,0}(Q,d)]^{vir}.\]

The total contribution we want to compute in Lemma \ref{typeI-nf3} case (b) is
\[T=\sum_{s\geq 2}\sum_{\stackrel{d_1,\ldots,d_s>0}{\sum_r d_r=5d}}f(d_1,\ldots,d_s).\]

In order to compute $T$, we fix $d_1=D$ and try to compute
\[T_D=\sum_{s\geq 2}\sum_{\stackrel{d_2,\ldots,d_s>0}{\sum_{r=2}^s d_r=5d-D}}f(D,\ldots,d_s)\]
at first. Then $T$ is a sum of different $T_D$.

If $d_2=1$, then it is easy to see from (\ref{intem}) that (\ref{compa-equclass1}) equals to zero. So $f=0$. If $d_2>1$ and $D>5d-3$, then it is easy to compute that the expansion of (\ref{compa-equclass1}) according to $t$ only contains negative powers of $t$. Still we have $f=0$. So we may assume that $D\leq 5d-3$.

Inspired by \cite{Ga1}, we try to determining $T_D$ by computing the following relative Gromov-Witten invariant
\begin{equation}\label{eva-test}
\Big\langle\Big|\Big\{D\tau\prod_{k_1=0\atop k_1\neq D}^{5d-3}\big(k_1\tau+[D_{\infty}]\big)\Big\}\Big\{\prod_{k_2=0}^{1}\big(k_2\tau+[D_{\infty}]\big)\Big\}\Big\rangle_{3,d\alpha_{\infty}}
\end{equation}
in the following two different ways.

(A) The first way is to use the relative virtual localization formula (\ref{lc-0}). Let
\[\pi':\overline{\mathcal{M}}_{\Gamma'}(Y,D_0)\longrightarrow \overline{\mathcal{M}}_{3,0}(Q,d)\]
be the morphism which projects the curves in $Y$ down to the base and forget all the marked points, where the data $\Gamma'=(3,d\alpha_{\infty},\emptyset,2)$. By the virtual property of Lemma \ref{vpfp-2}, we know that
\[(\pi')_*\Big(\omega'\cap [\overline{\mathcal{M}}_{\Gamma}(Y,D_0)]^{vir}\Big)=X[\overline{\mathcal{M}}_{3,0}(Q,d)]^{vir}\]
where $X$ is some scalar to be determined and
\[\omega'=D\psi_1\prod_{k_1=0\atop k_1\neq D}^{5d-3}\big(k_1\psi_1+ev_1^*([D_{\infty}])\big)\prod_{k_2=0}^{1}\big(k_2\psi_2+ev_2^*([D_{\infty}])\big).\]

So (\ref{eva-test}) equals to $XN_{3,d}$.

Similar to the proof of Lemma \ref{typeI-nf1} case (C), $\omega'$ has a natural equivariant lift $\omega'_T$. By the relative virtual localization formula (\ref{lc-0}), we know that $X$ can be derived from the non-equivariant limit of
\begin{equation}\label{eva-lcscd}
\sum_{G^{l'}_0}\frac{1}{|Aut(G^{l'}_0)|}(\pi'\circ\tau_{G^{l'}_0})_*\Biggr(\frac{\tau_{G^{l'}_0}^*(\omega'_T)\cap[\overline{\mathcal{M}}_{G^{l'}_0}]^{vir}_T}{e_T(N_{G^{l'}_0}^{vir})}\Biggr).
\end{equation}

As before, we only need to consider the following two types of graphs $G^{l'}_0$.

(I) Firstly, we suppose that there exists one vertex $v_{\infty}\in V_{\infty}$ of $G^{l'}_0$ such that $g(v_{\infty})=3$ and $\beta(v_{\infty})=d\alpha_{\infty}$. The graph is then uniquely determined. The contribution of such a graph $G^{l'}_0$ to the scalar $X$ can be similarly determined as in the proof of Lemma \ref{typeI-nf2} case (I). So we omit the details of computation. The result is
\[D(75d^2-15d+32-4D).\]

(II) Secondly, we suppose that there exists one vertex $v_0\in V_0$ satisfying $g(v_0)=3$ and $p_*\big(\beta(v_0)\big)=d\alpha$.

As before, we can deduce that $v'_0$ is the only vertex of $V_0$, and each component mapped to $D_{\infty}$ is contractible with genus $0$ whose corresponding vertex $v'_i$ connects to $v'_0$ via only one edge $e'_i$. The two absolute marked points must be assigned to vertices in $V_{\infty}$.

Next, we need to discuss according to the assignment of the two absolute marked points.

(a) If the two absolute marked points are assigned to the same vertex $v'_1$, then the restriction of $\psi_1$ becomes a $\psi$-class of
\[\overline{\mathcal{M}}_{0,val(v'_1)}(Q,0)=\overline{\mathcal{M}}_{0,3}(Q,0)\]
which vanishes obviously.

Now since $\tau_{G^{l'}_0}^*(\omega'_T)$ contains one factor $\tau_{G^{l'}_0}^*(D\psi_1)$, it vanishes as well.

So we may conclude that the contributions of graphs in case (a) are all $0$.

(b) We assume that the two absolute marked points are assigned to two different vertices $v'_1,v'_2$. We further assume that there are $s$ edges $e'_1,\ldots,e'_s$ with degrees $d_1,\ldots,d_s$ respectively. The graph $G_0^{l'}$ is fixed now. The corresponding contribution can be derived from the $t^0$-part of
\begin{equation}\label{compa-2}
\sum_{G^{l'}_0}\frac{1}{|Aut(G^{l'}_0)|}(\pi'\circ\tau_{G^{l'}_0})_*\Biggr(\frac{\tau_{G^{l'}_0}^*(\omega'_T)\cap[\overline{\mathcal{M}}_{G^{l'}_0}]^{vir}_T}{e_T(N_{G^{l'}_0}^{vir})}\Biggr).
\end{equation}

Similar to the discussion of (\ref{compa-1}), we can deduce that (\ref{compa-2}) equals to
\begin{equation}\label{compa-equclass2}
S\times(\pi'\circ\tau_{G^{l'}_0})_*\Biggr(\frac{\bigr(\widehat{ev}_1^*\big(c_1(N_{Q/\mathbb{P}^4})\big)+t\bigr)^{5d-d_1-2}\cap [\overline{\mathcal{M}}_{\Gamma_b}^{\sim}]^{vir}}{(-t-\psi_{\infty})\frac{d_1!}{d_1^{d_1}}N_{d_2,d_3,\ldots,d_s}}\Biggr)
\end{equation}
where
\[S=-\frac{D}{d_1}\prod_{\stackrel{k_1=0}{k_1\neq D}}^{5d-3}\Big(1-\frac{k_1}{d_1}\Big)\]
and $N_{d_2,d_3,\ldots,d_s}$ is the same as (\ref{intem}).

By the virtual property of Lemma \ref{vpfp-2}, we know that $t^0$-part of (\ref{compa-equclass2}) can be written as
\[h(d_1,\ldots,d_s)[\overline{\mathcal{M}}_{3,0}(Q,d)]^{vir}.\]

By comparing (\ref{compa-equclass1}) with (\ref{compa-equclass2}), it is easy to see that
\[h(d_1,\ldots,d_s)=-\frac{D}{d_1^2}\prod_{\stackrel{k_1=0}{k_1\neq D}}^{5d-3}\Big(1-\frac{k_1}{d_1}\Big)f(d_1,\ldots,d_s).\]

Using the above equation, it is easy to check that if $d_1\neq D$, then $h=0$. If $d_1=D$, then we have
\[h(D,d_2,\ldots,d_s)=\frac{(5d-3-D)!D!}{D^{5d-2}}(-1)^{5d-2-D}f(D,d_2,\ldots,d_s).\]

So the total contribution of those graphs in case (b) to the scalar $X$ is
\[\sum_{s\geq 2}\sum_{\stackrel{d_2,\ldots,d_s>0}{\sum_{r=2}^s d_r=5d-D}}h(D,d_2,\ldots,d_s)=\frac{(5d-3-D)!D!}{D^{5d-2}}(-1)^{5d-2-D}T_D.\]

From the discuss in both case (I) and case (II), we may conclude that
\begin{equation}\label{new-4}
X=D(75d^2-15d+32-4D)+\frac{(5d-3-D)!D!}{D^{5d-2}}(-1)^{5d-2-D}T_D.
\end{equation}

(B) There is another way to compute (\ref{eva-test}). We firstly express
\[D\tau\prod_{k_1=0\atop k_1\neq D}^{5d-3}\big(k_1\tau+[D_{\infty}]\big)\]
in terms of
\[H^m\prod_{k_1=0}^{5d-3-m}\big(k_1\tau+[D_{\infty}]\big)\]
where $0\leq m\leq 3$. Since
\[D\tau\prod_{k_1=0\atop k_1\neq D}^{5d-3}\big(k_1\tau+[D_{\infty}]\big)=\prod_{k_1=0}^{5d-3}\big(k_1\tau+[D_{\infty}]\big)-[D_{\infty}]\prod_{k_1=0\atop k_1\neq D}^{5d-3}\big(k_1\tau+[D_{\infty}]\big)\]
and
\begin{eqnarray*}
&&\prod_{k_1=0\atop k_1\neq D}^{5d-3}\big(k_1\tau+[D_{\infty}]\big)\\
&=&\prod_{k_1=0\atop k_1\neq D}^{5d-4}\big(k_1\tau+[D_{\infty}]\big)\Big(\big(D\tau+[D_{\infty}]\big)+(5d-3-D)\tau\Big)\\
&=&\prod_{k_1=0}^{5d-4}\big(k_1\tau+[D_{\infty}]\big)+\frac{5d-3-D}{D}\Bigg\{D\tau\prod_{k_1=0\atop k_1\neq D}^{5d-4}\big(k_1\tau+[D_{\infty}]\big)\Bigg\},\\
\end{eqnarray*}
we then apply the same procedure for
\[D\tau\prod_{k_1=0\atop k_1\neq D}^{5d-4}\big(k_1\tau+[D_{\infty}]\big).\]

Since $[D_{\infty}]^{m+1}=(5H)^m[D_{\infty}]=0$ for $m>3$, the procedure stops after finite steps. Then we have
\[\begin{aligned}
&D\tau\prod_{k_1=0\atop k_1\neq D}^{5d-3}\big(k_1\tau+[D_{\infty}]\big)=\prod_{k_1=0}^{5d-3}\big(k_1\tau+[D_{\infty}]\big)-\\
&~~~~~\frac{5d-3}{D}[D_{\infty}]\prod_{k_1=0}^{5d-4}\big(k_1\tau+[D_{\infty}]\big)+\\
&~~~~~~~\frac{(5d-4)(5d-3-D)}{D^2}[D_{\infty}]^2\prod_{k_1=0}^{5d-5}\big(k_1\tau+[D_{\infty}]\big)-\\
&~~~~~~~~~\frac{(5d-5)(5d-3-D)(5d-4-D)}{D^3}[D_{\infty}]^3\prod_{k_1=0}^{5d-6}\big(k_1\tau+[D_{\infty}]\big).
\end{aligned}\]
Since
\begin{eqnarray*}
[D_{\infty}]^m\prod_{k_1=0}^{5d-3-m}\big(k_1\tau+[D_{\infty}]\big)&=& [D_{\infty}]^{m+1}\prod_{k_1=1}^{5d-3-m}\big(k_1\tau+[D_{\infty}]\big)\\
&=&(5H)^m[D_{\infty}]\prod_{k_1=1}^{5d-3-m}\big(k_1\tau+[D_{\infty}]\big)\\
&=&(5H)^m\prod_{k_1=0}^{5d-3-m}\big(k_1\tau+[D_{\infty}]\big),\\
\end{eqnarray*}
we may further deduce that
\[\begin{aligned}
&D\tau\prod_{k_1=0\atop k_1\neq D}^{5d-3}\big(k_1\tau+[D_{\infty}]\big)=\prod_{k_1=0}^{5d-3}\big(k_1\tau+[D_{\infty}]\big)-\\
&~~~~~\frac{5d-3}{D}5H\prod_{k_1=0}^{5d-4}\big(k_1\tau+[D_{\infty}]\big)+\\
&~~~~~~~\frac{(5d-4)(5d-3-D)}{D^2}(5H)^2\prod_{k_1=0}^{5d-5}\big(k_1\tau+[D_{\infty}]\big)-\\
&~~~~~~~~~\frac{(5d-5)(5d-3-D)(5d-4-D)}{D^3}(5H)^3\prod_{k_1=0}^{5d-6}\big(k_1\tau+[D_{\infty}]\big).
\end{aligned}\]

So we may compute (\ref{eva-test}) by Lemma \ref{typeI-nf2}. The result is
\[\Big\{375d^3+225d^2-355d+84-\frac{(5d-3)5d(5d+8)}{D}\Big\}N_{3,d}.\]
It further implies that
\begin{equation}\label{new-5}
X=375d^3+225d^2-355d+84-\frac{(5d-3)5d(5d+8)}{D}.
\end{equation}

Now comparing (\ref{new-4}) with (\ref{new-5}), we get the identity
\[
\begin{aligned}
D(75d^2-15d+32-4D)+\frac{(5d-3-D)!D!}{D^{5d-2}}(-1)^{5d-2-D}T_D\\
=375d^3+225d^2-355d+84-\frac{(5d-3)5d(5d+8)}{D}.
\end{aligned}
\]
So
\[
\begin{aligned}
T_D=&\frac{(-1)^{5d-2-D}D^{5d-2}}{(5d-3-D)!D!}\times \Big\{375d^3+225d^2-355d+84\\
&-\frac{(5d-3)5d(5d+8)}{D}-D(75d^2-15d+32)+4D^2\Big\}.
\end{aligned}
\]

Summing over $D$, we have
\[
\begin{aligned}
T=\sum_{D=1}^{5d-3}T_D=\frac{109375d^6}{24}-\frac{315625d^5}{24}+\frac{246875d^4}{24}+\\
\frac{125d^3}{24}-\frac{12225d^2}{4}+\frac{3820d}{3}-188.
\end{aligned}
\]
Here, we have used the combinatorial identities for
\[\sum_{D=1}^{m}\frac{(-1)^{m-D}D^{m+k-1}}{(m-D)!D!}\]
which are given in the Appendix C, Lemma \ref{comb-idt}. The proof for Lemma \ref{keytrick} is complete.
\end{proof}

\section{Further discussion}\label{section-4}
In order to prove Conjecture \ref{conj} for all $g$. We may try to generalize Theorem \ref{keythm2} for all pairs $(g,d)$ with $d>0$. This is impossible. The reason is as follows.

It is easy to check that for $\zeta_{g,d}\in S_{g,d}'$, those constants $C_{\rho}$ in Theorem \ref{keythm2} can be solved by
\[
C_{\rho}=
\begin{cases}
1,&~ \text{if} ~\rho=\zeta_{g,d},\\
0,&~\text{otherwise}.
\end{cases}
\]

So we have
\[B_{g,d}(\zeta_{g,d})-\sum_{\rho\in S_{g,d}'}B_{g,d}(\rho)C_{\rho}=0.\]

In order to get $\zeta_{g,d}\in S_{g,d}$ such that
\[B_{g,d}(\zeta_{g,d})-\sum_{\rho\in S_{g,d}'}B_{g,d}(\rho)C_{\rho}\]
equals to $C_{g,d}N_{g,d}$ with $C_{g,d}\neq 0$, we need
\[\zeta_{g,d}=\{(l_i,m_i)\}_{i=1}^s\in S_{g,d}\backslash S_{g,d}'.\]
This requires that
\[\sum_{i=1}^s(l_i+1)\geq 5d+1,~~\sum_{i=1}^s(l_i+m_i)=5d+1-g.\]
We may deduce that
\[s-\sum_{i=1}^s m_i-g=\sum_{i=1}^s(l_i+1)- (5d+1)\geq 0.\]
So $s\geq g$. Since $l_i\geq 0$, $m_i\geq 0$ and $(l_i,m_i)\neq(0,0)$, we have $l_i+m_i\geq 1$. So
\[5d+1-g=\sum_{i=1}^s(l_i+m_i)\geq s\geq g.\]
It implies that
$$d\geq \frac{2g-1}{5}.$$
If $g\geq4$, we have $d>1$. So we can not generalize Theorem \ref{keythm2} to pair $(g,1)$ with $g\geq 4$. The right way to generalize Theorem \ref{keythm2} should be
\begin{conjecture}\label{conj-1}
For each pair $(g,d)$ with $d\geq \frac{2g-1}{5}$, there always exists $\zeta_{g,d}\in S_{g,d}$ such that
\[B_{g,d}(\zeta_{g,d})-\sum_{\rho\in S_{g,d}'}B_{g,d}(\rho)C_{\rho}\]
equals to $C_{g,d}N_{g,d}$ with $C_{g,d}\neq 0$, where those $C_{\rho}$ are determined by Theorem \ref{keythm}.
\end{conjecture}

If this conjecture is true and we further assume that $N_{g,d}$ are known for $0<d<\frac{2g-1}{5}$, then we can recursively determine all $N_{g,d}$.

\section{Appendix A}
In this appendix, we give a short proof of Theorem \ref{keythm2} for $g=2$.

We choose $\zeta_{2,d}=\{(5d-2,0),(1,0)\}$ and set $\zeta_m=\{(5d-1-m,m)\}$. Applying Theorem \ref{keythm}, we have
\[
\begin{aligned}
A_{2,d}(\zeta_{2,d})-&\sum_{m=0}^3 5^m(5d-m)A_{2,d}(\zeta_m)=\\
&B_{2,d}(\zeta_{2,d})-\sum_{m=0}^3 5^m(5d-m)B_{2,d}(\zeta_m)+NPT.
\end{aligned}
\]

We need to show that
\begin{equation}\label{coef-g2}
B_{2,d}(\zeta_{2,d})-\sum_{m=0}^3 5^m(5d-m)B_{2,d}(\zeta_m)
\end{equation}
can be written as $C_{2,d}N_{2,d}$ with $C_{2,d}\neq 0$.

By Lemma \ref{typeI-nf2}, we have
\[B_{2,d}(\zeta_{2,d})=\{4(5d -1)(5d- 2) + 5d(50d^2 - 5d) + 5d(5d - 1)\}N_{2,d}.\]
By Lemma \ref{typeI-nf1}, we have
\[
B_{2,d}(\zeta_m)=
\begin{cases}
(50d^2 - 5d)N_{2,d}, &\text{if}~m=0,\\
dN_{2,d}, &\text{if}~m=1,\\
0, &\text{if}~m=2,3.
\end{cases}
\]
So (\ref{coef-g2}) becomes
\[4(5d -1)(5d- 2)N_{2,d}.\]
Obviously, $4(5d -1)(5d- 2)\neq 0$. The proof for Theorem \ref{keythm2} with $g=2$ is complete.

\section{Appendix B}
We will apply relative virtual localization formula \cite{GrVa} to moduli space $\overline{\mathcal{M}}_{\Gamma}(Y,D_0)$. The form of presentation is similar to that of \cite{FP2} where they gave a description of relative virtual localization formula in the case of $\mathbb{P}^1$.

There is a natural $\mathbb{C}^*$-action on the bundle $N_{W/V}\oplus \mathcal{O}_W$, which is given by scaling on the second factor $\mathcal{O}_W$. It induces a natural $\mathbb{C}^*$-action on $Y=\mathbb{P}(N_{W/V}\oplus \mathcal{O}_W)$.

The target for a general stable relative map in $\overline{\mathcal{M}}_{\Gamma}(Y,D_0)$ can be written as $Y_s\cup Y$, where $D_{\infty}$ of $Y_s$ is glued to $D_0$ of $Y$. $\mathbb{C}^*$ acts on $Y$ and leaves $Y_s$ invariant. The $\mathbb{C}^*$-action on $\overline{\mathcal{M}}_{\Gamma}(Y,D_0)$ is given by composition.

Let
\[c_s: Y_s\cup Y\longrightarrow Y\]
be the natural contraction to the last copy of $Y$. For a $\mathbb{C}^*$-fixed point
\begin{equation}\label{fp}
[(\Sigma,x,y,f,Y_s\cup Y)]
\end{equation}
in $\overline{\mathcal{M}}_{\Gamma}(Y,D_0)$, the image of a connected component of $\Sigma$ under the composite map $c_s\circ f$  will either completely sit in the fixed points set $D_0\bigsqcup D_{\infty}$ or in a fiber of $Y$. Any irreducible component $\Sigma_j$ of $\Sigma$ which sits in a fiber of $Y$ must be a rational sphere.  We suppose that the map $c_s\circ f$ restricted to $\Sigma_j$ is a degree $\nu_j$ cover of a fiber.

Recall that $\Gamma=(g,\beta,\{\mu_i\}_{i=1}^n,m)$. We associate a localization graph $G^l_0$ to the fixed point (\ref{fp}). It consists of
\begin{itemize}
\item[(i)] A set $S$ of vertices $v_1,v_2,\ldots,v_{|S|}$ corresponding to connected components of $(c_s\circ f)^{-1}(D_0\bigsqcup D_{\infty})$. The images of the connected components naturally give an assignment
    \[\pi:S \rightarrow \{D_0,D_{\infty}\}.\]
\item[(ii)] An assignment of genera $g:S
\rightarrow \mathbb{Z}_{\geq 0}$ (if the connected component is a point, we take it to be 0), and an assignment of curve classes $\beta:S\rightarrow H_2(Y,\mathbb{Z})$ given by push-forward under the map $c_s\circ f$.
\item[(iii)] An assignment of absolute marked points $a:\{1,\ldots,m\}\rightarrow S$, an assignment of boundary marked points $b_0:\{1,\ldots,n\}\rightarrow \{v\in S:\pi(v)=D_{0}\}$, and an assignment of contact orders to the boundary marked points $\mu_0:\{1,\ldots,n\}\rightarrow \mathbb{Z}_{>0}$ such that $\mu_0(j)=\mu_j$.
\item[(iv)] A set $E$ of edges $e_1,e_2,\ldots,e_{|E|}$ corresponding to the irreducible components of $\Sigma$ mapped to a fiber of $Y$ under $c_s\circ f$. Edge $e_r$ is incident to a vertex $v_l$ if the corresponding two components intersect. We require that $(S,E)$ forms a connected graph.
\item[(v)] An assignment of degrees $d:E\rightarrow \mathbb{Z}_{>0}$ given by degrees of the covers.
\end{itemize}

The fixed points of $\overline{\mathcal{M}}_{\Gamma}(Y,D_0)$ with the same localization graph $G^l_0$ form a connected component of the fixed loci. We denote it as $\overline{\mathcal{F}}_{G^l_0}$.

The valence $val(v)$ of a vertex $v$ is defined to be the number of all marked points and edges associated to $v$. The set of those vertices which satisfy $\pi(v)=D_{\infty}$ (resp. $D_{0}$) is denoted as $V_{\infty}$ (resp. $V_0$).

Let $v\in V_{\infty}$. The restriction of relative map $f$ to the corresponding connected component of $v$ can be identified as a stable map in
$$\overline{\mathcal{M}}_{g(v),val(v)}\bigr(W,p_*\big(\beta(v)\big)\bigr).$$
We set
\[\overline{\mathcal{M}}_{\infty}=\prod_{v\in V_{\infty}} \overline{\mathcal{M}}_{g(v),val(v)}\bigr(W,p_*\big(\beta(v)\big)\bigr).\]
The possible unstable moduli spaces $\overline{\mathcal{M}}_{0,1}\bigr(W,0)$ or $\overline{\mathcal{M}}_{0,2}\bigr(W,0)$ in $\overline{\mathcal{M}}_{\infty}$ should be treated as $W$.

For $e_r\in E$, we set $\overline{\mathcal{M}}_{e_r}$ to be the stack parametrizing those maps from $\mathbb{P}^1$ to $Y$ which are $\mathbb{C}^*$-invariant ($\mathbb{C}^*$-action induced from $\mathbb{C}^*$-action on $Y$) and degree $d(e_r)$ cover of a fiber. For $f_{e_r},f'_{e_r}\in \overline{M}_{e_r}$, an arrow from $f_{e_r}$ to $f'_{e_r}$ consists of an isomorphism $\psi:\mathbb{P}^1\rightarrow \mathbb{P}^1$ such that $f'_{e_r}\circ \psi=f_{e_r}$. We have
\begin{lemma}\label{rstack}
\[\overline{\mathcal{M}}_{e_r}\simeq \sqrt[d(e_r)]{N_{W/V}/W},\]
where $\sqrt[d(e_r)]{N_{W/V}/W}$ is the stack over $W$ of $d(e_r)$th roots of $N_{W/V}$ (for the definition of $\sqrt[d(e_r)]{N_{W/V}/W}$, see \cite{AGV} Appendix B).
\end{lemma}
\begin{proof}
The stack $\overline{\mathcal{M}}_{e_r}$ is a category whose objects over a $\mathbb{C}$-scheme $T$ are pairs $(\mathcal{C},f)$, where $\mathcal{C}$ is a flat family of smooth rational curves over $T$, and $f$ is a morphism from $\mathcal{C}$ to $Y$ such that when restricted to each fiber of $\mathcal{C}$, $f$ is $\mathbb{C}^*$-invariant and degree $d(e_r)$ cover of a fiber of $Y$.

Since $f$ maps fibers of $\mathcal{C}$ to fibers of $Y$, it induces a morphism $\varphi$ between bases $T$ and $W$. The pullbacks of two $\mathbb{C}^*$-invariant divisors $D_0$ and $D_{\infty}$ via $f$ naturally give two separate sections $\sigma_0,\sigma_{\infty}: T\rightarrow\mathcal{C}$. According to \cite{JI}, Section 1.1.1, it implies that $\mathcal{C}\simeq \mathbb{P}(\mathcal{L}\oplus \mathcal{O}_{T})$ for some line bundle $\mathcal{L}$ on $T$. We may choose $\mathcal{L}$ such that $\sigma_0,\sigma_{\infty}$ are determined by the factors $\mathcal{L},
\mathcal{O}_T$ respectively.

Since when restricted to each fiber of $\mathbb{P}(\mathcal{L}\oplus \mathcal{O}_{T})$, $f$ is $\mathbb{C}^*$-invariant and degree $d(e_r)$ cover of a fiber of $\mathbb{P}(N_{W/V}\oplus \mathcal{O}_W)$ , it naturally induces an isomorphism $\bar{f}: \mathcal{L}^{\otimes d(e_r)}\simeq\varphi^*(N_{W/V})$.

The triple $(\varphi,\mathcal{L},\bar{f})$ gives an object of $\sqrt[d(e_r)]{N_{W/V}/W}$. It naturally induce a functor $\mathcal{F}:\overline{\mathcal{M}}_{e_r}\rightarrow \sqrt[d(e_r)]{N_{W/V}/W}$. It is easy to check that $\mathcal{F}$ gives an isomorphism between $\overline{\mathcal{M}}_{e_r}$ and $\sqrt[d(e_r)]{N_{W/V}/W}$.
\end{proof}

Let $\boldsymbol{\mu}_d$ be a subgroup of $\mathbb{G}_{m}$ consisting of $d$-roots of unity.

The stack $\sqrt[d(e_r)]{N_{W/V}/W}$ is a gerbe over $W$ banded by $\boldsymbol{\mu}_{d(e_r)}$. Locally, $\sqrt[d(e_r)]{N_{W/V}/W}$ is a quotient of $W$ by the trivial action of $\boldsymbol{\mu}_{d(e_r)}$, but it is not true globally (see \cite{AGV}, Appendix B). The coarse moduli space of $\sqrt[d(e_r)]{N_{W/V}/W}$ is $W$.

Next, we set
\[\overline{\mathcal{M}}_{E}:=\prod_{e_r\in E}\overline{\mathcal{M}}_{e_r}.\]
The coarse moduli space for $\overline{\mathcal{M}}_{E}$ is $W^{|E|}$. So there is a natural morphism
\[\gamma:\overline{\mathcal{M}}_{E}\longrightarrow W^{|E|}.\]

The evaluation at those marked points coming from edges gives a natural map
\[ev_{\infty}:\overline{\mathcal{M}}_{\infty}\longrightarrow W^{|E|}.\]

The fiber product
\[\overline{\mathcal{M}}_1:=\overline{\mathcal{M}}_{\infty}\times_{W^{|E|}} \overline{\mathcal{M}}_{E}\]
can be seen as a stack parametrizing $\mathbb{C}^*$-invariant (possibly disconnected) stable maps from curves to $Y$ with data of stable maps inherited from $G_0^l$.

The restriction of map $f$ in (\ref{fp}) to those components corresponding to vertices in $V_0$, can be seen as a rubber map in $\overline{\mathcal{M}}_{G^r_0}(Y,D_0\cup D_{\infty})^{\bullet,\sim}$, where $G^r_0$ is a relative graph induced from $G_0^l$. We may abbreviate $\overline{\mathcal{M}}_{G^r_0}(Y,D_0\cup D_{\infty})^{\bullet,\sim}$ as $\overline{\mathcal{M}}_{G^r_0}^{\bullet,\sim}$.

The evaluation at those marked points coming from edges also gives a natural map from $\overline{\mathcal{M}}_{G^r_0}^{\bullet,\sim}$ to $W^{|E|}$. We use $\overline{\mathcal{M}}_{G^l_0}$ to denote the fiber product
\[\overline{\mathcal{M}}_{1}\times_{W^{|E|}} \overline{\mathcal{M}}_{G^r_0}^{\bullet,\sim}.\]

There is a natural map
\begin{equation}\label{rel-eva}
ev^{E}=(ev_1^{E},ev_2^{E},\ldots,ev_{|E|}^{E}):\overline{\mathcal{M}}_{G^l_0}\longrightarrow W^{|E|}.
\end{equation}

Suppose that $V_0$ is empty. Since each edge connects some vertex of $V_0$ to that of $V_{\infty}$, we know that $E$ is also empty. Then $\overline{\mathcal{M}}_{G^l_0}$ simply becomes $\overline{\mathcal{M}}_{\infty}$.

The automorphism group $Aut(G^l_0)$ of a localization graph $G^l_0$ consists of those automorphisms of the graph $G^l_0$ which leave all the assignments invariant. $Aut(G^l_0)$ naturally acts on $\overline{\mathcal{M}}_{G^l_0}$.

Now it is easy to see that the fixed locus $\overline{\mathcal{F}}_{G^l_0}$ is simply a quotient of $\overline{\mathcal{M}}_{G^l_0}$ by the automorphism group $Aut(G^l_0)$. We denote the quotient map as
\[\tau_{G^l_0}:\overline{\mathcal{M}}_{G^l_0}\rightarrow \overline{\mathcal{F}}_{G^l_0}.\]

The virtual fundamental class of $\overline{\mathcal{M}}_{G^l_0}$ is given by
\begin{equation*}
[\overline{\mathcal{M}}_{G^l_0}]^{vir}=\triangle^{!}\big([\overline{\mathcal{M}}_{1}]^{vir}\times [\overline{\mathcal{M}}_{G^r_0}^{\bullet,\sim}]^{vir}\big)
\end{equation*}
where $\triangle:W^{|E|}\rightarrow W^{|E|}\times W^{|E|}$ is the diagonal map, and $[\overline{\mathcal{M}}_{1}]^{vir}$, $[\overline{\mathcal{M}}_{G^r_0}^{\bullet,\sim}]^{vir}$ are induced by the $\mathbb{C}^*$-fixed part of the pullback of the obstruction theory of $\overline{\mathcal{M}}_{\Gamma}(Y,D_0)$ (see \cite{GrVa}, Section 3.2).

Since $\overline{\mathcal{M}}_{E}$ is a $\prod_{e_r\in E}d(e_r)$-gerbe over $W^{|E|}$, we know that $\overline{\mathcal{M}}_{1}$ is a $\prod_{e_r\in E}d(e_r)$-gerbe over $\overline{\mathcal{M}}_{\infty}$. So as classes in the same coarse moduli space, we have
\[[\overline{\mathcal{M}}_1]^{vir}=\frac{1}{\prod_{e_r\in E}d(e_r)}[\overline{\mathcal{M}}_{\infty}]^{vir},\]
where $[\overline{\mathcal{M}}_{\infty}]^{vir}$ is the usual virtual fundamental class of $\overline{\mathcal{M}}_{\infty}$ treated as moduli space of (possibly disconnected) stable maps.

Now we have
\begin{equation}\label{virclass}
[\overline{\mathcal{M}}_{G^l_0}]^{vir}=\frac{1}{\prod_{e_r\in E}d(e_r)}\triangle^{!}\big([\overline{\mathcal{M}}_{\infty}]^{vir}\times [\overline{\mathcal{M}}_{G^r_0}^{\bullet,\sim}]^{vir}\big).
\end{equation}

When the set $V_0$ is empty, we have $\overline{\mathcal{M}}_{G^l_0}=\overline{\mathcal{M}}_{\infty}$. So in that case $[\overline{\mathcal{M}}_{G^l_0}]^{vir}=[\overline{\mathcal{M}}_{\infty}]^{vir}$.

With some abuse of notations, the pullback of virtual normal bundle $N_{G^l_0}^{vir}$ of $\overline{\mathcal{F}}_{G^l_0}$ to $\overline{\mathcal{M}}_{G^l_0}$ will still be denoted as $N_{G^l_0}^{vir}$. The pullback of class $\psi_{\infty}$ in $\overline{\mathcal{M}}_{G^r_0}^{\bullet,\sim}$ to $\overline{\mathcal{M}}_{G^l_0}$ will also be denoted as $\psi_{\infty}$.

The equivariant Euler class of virtual normal bundle $e_T(N_{G^l_0}^{vir})$ is important in localization formula. Gathmann has given a description of $e_T(N_{G^l_0}^{vir})$ in \cite{Ga1}, Section 5.2. We will give a description of the inverse of $e_T(N_{G^l_0}^{vir})$ according to \cite{GrVa}.

If the target for any map in $\overline{\mathcal{F}}_{G^l_0}$ is $Y$, we call it simple fixed locus. Otherwise, we call it composite fixed locus.

We recall that $t$ is the generator of $H^*_{\mathbb{C}^*}(pt)$ corresponding to the dual of the standard representation of $\mathbb{C}^*$. In other words, $t$ is the hyperplane class of $\mathbb{C}\mathbb{P}^{\infty}$.

Let us consider composite fixed locus at first.

The contribution of each edge $e_r$ to the inverse of $e_T(N_{G^l_0}^{vir})$ is given by
\[\frac{t+(ev_r^{E})^*\big(c_1(N_{W/V})\big)}{\bigr(t+(ev_r^{E})^*\big(c_1(N_{W/V})\big)\bigr)^{d(e_r)}\frac{d(e_r)!}{d(e_r)^{d(e_r)}}}.\]

Let $v\in V_{\infty}$. We assume that edges $e_{l_1},e_{l_2},\ldots,e_{l_s}$ are incident to $v$.

If $2g(v)-2+val(v)> 0$ or $p_*(\beta(v))\neq 0$, then the contribution of vertex $v$ which is denoted by $1/N(v)$ becomes
\begin{equation}\label{rvlf-main}
\frac{1}{e_{T}(H^0/H^1(f_v^*N_{W/V}))}\prod_{i=1}^s\frac{1}{\frac{t+(ev_{l_i}^{E})^*\big(c_1(N_{W/V})\big)}{d(e_{l_i})}-\psi_{e_{l_i}}},
\end{equation}
where $\psi_{e_{l_i}}$ is the $\psi$-class of
$$\overline{\mathcal{M}}_{g(v),val(v)}\bigr(W,p_*\big(\beta(v)\big)\bigr)$$
associated to the marked point coming from edge $e_{l_i}$, $f_v$ is the restriction of $f$ to the component $\Sigma_v$ corresponding to $v$, and $H^0/H^1(f_v^*N_{W/V})$ is the virtual vector bundle
$$H^0(\Sigma_v,f_v^*N_{W/V})\ominus H^1(\Sigma_v,f_v^*N_{W/V}),$$
see \cite{CG} for more details of virtual vector bundle.

If $2g(v)-2+val(v)\leq 0$ and $p_*(\beta(v))=0$, then the component corresponding to $v$ must be an isolated point $q$ by stable condition of relative maps. There are three cases.
\begin{itemize}
\item[i)] If $val(v)=1$, then only one edge $e_{l_1}$ is incident to $v$. The contribution of vertex $v$ is
\begin{equation}\label{rvlf-i}
\frac{1}{d(e_{l_1})}.
\end{equation}
\item[ii)] If $val(v)=2$ and $q$ is a marked point, then still only one edge $e_{l_1}$ is incident to $v$. The contribution is
\begin{equation}\label{rvlf-ii}
\frac{1}{t+(ev_{l_1}^{E})^*(c_1(N_{W/V}))}.
\end{equation}
\item[iii)] If $val(v)=2$ and two edges $e_{l_1}$ and $e_{l_2}$ are incident to $v$, the contribution is
\[\frac{1}{t+(ev_x^{E})^*\big(c_1(N_{W/V})\big)}\frac{1}{\frac{t+(ev_{l_1}^{E})^*\big(c_1(N_{W/V})\big)}{d(e_{l_1})}+\frac{t+(ev_{l_2}^{E})^*\big(c_1(N_{W/V})\big)}{d(e_{l_2})}}.\]
\end{itemize}
Here, $x$ could either be $l_1$ or $l_2$. Since $ev_{l_1}^{E}=ev_{l_2}^{E}$, it is well defined.

There is also one contribution from deforming the target singularity, which is given by
\[\frac{\prod_{e_r\in E}d(e_r)}{-t-{\psi}_{\infty}}.\]

Since $\mathbb{C}^*$ acts only on the second factor of $Y_s\cup Y$, there will be no contributions from vertices of $V_0$.

So the inverse of equivariant Euler class of $N_{G^l_0}^{vir}$ is given by
\begin{equation}\label{rvlf}
\frac{\prod_{e_r\in E}d(e_r)}{-t-{\psi}_{\infty}}\prod_{v\in V_{\infty}}\frac{1}{N(v)}\prod_{e_r\in E}\frac{t+(ev_r^{E})^*\big(c_1(N_{W/V})\big)}{\bigr(t+(ev_r^{E})^*\big(c_1(N_{W/V})\big)\bigr)^{d(e_r)}\frac{d(e_r)!}{d(e_r)^{d(e_r)}}}.
\end{equation}

In the situation of simple fixed locus, $\overline{\mathcal{M}}_{G^r_0}^{\bullet,\sim}$ degenerates into $W^{|E|}$. We have
$$\overline{\mathcal{M}}_{G^l_0}=\overline{\mathcal{M}}_{1}\times_{W^{|E|}} \overline{\mathcal{M}}_{G^r_0}^{\bullet,\sim}=\overline{\mathcal{M}}_{1}.$$
The connectedness condition implies that there is only one vertex $v$ in $V_{\infty}$.

The inverse of $e_T(N_{G^l_0}^{vir})$ becomes
\begin{equation}\label{el-simpls}
\frac{1}{N(v)}\prod_{e_r\in E}\frac{t+(ev_r^{E})^*\big(c_1(N_{W/V})\big)}{\bigr(t+(ev_r^{E})^*\big(c_1(N_{W/V})\big)\bigr)^{d(e_r)}\frac{d(e_r)!}{d(e_r)^{d(e_r)}}}.
\end{equation}
Here, $\frac{1}{N(v)}$ can be determined as in the case of composite fixed locus.

The relative virtual localization formula expresses the equivariant virtual fundamental class of $\overline{\mathcal{M}}_{\Gamma}(Y,D_0)$ in terms of contribution from each localization graph $G^l_0$:
\begin{equation}\label{lc-0}
[\overline{\mathcal{M}}_{\Gamma}(Y,D_0)]^{vir}_T=\sum_{G^l_0}\frac{1}{|Aut(G^l_0)|}(\tau_{G^l_0})_*\Biggr(\frac{[\overline{\mathcal{M}}_{G^l_0}]^{vir}}{e_T(N_{G^l_0}^{vir})}\Biggr).
\end{equation}

We remark that when $(Y,D_0)=(\mathbb{P}^1,\mathbf{0})$, (\ref{lc-0}) is a special case of Formula (6) in \cite{FP2}.

\section{Appendix C}
We give some combinatorial identities which will be used in the proofs of Lemma \ref{typeI-nf3} and Lemma \ref{keytrick}.

For $m$ positive integer, $k$ non-negative integer, we define
\[C_k(m)=\sum_{D=1}^{m}\frac{(-1)^{m-D}D^{m+k-1}}{(m-D)!D!}.\]
\begin{lemma}\label{comb-idt}
The generating function
\[f_k(x)=\sum_{m=1}^{\infty}C_k(m)x^m\]
equals to $P^k(x)$,
where $P^k$ is the $k$th product of the operator
$$\frac{x}{1-x}\frac{d}{dx}$$
which acts on the function $x$.
\end{lemma}
\begin{proof}
\[
\begin{aligned}
f_k(x)&=\sum_{m=1}^{\infty}\sum_{D=1}^m\frac{(-1)^{m-D}D^{m+k-1}}{(m-D)!D!}x^m\\
&=\sum_{D=1}^{\infty}\frac{D^{D+k-1}}{D!}x^D\sum_{m\geq D}\frac{(-1)^{m-D}}{(m-D)!}(Dx)^{m-D}=\sum_{D=1}^{\infty}\frac{D^{D+k-1}}{D!}(xe^{-x})^D.
\end{aligned}
\]
Taking the derivative on both sides, it is easy to see that
\[\left(\frac{x}{1-x}\frac{d}{dx}\right)f_k(x)=f_{k+1}(x).\]
So we have
\[f_k(x)=\left(\frac{x}{1-x}\frac{d}{dx}\right)^kf_0(x).\]
Recall that the famous Lambert W function is
\[W(x)=\sum_{D=1}^{\infty}\frac{(-D)^{D-1}}{D!}x^D.\]
It is an inverse function of $xe^x$. So we have
\[x=\sum_{D=1}^{\infty}\frac{(-D)^{D-1}}{D!}(xe^x)^D.\]
It is easy to deduce from above that $f_0(x)=x$. So we conclude that
\[f_k(x)=P^k(x).\]
\end{proof}
\begin{example}
\begin{eqnarray*}
f_1(x) & = & \frac{x}{1-x},\\
f_2(x) & = & \frac{x}{(1-x)^3},\\
f_3(x) & = & \frac{x+2x^2}{(1-x)^5},\\
f_4(x) & = & \frac{x+8x^2+6x^3}{(1-x)^7}.
\end{eqnarray*}
So we have the combinatorial identities
\begin{eqnarray*}
C_1(m) & = & 1,\\
C_2(m) & = & {m+1\choose 2},\\
C_3(m) & = & {m+3 \choose 4}+2{m+2\choose 4},\\
C_4(m) & = & {m+5\choose 6}+8{m+4\choose 6}+6{m+3\choose 6 }.
\end{eqnarray*}
\end{example}

\section{Appendix D}
In this appendix, we will give a computation of $N_{2,1}$ using our method. Those Gromov-Witten invariants of $\mathbb{P}^4$ we need will be computed by Gathmann's program GROWI \cite{Ga3}.

The system of equations we need to determine $N_{2,1}$ can be derived from computing the following four one-point invariants
\[\Big\langle\Big\{H^m\prod_{k=0}^{4-m}\bigr(k\tau+\iota_*(Id)\bigr)\Big\}\Big\rangle_{2,1}^{\mathbb{P}^4},~~~m=0,1,2,3\]
and one two-point invariant
\[\Big\langle\Big\{\prod_{k_1=0}^3\bigr(k_1\tau+\iota_*(Id)\bigr)\Big\}\Big\{\prod_{k_2=0}^1\bigr(k_1\tau+\iota_*(Id)\bigr)\Big\}\Big\rangle_{2,1}^{\mathbb{P}^4}\]
by degeneration formula.

Those principal terms in the degeneration formula have already be computed in the proof of Theorem \ref{thm-2}. And those non-principal parts only involve of relative Gromov-Witten invariants of the pair $(\mathbb{P}^4,Q)$ whose genus $g\leq 1$ together with some fiber invariants. Both of them can be determined by Maulik-Pandharipande's algorithm. For simplicity, we will write down the total contribution of those non-principal parts and omit the details of computation here. The results can be listed as follows.
\begin{eqnarray*}
\Big\langle\Big\{H^3\prod_{k=0}^1\bigr(k\tau+\iota_*(Id)\bigr)\Big\}\Big\rangle_{2,1}^{\mathbb{P}^4}&=&\Big\langle\Big|\eta_3\Big\rangle_{2,1}^{\mathbb{P}^4,Q},\\
\Big\langle\Big\{H^2\prod_{k=0}^2\bigr(k\tau+\iota_*(Id)\bigr)\Big\}\Big\rangle_{2,1}^{\mathbb{P}^4}&=&\Big\langle\Big|\eta_2\Big\rangle_{2,1}^{\mathbb{P}^4,Q}-\frac{25}{24},\\
\Big\langle\Big\{H^1\prod_{k=0}^3\bigr(k\tau+\iota_*(Id)\bigr)\Big\}\Big\rangle_{2,1}^{\mathbb{P}^4}&=&\Big\langle\Big|\eta_1\Big\rangle_{2,1}^{\mathbb{P}^4,Q}+N_{2,1}+\frac{725}{24},\\
\Big\langle\Big\{H^0\prod_{k=0}^4\bigr(k\tau+\iota_*(Id)\bigr)\Big\}\Big\rangle_{2,1}^{\mathbb{P}^4}&=&\Big\langle\Big|\eta_0\Big\rangle_{2,1}^{\mathbb{P}^4,Q}+45N_{2,1}+\frac{45775}{72},\\
\end{eqnarray*}
where those cohomology weighted partitions
\[\eta_i=\bigr\{(5-i,H^i),(1,Id)^i\bigr\},~~~i=0,1,2,3.\]

As for the two-point invariant
\[\Big\langle\Big\{\prod_{k_1=0}^3\bigr(k_1\tau+\iota_*(Id)\bigr)\Big\}\Big\{\prod_{k_2=0}^1\bigr(k_1\tau+\iota_*(Id)\bigr)\Big\}\Big\rangle_{2,1}^{\mathbb{P}^4},\]
it equals to
\[\sum_{i=0}^3(5-i)5^i\Big\langle\Big|\eta_i\Big\rangle_{2,1}^{\mathbb{P}^4,Q}+293N_{2,1}+\frac{157975}{36}.\]

Using the program GROWI \cite{Ga3} of Gathmann, we can compute that
\begin{eqnarray*}
\Big\langle\Big\{H^3\prod_{k=0}^1\bigr(k\tau+\iota_*(Id)\bigr)\Big\}\Big\rangle_{2,1}^{\mathbb{P}^4}&=&\frac{185}{1152},\\
\Big\langle\Big\{H^2\prod_{k=0}^2\bigr(k\tau+\iota_*(Id)\bigr)\Big\}\Big\rangle_{2,1}^{\mathbb{P}^4}&=&\frac{2885}{1152},\\
\Big\langle\Big\{H^1\prod_{k=0}^3\bigr(k\tau+\iota_*(Id)\bigr)\Big\}\Big\rangle_{2,1}^{\mathbb{P}^4}&=&\frac{28345}{1152},\\
\Big\langle\Big\{H^0\prod_{k=0}^4\bigr(k\tau+\iota_*(Id)\bigr)\Big\}\Big\rangle_{2,1}^{\mathbb{P}^4}&=&\frac{20125}{128},\\
\Big\langle\Big\{\prod_{k_1=0}^3\bigr(k_1\tau+\iota_*(Id)\bigr)\Big\}\Big\{\prod_{k_2=0}^1\bigr(k_1\tau+\iota_*(Id)\bigr)\Big\}\Big\rangle_{2,1}^{\mathbb{P}^4}&=&\frac{1592375}{576}.
\end{eqnarray*}

Now we take them into the above five equations, it is easy to compute that
\[N_{2,1}=\frac{2875}{240}.\]
The result agrees with the prediction of physicists Bershadsky et al. in \cite{BCOV1}\cite{BCOV2} which is recently proven by Guo et al. in \cite{GJR}.

\bibliographystyle{plain}

\end{document}